\documentclass[pdflatex,sn-mathphys-num]{sn-jnl}

\usepackage{mathtools}
\usepackage{amssymb}
\usepackage{enumitem}
\usepackage{graphicx}
\usepackage{amsthm}
\usepackage[percent]{overpic}

\theoremstyle{thmstyleone}
\newtheorem{theorem}{Theorem}[section]
\newtheorem{lemma}[theorem]{Lemma}
\newtheorem{proposition}[theorem]{Proposition}
\newtheorem{corollary}[theorem]{Corollary}

\theoremstyle{thmstyletwo}
\newtheorem{remark}{Remark}

\theoremstyle{thmstylethree}
\newtheorem{definition}[theorem]{Definition}

\newcommand{\Lcal}{\mathcal L}

\newcommand{\dom}{\operatorname{dom}}
\newcommand{\ran}{\operatorname{ran}}
\newcommand{\Ree}{\operatorname{Re}}
\newcommand{\Cplus}{\mathbb C_+}

\newcommand{\setdef}[2]{\left\{\,#1\;\middle|\;#2\,\right\}}
\newcommand{\spvek}[2]{\left(\begin{smallmatrix}#1\\#2\end{smallmatrix}\right)}
\newcommand{\sbvek}[2]{\left[\begin{smallmatrix}#1\\#2\end{smallmatrix}\right]}
\newcommand{\pvek}[2]{\left(\begin{matrix}#1\\#2\end{matrix}\right)}
\newcommand{\bvek}[2]{\left[\begin{matrix}#1\\#2\end{matrix}\right]}

\title{A Newton--Kleinman Iteration for Infinite-Dimensional Linear-Quadratic Optimal Control}

\author*[1]{\fnm{Timo} \sur{Reis}}\email{timo.reis@tu-ilmenau.de}

\affil*[1]{\orgdiv{Institute of Mathematics}, \orgname{Technische Universit\"at Ilmenau}, \orgaddress{\city{Ilmenau}, \postcode{98693}, \country{Germany}}}

\begin{document}

\abstract{
We develop a Newton--Kleinman iteration for linear-quadratic optimal control
of infinite-dimensional systems with possibly unbounded control and
observation operators. In this setting, the classical algebraic Riccati
equation need not be meaningful and therefore cannot serve as the basis of
the iteration. We instead linearize a Lur'e identity formulated directly on
the system-node domain. Starting from any stabilizing feedback, the resulting
iteration is well defined, preserves stability, and yields a monotonically
decreasing sequence of cost operators. A resolvent-based normalization makes
the construction applicable to unbounded control. Under finite-cost
detectability and for a finite-dimensional input space, the iterates converge
strongly to the optimal value operator, independently of the fixed
normalization parameter. In finite dimensions, the natural zero-state
normalization recovers exactly the classical Newton--Kleinman iteration.

The formulation leads to a low-rank numerical realization based on shifted
system solves and inversions on the input space. We illustrate the method on
a linearized cylinder wake with boundary traction control and unbounded
observation, using one Newton--Kleinman update per adaptively refined spatial
mesh, and on a damped linear elasticity problem with boundary traction
control and boundary displacement observation.
}

\keywords{
Newton--Kleinman iteration, system node, unbounded control and observation,
linear-quadratic optimal control, low-rank ADI, adaptive discretization
}
\pacs[MSC Classification]{49N10, 47N70, 93B52, 93C20}

\maketitle


\section{Introduction}
\label{sec:intro}

The Newton--Kleinman iteration is a classical approach to the numerical
solution of algebraic Riccati equations \cite{Kleinman1968}.  Its particular
appeal for large-scale problems is that the nonlinear Riccati equation is
replaced by a sequence of Lyapunov equations, which allows the use of inexact
and low-rank solvers \cite{FeitzingerHyllaSachs2009,BennerEtAl2016}.
For PDE control problems, Newton--Kleinman and Newton--ADI (alternating-direction implicit) methods have been
used extensively after spatial discretization, where the resulting large-scale
matrix Riccati equations can be treated by low-rank techniques; see, for
instance, \cite{MorrisNavasca2010,BennerSaak2010,BennerEtAl2016}. Infinite-dimensional extensions have also been developed directly for
operator Riccati equations
\cite{BurnsSachsZietsman2008,Iftime2017,CurtainZwartIftime2017}.
Their convergence theory retains bounded control, and mostly also bounded
observation, operators.  Roughly speaking, bounded control means that the
input acts directly in the state space, whereas bounded observation means
that the output can be evaluated for every state.  This excludes many
boundary control and observation problems.  Unbounded observation for
regular linear systems is treated in
\cite{MassoudiReis2016,MassoudiThesis}, but the complete convergence theory
there still requires bounded control.

The present paper allows both control and observation to be unbounded.
In this setting the algebraic Riccati equation
\eqref{eq:intro-fd-are} can no longer be taken as the starting point:
the operator products occurring in it need not be well defined, and even
when an algebraic Riccati equation can be formulated, the value operator need
not be characterized by it; see \cite{Opmeer2014}.  The
problem is therefore not merely to prove Newton--Kleinman convergence under
weaker assumptions.  One first has to replace the algebraic Riccati equation
by an identity for the value operator that remains meaningful for such
systems, and then formulate the Newton step on the basis of this identity.

The finite-dimensional case provides the benchmark for the construction
below. 
For $A\in\mathbb C^{n\times n}$, $B\in\mathbb C^{n\times m}$,
$C\in\mathbb C^{p\times n}$, and $D\in\mathbb C^{p\times m}$, consider
\begin{equation}
\dot x=Ax+Bu,
\qquad
y=Cx+Du,
\label{eq:intro-fd-system}
\end{equation}
together with the quadratic cost
\begin{equation}
J(u,y)
=
\int_0^\infty\|y(t)\|^2+\|u(t)\|^2\,dt.
\label{eq:intro-fd-cost}
\end{equation}
Under the standard stabilizability and detectability assumptions, the optimal
cost is quadratic:
\[
V(x_0):=\inf J(u,y)=x_0^*P_\star x_0,
\]
where the infimum is taken over all admissible trajectories starting from
$x_0$, and $P_\star\in\mathbb C^{n\times n}$ is Hermitian and positive
semidefinite.  This matrix is the stabilizing solution of the algebraic
Riccati equation
\begin{equation}
A^*P_\star +P_\star A+C^*C
-
(P_\star B+C^*D)(I+D^*D)^{-1}(B^*P_\star +D^*C)
=0.
\label{eq:intro-fd-are}
\end{equation}
Here stabilizing means that the feedback
\begin{equation}\label{eq:intro-fd-feedback}
K
=
-(I+D^*D)^{-1}(B^*P_\star +D^*C)
\end{equation}
makes $A+BK$ stable.  Thus the stabilizing Riccati solution determines both
the optimal cost and the optimal feedback.

Newton--Kleinman computes the stabilizing solution of
\eqref{eq:intro-fd-are} by Newton's method, with each feedback evaluation
given by a Lyapunov equation. A suitable initialization is provided by any
stabilizing feedback $K_0$, rather than by an initial approximation of $P$.
For a current stabilizing feedback $K_k$, set
\[
A_k=A+BK_k,
\qquad
Z_k=\sbvek{C+DK_k}{K_k}.
\]
The corresponding feedback cost operator is determined by
\begin{equation}
A_k^*P_{k+1}+P_{k+1}A_k+Z_k^*Z_k=0,
\qquad
K_{k+1}=-(I+D^*D)^{-1}(B^*P_{k+1}+D^*C).
\label{eq:intro-fd-lyap}
\end{equation}
For $k\geq1$, the same feedback evaluation can be written incrementally.
Setting
\[
S_k=P_{k+1}-P_k,
\qquad
\Theta_k=(I+D^*D)^{1/2}(K_k-K_{k-1}),
\]
the Newton correction satisfies
\begin{equation}
A_k^*S_k+S_kA_k-\Theta_k^*\Theta_k=0.
\label{eq:intro-fd-increment}
\end{equation}
The update is then completed by
\[
P_{k+1}=P_k+S_k,
\qquad
K_{k+1}=-(I+D^*D)^{-1}(B^*P_{k+1}+D^*C).
\]
Thus, the classical Newton--Kleinman iteration admits two equivalent
Lyapunov realizations. One may compute the current feedback cost operator
$P_{k+1}$ directly from \eqref{eq:intro-fd-lyap}, or, after the first
step, compute the Newton correction from \eqref{eq:intro-fd-increment}
and recover $P_{k+1}$ from $P_k$. The incremental realization is
particularly attractive when the feedback correction has substantially
lower rank than the full performance output
\cite{FeitzingerHyllaSachs2009,BennerEtAl2016}.

Our aim is to retain the structure of the classical Newton--Kleinman
iteration directly at the infinite-dimensional level, before spatial
discretization. The obstruction described above prevents a direct
matrix-to-operator extension. We instead use the Lur'e formulation of
\cite{Opmeer_Staffans,Opmeer2014}, which is defined on the system-node
domain and remains meaningful for simultaneously unbounded control and
observation. Feedback evaluation continues to be an observability-Gramian
computation, while the Newton update is obtained by linearizing the
corresponding Lur'e identity.

The essential modification concerns the normalization of the linearized
Lur'e factor. In finite dimensions, the usual zero-state normalization
recovers the classical Newton--Kleinman update exactly. For genuinely
unbounded control, however, changing the input while keeping the state fixed
need not define an element of the system-node domain. We therefore introduce
a resolvent-based normalization with a fixed parameter $\alpha>0$.
Different choices of this parameter may produce different intermediate
feedbacks, but, under the assumptions of our convergence result, they lead
to the same limiting value operator. The precise construction and its
finite-dimensional interpretation are given in Section~\ref{sec:nk}.

Flow inversion shows that every resulting Newton update is infinite-time
well posed and preserves strong stability.  Moreover, the cost operators
satisfy the exact decrement identity
\[
P_k-P_{k+1}
=
\mathcal C_{\Theta,k,\infty}^*
\mathcal C_{\Theta,k,\infty}
\geq0.
\]
Hence every stabilizing initialization generates a monotonically decreasing
sequence, and there is a bounded nonnegative operator $P_\infty$ such that
$P_kx\to P_\infty x$ for all $x\in X$.  Under finite-cost detectability
and for a finite-dimensional input space, $P_kx\to P_\star x$ for all
$x\in X$, independently of the fixed normalization parameter. If the zero-state normalization is available, it may be used instead; in
finite dimensions this recovers exactly the classical Newton--Kleinman
iteration.

The formulation also leads directly to a resolvent-based numerical
realization. Spatial discretization enters through the shifted state and
adjoint problems required for the Gramian and feedback computations, which
permits the use of low-rank operator-ADI techniques
\cite{OpmeerReisWollner2013}. Throughout, we use the system-node framework
of \cite{Staffans2005}, which accommodates simultaneously unbounded control
and observation operators and has been successfully applied to a range of
PDE systems with distributed and boundary interaction; see, for instance,
\cite{PhilippReisSchaller2025,ReisSchallerOseen,SkrepekReis2026,
ReisSchaller2025}.

Two PDE examples illustrate complementary aspects of the approach. For a
linearized cylinder wake, both the boundary traction input and the
observation are unbounded at the continuous level. We combine the
Newton--Kleinman iteration with adaptive Taylor--Hood discretization,
performing exactly one Newton--Kleinman update on each spatial level. A
second example considers damped linear elasticity on an L-shaped domain
with boundary traction control and boundary displacement observation and
illustrates the low-rank realization of the iteration in a different PDE
setting.

The paper is organized as follows. Section~\ref{sec:prelim} recalls the
required system-node framework. Section~\ref{sec:lq} introduces the value-operator and Lur'e formulation. Section~\ref{sec:nk} develops the
infinite-dimensional Newton--Kleinman iteration,
Section~\ref{sec:nk-wellposed} establishes feasibility and stability of the
Newton step, and Section~\ref{sec:nk-convergence} proves convergence.
Sections~\ref{sec:cylinder-wake} and~\ref{sec:numerics} present the two
numerical examples concerning a linearized cylinder wake and
boundary-controlled elasticity, respectively.

\section{Notation and preliminaries}
\label{sec:prelim}

\subsection{Notation}

All Hilbert spaces are complex and their inner products are linear in the
second argument.  We write $\mathcal L(X,Y)$ for the bounded operators from
$X$ to $Y$ and $\mathcal L(X):=\mathcal L(X,X)$.  The adjoint of an
operator $T$ is denoted by $T^*$, and $\rho(A)$ denotes the resolvent set of
a closed operator $A$.  We write
$\mathbb C_\pm
=
\{\lambda\in\mathbb C:\pm\operatorname{Re}\lambda>0\}$.
For self-adjoint operators $P,Q\in\mathcal L(X)$, the notation
$P\geq Q$ means that $P-Q$ is nonnegative. Norm subscripts are omitted whenever the underlying space is clear.

Whenever the domain of a closed operator is regarded as a normed space, it
is equipped with its graph norm.  If $A$ generates a $C_0$-semigroup on
$X$, then $X_{-1}$ and $A_{-1}$ denote the associated extrapolation space
and extrapolated generator, respectively.

We use $H^\infty(\mathbb C_+;\mathcal L(U))$ in the standard Hardy-space
sense, with
\[
\|Q\|_{H^\infty}
=
\sup_{\lambda\in\mathbb C_+}\|Q(\lambda)\|.
\]
By the $L^2$ multiplier theorem
\cite[Thm.~10.3.5]{Staffans2005}, every
$Q\in H^\infty(\mathbb C_+;\mathcal L(U))$ is the symbol of a unique bounded
causal time-invariant operator
\[
\mathfrak Q:L^2(0,\infty;U)\to L^2(0,\infty;U),
\qquad
\|\mathfrak Q\|=\|Q\|_{H^\infty}.
\]
For $T>0$, its truncation to $[0,T]$, obtained by zero extension followed
by restriction to $[0,T]$, is denoted by $\mathfrak Q_T$.  In particular,
$\|\mathfrak Q_T\|
\leq
\|\mathfrak Q\|
=
\|Q\|_{H^\infty}$.
\subsection{System nodes and well-posed systems}
\label{subsec:prelim-system-nodes}

We briefly recall the system-node framework used throughout the paper; see
\cite[Def.~4.7.2 and Secs.~4.6--4.7]{Staffans2005}.  Let $X$, $U$, and $Y$
be complex Hilbert spaces, representing the state, input, and output spaces,
respectively.

\begin{definition}[System node]
\label{def:prelim-system-node}
A system node on $(X,U,Y)$ is an operator
\[
\mathcal S
=
\begin{bmatrix}
A\&B\\
C\&D
\end{bmatrix}
:
\dom(A\&B)\subset X\times U
\rightarrow
X\times Y
\]
with the following properties:
\begin{enumerate}[label=(\alph*)]
\item $A\&B$ is closed;
\item $C\&D$ is bounded on $\dom(A\&B)$ equipped with the graph norm of
      $A\&B$;
\item for all $u\in U$, there exists $x\in X$ such that
      $\spvek{x}{u}\in\dom(A\&B)$;
\item the main operator
\[
Ax=A\&B\spvek{x}{0},
\qquad
\dom(A)
=
\setdef{x\in X}{\spvek{x}{0}\in\dom(A\&B)},
\]
generates a strongly continuous semigroup
$\mathcal T=(\mathcal T(t))_{t\geq0}$ on $X$.
\end{enumerate}
\end{definition}

The notation $C\&D$ denotes a joint output operator and does not imply that
separate operators $C$ and $D$ are available.  In finite dimensions,
Definition~\ref{def:prelim-system-node} reduces to the familiar state-space
representation.

The input--state--output system associated with $\mathcal S$ is formally
written as
\begin{equation}
\spvek{\dot x(t)}{y(t)}
=
\mathcal S\spvek{x(t)}{u(t)}.
\label{eq:prelim-system-node-equation}
\end{equation}
A \emph{classical trajectory} of
\eqref{eq:prelim-system-node-equation} on $[0,T]$ is a triple
\[
(x,u,y)
\in
C^1([0,T];X)
\times C([0,T];U)
\times C([0,T];Y)
\]
such that
$\spvek{x(t)}{u(t)}\in\dom(A\&B)$
and \eqref{eq:prelim-system-node-equation} holds for all $t\in[0,T]$. A \emph{generalized trajectory} on $[0,T]$ is a triple
\[
(x,u,y)
\in
C([0,T];X)
\times L^2(0,T;U)
\times L^2(0,T;Y)
\]
which is the limit, in this product space, of a sequence of classical
trajectories; see \cite[Sec.~4.7]{Staffans2005}.

For later use, we recall the extrapolation-space representation of the state
equation.  Fix $\beta\in\rho(A)$.  The extrapolation space $X_{-1}$ is the
completion of $X$ with respect to the norm
\[
\|x\|_{-1}
=
\|(\beta-A)^{-1}x\|_X,
\]
and $A$ extends continuously to an operator
$A_{-1}:X\to X_{-1}$; see, for instance,
\cite[Sec.~2.10]{TucsnakWeiss2009}.  Different choices of $\beta$ give
equivalent norms. There exists a unique
operator
$B\in\mathcal L(U,X_{-1})$
such that
\begin{equation}
\dom(A\&B)
=
\setdef{\spvek{x}{u}\in X\times U}
{A_{-1}x+Bu\in X},
\;\;
A\&B\spvek{x}{u}
=
A_{-1}x+Bu;
\label{eq:prelim-node-domain}
\end{equation}
see \cite[Def.~4.7.2 and Lem.~4.7.3]{Staffans2005}.  In particular, the
state equation may be interpreted in the extrapolation space as
$\dot x(t)=A_{-1}x(t)+Bu(t)$.
This representation is especially useful for boundary control, where the
control action need not take values in $X$.

By \cite[Lem.~4.7.3]{Staffans2005}, for all $\lambda\in\rho(A)$ the map
\[
u
\mapsto
\spvek{(\lambda-A_{-1})^{-1}Bu}{u}
\]
is bounded from $U$ into $\dom(A\&B)$ equipped with the graph norm.
Consequently, the transfer function of $\mathcal S$ is defined by
\begin{equation}
\mathbf G(\lambda)u
=
C\&D
\pvek{(\lambda-A_{-1})^{-1}Bu}{u},
\qquad
\lambda\in\rho(A).
\label{eq:prelim-transfer}
\end{equation}
In finite dimensions this is simply
$\mathbf G(\lambda)=C(\lambda-A)^{-1}B+D$, whereas
\eqref{eq:prelim-transfer} remains meaningful when $C$ and $D$ do not exist
as separate operators.

A system node is \emph{well posed} if, for all $T>0$, each initial state
$x_0\in X$ and input $u\in L^2(0,T;U)$ determine a unique generalized
trajectory and the resulting input--state--output map is bounded.  We denote
the corresponding finite-time maps by
\begin{equation}
\pvek{x(T)}{y}
=
\begin{bmatrix}
\mathcal T(T)&\mathcal B_T\\
\mathcal C_T&\mathcal D_T
\end{bmatrix}
\pvek{x_0}{u}.
\label{eq:prelim-finite-maps}
\end{equation}
Thus
\[
\mathcal B_T:L^2(0,T;U)\to X,
\quad
\mathcal C_T:X\to L^2(0,T;Y),
\quad
\mathcal D_T:L^2(0,T;U)\to L^2(0,T;Y).
\]
The system node is \emph{infinite-time well posed} if the four families in
\eqref{eq:prelim-finite-maps} are bounded uniformly for $T>0$.

We shall also use the corresponding admissibility terminology for individual
control and observation operators; see
\cite[Secs.~4.2--4.4]{TucsnakWeiss2009}.  Thus
$B_0\in\mathcal L(U,X_{-1})$ is called \emph{admissible for $\mathcal T$}
if its associated finite-time input maps
$\mathcal B_{0,T}:L^2(0,T;U)\to X$
are bounded for all $T>0$, and \emph{infinite-time admissible for
$\mathcal T$} if their norms are bounded uniformly in $T$.  Likewise,
$C_0\in\mathcal L(\dom(A),Y)$ is called \emph{admissible for $\mathcal T$}
if its associated finite-time observation maps
$\mathcal C_{0,T}:X\to L^2(0,T;Y)$
are bounded for all $T>0$, and \emph{infinite-time admissible for
$\mathcal T$} if their norms are bounded uniformly in $T$.  In the latter
case the maps $\mathcal C_{0,T}$ determine a bounded infinite-time
observation map
$\mathcal C_{0,\infty}:X\to L^2(0,\infty;Y)$,
and $\mathcal C_{0,\infty}^*\mathcal C_{0,\infty}$ is its infinite-time
observability Gramian.

For a well-posed system node, the control operator $B$ and the zero-input
observation operator
\[
Cx=C\&D\spvek{x}{0},
\qquad x\in\dom(A),
\]
are admissible for $\mathcal T$, with finite-time maps $\mathcal B_T$ and
$\mathcal C_T$ from \eqref{eq:prelim-finite-maps}.  If the system node is
infinite-time well posed, they are infinite-time admissible.

\subsection{Feedback and flow inversion}
\label{subsec:prelim-feedback}

We recall two system-theoretic constructions that will be used later:
admissible feedback and flow inversion. Admissible feedback describes the
closed-loop interconnection obtained by feeding an auxiliary output back
through the input. Flow inversion reverses such an interconnection by
exchanging the corresponding input and output variables. Here and below,
we refer to these variables as ports.

Let
$\mathcal S=\sbvek{A\&B}{C\&D}$
be a system node on $(X,U,Y)$ and let
$F\&G\in\mathcal L(\dom(A\&B),U)$. Consider the feedback relation
\[
u=F\&G\spvek{x}{u}+v.
\]
If this relation defines an admissible feedback transformation in the sense
of \cite[Sec.~7.3]{Staffans2005}, it yields a system node with external
input $v$; we retain the physical input $u$ as an additional output.
We call the feedback \emph{stabilizing} if this system node is infinite-time
well posed and its zero-input semigroup is strongly stable. The
Newton--Kleinman iteration is initialized by such a stabilizing feedback.
In finite dimensions this is the usual requirement of a stabilizing initial
feedback matrix.

\begin{definition}[Finite-cost detectability]
\label{def:prelim-finite-cost-detectability}
The system node is called \emph{finite-cost detectable} if every
infinite-time generalized trajectory satisfying
$u\in L^2(0,\infty;U)$ and
$y\in L^2(0,\infty;Y)$
also satisfies $x(t)\to0$ as $t\to\infty$.
\end{definition}
Strong detectability in the standard well-posed-system sense, formulated in
terms of a strongly stabilizing output injection, implies finite-cost
detectability; see
\cite[Defs.~7.3.2, 8.1.1(ii), 8.2.4(iii),(v) and Thm.~8.1.6(i)]{Staffans2005}. In finite dimensions, finite-cost detectability is equivalent to the
classical detectability of the pair $(C,A)$.

Flow inversion is the mechanism used below to solve the newly constructed
port relation for the old input while leaving the state equation unchanged.
It also provides a criterion ensuring that the resulting interconnection is
again a system node.  Consider a system node
$\mathcal N
=
\sbvek{A\&B}{N\&M}$
whose input and output spaces are both $U$.  Flow inversion interchanges these two port variables while leaving the state
unchanged.  We use the full-flow case of Staffans'
partial-flow-inversion criterion \cite[Thm.~6.6.11]{Staffans2005}. For $\lambda\in\mathbb C$ define
\[
\mathcal R_\lambda\spvek{x}{u}
=
\pvek{
\lambda x-A\&B\spvek{x}{u}
}{
-N\&M\spvek{x}{u}
}.
\]
The criterion states that the port is flow invertible if and only if there
exists $\lambda\in\mathbb C$ such that $\mathcal R_\lambda$ is bijective and
the upper-left block $M_{11}$ of $\mathcal R_\lambda^{-1}$ is injective with
dense range and $-M_{11}^{-1}$ generates a strongly continuous semigroup. In this case,
interchanging the port variables defines a system node with input
$v=N\&M\spvek{x}{u}$
and output $u$, while the state equation
$\dot x=A\&B\spvek{x}{u}$
remains unchanged.  This system node is called the flow inverse of
$\mathcal N$.  Classical and generalized trajectories of the original node
and its flow inverse are obtained from one another by interchanging the
corresponding input and output signals; see
\cite[Thm.~6.6.15]{Staffans2005}.

\section{The value operator and the Lur'e equation}
\label{sec:lq}

We use the value-operator and Riccati formulation of infinite-horizon
linear-quadratic optimal control developed in
\cite{Staffans1997,Staffans1998,Opmeer_Staffans}; see also
\cite{Opmeer2014}. For the present purpose, its essential advantage is that
the optimal value and the corresponding Riccati equation remain meaningful
even when the classical algebraic Riccati equation is not well defined.

Let
$\mathcal S
=
\sbvek{A\&B}{C\&D}$
be a system node on $(X,U,Y)$.  We consider the quadratic cost
\begin{equation}
J(u,y)
=
\|u\|_{L^2(0,\infty;U)}^2
+
\|y\|_{L^2(0,\infty;Y)}^2.
\label{eq:lq-cost}
\end{equation}
For $x_0\in X$, let $V(x_0)$ denote the optimal cost in the sense of
\cite[Sec.~3]{Opmeer_Staffans}.  If the system node is well posed, this is
precisely the infimum of \eqref{eq:lq-cost} over all infinite-time
generalized trajectories with initial state $x_0\in X$ and
$u\in L^2(0,\infty;U)$,
$y\in L^2(0,\infty;Y)$.

The infinite-time well-posedness of the system node associated with a stabilizing feedback implies the
state finite future cost condition of
\cite[Def.~5.7]{Opmeer_Staffans}.  Indeed, setting its external input to
zero produces, for all $x_0\in X$, a trajectory whose physical input and
output belong to the corresponding $L^2$ spaces.  Hence every initial state
has finite future cost.

Let
$\mathfrak C_{\rm o}:X\to
L^2(0,\infty;U)\times L^2(0,\infty;Y)$
denote the map that assigns to each initial state its optimal future
input--output pair; see \cite[Sec.~3]{Opmeer_Staffans}.  By
\cite[Lem.~3.6]{Opmeer_Staffans}, $\mathfrak C_{\rm o}$ is closed, and by
\cite[Thm.~3.7]{Opmeer_Staffans},
$V(x_0)=\|\mathfrak C_{\rm o}x_0\|^2$ for all $x_0\in X$.
The state finite future cost condition gives
$\dom(\mathfrak C_{\rm o})=X$.  Hence the closed graph theorem yields
$\mathfrak C_{\rm o}\in\mathcal L
(X,L^2(0,\infty;U)\times L^2(0,\infty;Y))$.  Consequently,
\[
V(x_0)
=
\langle P_\star x_0,x_0\rangle_X,
\qquad
P_\star=\mathfrak C_{\rm o}^*\mathfrak C_{\rm o}\in\mathcal L(X),
\qquad
P_\star=P_\star^*\geq0.
\]
By \cite[Thm.~5.9]{Opmeer_Staffans}, the optimal future cost is the minimal
nonnegative solution of the continuous-time control Riccati equation of
\cite[Def.~5.1]{Opmeer_Staffans}.  For the present quadratic cost, the
corresponding Lur'e form is defined on the full system-node domain and yields
an operator
$L_\star\&M_\star\in\mathcal L(\dom(A\&B),U)$
such that
\begin{multline}
2\Ree
\left\langle
P_\star x,
A\&B\spvek{x}{u}
\right\rangle_X
+
\left\|
C\&D\spvek{x}{u}
\right\|_Y^2
+
\|u\|_U^2
=
\left\|
L_\star\&M_\star\spvek{x}{u}
\right\|_U^2\\
\text{for all $\spvek{x}{u}\in\dom(A\&B)$}.
\label{eq:lq-lure}
\end{multline}
Equation~\eqref{eq:lq-lure} is the Lur'e form of the Riccati equation. Its
importance here is that every term is well defined on the system-node domain:
no products such as $P_\star B$, $B^*P_\star$, or $D^*C$ are required.
The same Lur'e identity may equivalently be obtained from the bounded-storage
dissipation inequality; see \cite[Thm.~4.1]{Reis2026_Dissip}.

For comparison, consider the finite-dimensional optimal control problem
\eqref{eq:intro-fd-system}--\eqref{eq:intro-fd-cost}. Then $P_\star$ is the
stabilizing solution of \eqref{eq:intro-fd-are}, with feedback $K$ given by
\eqref{eq:intro-fd-feedback}, and the algebraic Riccati equation is
equivalently written as
\begin{multline*}
2\Ree
\langle P_\star x,Ax+Bu\rangle
+
\|Cx+Du\|^2
+
\|u\|^2
=
\|(I+D^*D)^{1/2}(u-Kx)\|^2\\
\text{for all $x\in\mathbb C^n$ and $u\in\mathbb C^m$}.
\end{multline*}
Thus \eqref{eq:lq-lure} reduces exactly to the familiar matrix Riccati
equation, with
\[
L_\star\&M_\star\spvek{x}{u}
=
(I+D^*D)^{1/2}(u-Kx).
\]
In particular, setting the factor equal to zero gives the optimal feedback
relation $u=Kx$.

Thus the Lur'e identity remains meaningful for unbounded control and
observation and reduces exactly to the classical Riccati equation in finite
dimensions.  We use it as the basis of the Newton step in the next section.

\section{The infinite-dimensional Newton--Kleinman iteration}
\label{sec:nk}

We now formulate the Newton--Kleinman iteration at the system-node level.
As in the matrix case, one step consists of evaluating the cost associated
with the current feedback and then performing a Newton update. The feedback
cost is represented by an infinite-time observability Gramian, while the
Newton update is obtained by linearizing the Lur'e identity from
Section~\ref{sec:lq} in the coordinates of the system node associated with
the current feedback. The analysis in Section~\ref{sec:nk-wellposed} shows that the resulting
port is flow invertible and defines the next infinite-time well-posed and
strongly stable feedback.

Let $\Sigma_k$ denote the system node associated with the $k$-th
stabilizing feedback, written as
\[
\Sigma_k
=
\begin{bmatrix}
A_k\&B_k\\
Z_k\&W_k
\end{bmatrix},
\]
with auxiliary input $v_k\in U$. Its performance output is
$z_k=(y,u)\in Y\times U$, so that $\|z_k\|^2=\|y\|^2+\|u\|^2$.
Setting $v_k=0$ closes the current feedback; the physical control input
$u$ is generally nonzero.
We denote its zero-input semigroup by $\mathcal T_k$ and its infinite-time
observation map by
$\mathcal C_{k,\infty}:X\to L^2(0,\infty;Y\times U)$.
The iteration is initialized by the system node $\Sigma_0$ associated with
an arbitrary stabilizing feedback; no initial approximation of the optimal
value operator is required. Thus the initialization includes the auxiliary-input coordinate, not only
the corresponding feedback relation obtained by setting the auxiliary input to zero.

\subsection{Feedback evaluation and Newton update}
\label{subsec:nk-feedback-update}

For the fixed $k$-th feedback, define
\begin{equation}
P_{k+1}
=
\mathcal C_{k,\infty}^*\mathcal C_{k,\infty}.
\label{eq:nk-feedback-value}
\end{equation}
Thus $\langle P_{k+1}x_0,x_0\rangle$ is precisely the infinite-horizon cost
of the zero-input trajectory of $\Sigma_k$ starting from $x_0$.
Since $Z_k$ is infinite-time admissible, $P_{k+1}$ satisfies
\begin{equation}
2\Ree\langle P_{k+1}x,A_kx\rangle+\|Z_kx\|^2=0,
\qquad x\in\dom(A_k).
\label{eq:nk-feedback-lyap}
\end{equation}
Strong stability of $\mathcal T_k$ implies uniqueness among bounded
self-adjoint solutions. Hence feedback evaluation may equivalently be viewed
as an operator Lyapunov solve.

For $P=P^*\in\mathcal L(X)$ define, on $\dom(A_k\&B_k)$,
\begin{equation}
\mathfrak l_{k,P}\!\left[\spvek{x}{v}\right]
:=
2\Ree\left\langle
Px,A_k\&B_k\spvek{x}{v}
\right\rangle
+
\left\|Z_k\&W_k\spvek{x}{v}\right\|^2,
\qquad
\spvek{x}{v}\in\dom(A_k\&B_k).
\label{eq:nk-lure-form}
\end{equation}
Along a classical trajectory this is
\[
\tfrac{\mathrm d}{\mathrm dt}\langle Px(t),x(t)\rangle+\|z_k(t)\|^2.
\]
In particular, \eqref{eq:nk-feedback-lyap} gives
\begin{equation}
\mathfrak l_{k,P_{k+1}}\!\left[\spvek{x}{0}\right]=0,
\qquad x\in\dom(A_k).
\label{eq:nk-lure-zero-slice}
\end{equation}

In the coordinates of $\Sigma_k$, the factor at the linearization point is
represented by the auxiliary-input coordinate itself: its value at
$\spvek{x}{v}\in\dom(A_k\&B_k)$ is $v$. Let $S=S^*\in\Lcal(X)$ be an increment of the cost operator and let
$\delta\Gamma\&\delta\Lambda\in
\Lcal(\dom(A_k\&B_k),U)$ be an increment of the factor. Since
\eqref{eq:nk-lure-form} is affine in $P$, linearizing the residual
\[
\mathfrak l_{k,P_k+S}\!\left[\pvek{x}{v}\right]
-
\left\|v+\delta\Gamma\&\delta\Lambda\pvek{x}{v}\right\|^2
\]
amounts to dropping only the quadratic term
$\|\delta\Gamma\&\delta\Lambda\pvek{x}{v}\|^2$.
We write the resulting updated factor as
\[
\Gamma_{k+1}\&\Lambda_{k+1}\pvek{x}{v}
=
v+\delta\Gamma\&\delta\Lambda\pvek{x}{v}.
\]
Setting $S=P_{k+1}-P_k$ then gives the Newton equation
\begin{equation}
\mathfrak l_{k,P_{k+1}}\!\left[\spvek{x}{v}\right]
=
2\Ree\left\langle
\Gamma_{k+1}\&\Lambda_{k+1}\spvek{x}{v},v
\right\rangle
-\|v\|^2,
\qquad
\spvek{x}{v}\in\dom(A_k\&B_k).
\label{eq:nk-linearized-lure}
\end{equation}
For $k=0$, no initial cost-operator approximation is needed, since only
the sum $P_0+(P_1-P_0)=P_1$ enters the linearization.

Setting $v=0$ in \eqref{eq:nk-linearized-lure} recovers
\eqref{eq:nk-lure-zero-slice}; hence the feedback evaluation provides exactly the
compatibility condition required by the Newton equation.
Section~\ref{sec:nk-wellposed} proves that a factor
$\Gamma_{k+1}\&\Lambda_{k+1}$ exists and fixes its remaining
skew-adjoint input-side freedom by a resolvent lift. Defining
$v_{k+1}=\Gamma_{k+1}\&\Lambda_{k+1}\spvek{x}{v_k}$ and using
$2\Ree\langle a,b\rangle-\|b\|^2=\|a\|^2-\|a-b\|^2$ yields
\begin{equation}
\mathfrak l_{k,P_{k+1}}\!\left[\spvek{x}{v_k}\right]
=
\|v_{k+1}\|^2-\|v_k-v_{k+1}\|^2.
\label{eq:nk-lure-residual-factorization}
\end{equation}
The difference $\vartheta_{k+1}:=v_k-v_{k+1}$ is the Newton correction
output. We define the next feedback relation by setting $v_{k+1}=0$.
Theorem~\ref{thm:nk-flow-invertible} shows that this relation defines a
new stabilizing system node.

The linearized identity does not determine the Newton factor uniquely: one
may add a skew-adjoint operator acting on the auxiliary input.  We fix
this freedom by choosing once and for all a real parameter
$\alpha>0$ with $\alpha\in\rho(A)$, where $A$ is the main
operator of the original system. Such a choice is always possible for a
$C_0$-semigroup generator. Since $\Sigma_k$ is associated with a stabilizing
feedback, its semigroup is bounded and strongly stable, and hence
$\mathbb C_+\subset\rho(A_k)$. At every Newton step we use the resolvent lift
\[
\iota_kv
=
\spvek{(\alpha-A_{k,-1})^{-1}B_kv}{v},
\qquad v\in U,
\]
and require the operator
\[
(\Gamma_{k+1}\&\Lambda_{k+1})\iota_k:U\to U
\]
to be self-adjoint. Section~\ref{sec:nk-wellposed} proves that this
normalization determines a unique Newton factor. The same value of $\alpha$ is used throughout the iteration.

\begin{remark}[Role of the normalization parameter]
The fixed parameter $\alpha>0$ introduced above is a normalization
parameter for the Newton factor and should not be interpreted as a step
size or regularization parameter. For a fixed current feedback, the associated cost operator
$P_{k+1}$ is independent of $\alpha$; the parameter enters only when
selecting the next Newton factor from the non-unique linearized Lur'e
factorization. Different choices of $\alpha$ may therefore lead to
different intermediate feedbacks and cost-operator iterates, while the limiting
optimal value operator is independent of the fixed normalization parameter
under the assumptions of the convergence theorem. In finite dimensions, for a fixed current feedback, the
resolvent-normalized update converges to the classical
zero-state-normalized Newton--Kleinman update as $\alpha\to\infty$.
\end{remark}

\subsection{Classical Newton--Kleinman and the zero-state normalization}
\label{subsec:nk-matrix}

For comparison, consider the finite-dimensional problem
\eqref{eq:intro-fd-system}--\eqref{eq:intro-fd-cost} and let $K_k$ be the
current stabilizing feedback. In the notation of the preceding subsection,
the current auxiliary input may then be chosen as the standard coordinate
\begin{equation}
w=(I+D^*D)^{1/2}(u-K_kx),
\qquad
u=K_kx+(I+D^*D)^{-1/2}w.
\label{eq:nk-matrix-reopening}
\end{equation}
Then
\[
\dot x=(A+BK_k)x+B(I+D^*D)^{-1/2}w
\]
and
\[
\spvek{y}{u}
=
\sbvek{C+DK_k}{K_k}x
+
\sbvek{D}{I}(I+D^*D)^{-1/2}w.
\]
The coefficient of $w$ in the performance output is isometric.  Hence
feedback evaluation gives
\[
(A+BK_k)^*P_{k+1}+P_{k+1}(A+BK_k)
+(C+DK_k)^*(C+DK_k)+K_k^*K_k=0.
\]
Using this identity in \eqref{eq:nk-lure-form} yields
\[
\mathfrak l_{k,P_{k+1}}\!\left[\spvek{x}{w}\right]
=
\|w\|^2
+
2\Ree\left\langle
(I+D^*D)^{1/2}(K_k-K_{k+1}^{\rm NK})x,w
\right\rangle,
\]
where
\[
K_{k+1}^{\rm NK}
=
-(I+D^*D)^{-1}(B^*P_{k+1}+D^*C).
\]
The zero-state normalization selects
\[
(\Gamma_{k+1}\&\Lambda_{k+1})\spvek{x}{w}
=
w+(I+D^*D)^{1/2}(K_k-K_{k+1}^{\rm NK})x.
\]
By \eqref{eq:nk-matrix-reopening}, this is
\[
(I+D^*D)^{1/2}(u-K_{k+1}^{\rm NK}x).
\]
Thus setting the new auxiliary input to zero gives precisely
$u=K_{k+1}^{\rm NK}x$, while
\begin{equation}
\vartheta_{k+1}
=
(I+D^*D)^{1/2}(K_{k+1}^{\rm NK}-K_k)x.
\label{eq:nk-matrix-correction}
\end{equation}
Moreover, the new auxiliary input again has the standard form
\[
v_{k+1}
=
(I+D^*D)^{1/2}(u-K_{k+1}^{\rm NK}x).
\]
Hence repeating the zero-state-normalized construction recovers exactly the
classical Newton--Kleinman iteration.

For the resolvent normalization, retain the same auxiliary input
coordinate \eqref{eq:nk-matrix-reopening}.  Every factor satisfying
\eqref{eq:nk-linearized-lure} then has the form
\[
(\Gamma_{k+1}\&\Lambda_{k+1})\spvek{x}{w}
=
(I+J)w
+
(I+D^*D)^{1/2}(K_k-K_{k+1}^{\rm NK})x,
\qquad
J^*=-J.
\]
On the resolvent lift
\[
w
\mapsto
\spvek{
(\alpha-A_k)^{-1}
B(I+D^*D)^{-1/2}w
}{
w
},
\]
set
\[
T_{k,\alpha}
=
(I+D^*D)^{1/2}(K_k-K_{k+1}^{\rm NK})
(\alpha-A_k)^{-1}
B(I+D^*D)^{-1/2}.
\]
Requiring the factor on this lift to be self-adjoint on $U$ gives
\[
J
=
J_{k,\alpha}
=
\tfrac12
\bigl(
T_{k,\alpha}^*-T_{k,\alpha}
\bigr).
\]
The corresponding feedback matrix is
\[
K_{k+1}^{(\alpha)}
=
K_k
-
(I+D^*D)^{-1/2}
(I+J_{k,\alpha})^{-1}
(I+D^*D)^{1/2}
(K_k-K_{k+1}^{\rm NK}),
\]
and the new auxiliary input is
\[
v_{k+1}
=
(I+J_{k,\alpha})
(I+D^*D)^{1/2}
\bigl(u-K_{k+1}^{(\alpha)}x\bigr).
\]
Unlike under the zero-state normalization, this auxiliary input is generally no longer
in the standard form \eqref{eq:nk-matrix-reopening} and is propagated to the
next Newton step.

The zero-state normalization is available whenever
\[
\spvek{0}{v}\in\dom(A_k\&B_k)
\qquad\text{for all }v\in U,
\]
or equivalently when the control operator $B_k$ of $\Sigma_k$ maps
$U$ into $X$. If the classical feedback products are then well defined,
the usual operator Newton--Kleinman update is recovered. This includes the
bounded case and, with the appropriate $\Lambda$-extension, regular systems
with bounded control and unbounded observation
\cite{MassoudiReis2016,MassoudiThesis}. For genuinely unbounded control,
however, $B^*P$ need not be meaningful, which motivates the system-node
formulation used here.

\subsection{The Newton decrement}
\label{subsec:nk-increment}

\begin{proposition}[Newton decrement]
\label{prop:nk-decrement}
Let $k\geq1$ and let $\Sigma_k$ be the system node associated with the
stabilizing feedback obtained after $k$ Newton steps.  Denote by
\[
\Theta_k:\dom(A_k)\to U
\]
the zero-input observation operator of the correction output
$\vartheta_k=v_{k-1}-v_k$.  Then
\begin{equation}
P_k-P_{k+1}
=
\mathcal C_{\Theta,k,\infty}^*
\mathcal C_{\Theta,k,\infty}
\geq0.
\label{eq:nk-decrement}
\end{equation}
In particular, $P_{k+1}\leq P_k$.
\end{proposition}
\begin{proof}
Apply \eqref{eq:nk-lure-residual-factorization} from Newton step $k-1$ and
set the new auxiliary input $v_k$ equal to zero. Along classical zero-input
trajectories of $\Sigma_k$ this gives
\[
\tfrac{\mathrm d}{\mathrm dt}\langle P_kx(t),x(t)\rangle
+\|z_k(t)\|^2+\|\vartheta_k(t)\|^2=0.
\]
Integration over $[0,T]$ and density give the corresponding identity for
every initial state. Strong stability of $\mathcal T_k$ allows
$T\to\infty$, so that
\[
P_k
=
\mathcal C_{k,\infty}^*\mathcal C_{k,\infty}
+
\mathcal C_{\Theta,k,\infty}^*\mathcal C_{\Theta,k,\infty}.
\]
Using \eqref{eq:nk-feedback-value} proves \eqref{eq:nk-decrement}.
\end{proof}

The decrement is itself an observability Gramian and hence satisfies
\begin{equation}
2\Ree\langle(P_k-P_{k+1})x,A_kx\rangle+\|\Theta_kx\|^2=0,
\qquad x\in\dom(A_k).
\label{eq:nk-decrement-lyap}
\end{equation}
Under the zero-state normalization, \eqref{eq:nk-matrix-correction} turns
this in finite dimensions into the classical incremental Lyapunov equation
\eqref{eq:intro-fd-increment}. For the resolvent-normalized iteration,
\eqref{eq:nk-decrement-lyap} is the corresponding incremental Lyapunov
identity.

Consequently, the next cost operator can be evaluated in two equivalent
ways. The direct realization uses
$P_{k+1}
=
\mathcal C_{k,\infty}^*\mathcal C_{k,\infty}$,
whereas, for $k\geq1$, the incremental realization uses the positive
Newton decrement,
\begin{equation}
P_{k+1}
=
P_k-
\mathcal C_{\Theta,k,\infty}^*
\mathcal C_{\Theta,k,\infty}.
\label{eq:nk-incremental-update}
\end{equation}
The latter exploits the monotonicity structure of the Newton iteration.
Which realization is preferable computationally depends on the problem
and on the available low-rank representations.

\subsection{ADI and the feedback resolvent}
\label{subsec:nk-computational}

The Gramian representations above make the operator ADI iteration of
\cite{OpmeerReisWollner2013} directly applicable to either realization
of the feedback evaluation. To compute $P_{k+1}$ directly, ADI
is applied to the adjoint pair $(A_k^*,Z_k^*)$. Alternatively, for
$k\geq1$, the incremental realization \eqref{eq:nk-incremental-update}
applies ADI to $(A_k^*,\Theta_k^*)$ and updates the previously computed
cost operator by the resulting decrement.

To describe both cases simultaneously, let $\mathcal O_k$ denote either $Z_k$ in the
direct realization or $\Theta_k$ in the incremental realization. In the
following recursion, $\mathcal O_k^*$ is understood as the adjoint admissible control
operator, and the shifted resolvents are interpreted in the corresponding
extrapolation-space sense; this is precisely the setting of
\cite{OpmeerReisWollner2013}.  For shifts $p_j\in\mathbb C_-$, the ADI
recursion reads
\begin{align*}
V_1&=(A_k^*+p_1I)^{-1}\mathcal O_k^*,
&
S_1&=\sqrt{-2\Ree p_1}\,V_1,\\
V_j&=V_{j-1}-(p_j+\overline{p_{j-1}})
      (A_k^*+p_jI)^{-1}V_{j-1},
&
S_j&=
\begin{bmatrix}
S_{j-1}&\sqrt{-2\Ree p_j}\,V_j
\end{bmatrix},
\qquad j\geq2.
\end{align*}
Then $S_jS_j^*$ approximates the corresponding infinite-time Gramian.
Since $\mathcal T_k$ is boun\-ded and strongly stable and $\mathcal O_k$ is
infinite-time admissible, the convergence theory of
\cite{OpmeerReisWollner2013} applies. In particular,
if the shifts satisfy the non-Blaschke condition
\begin{equation}
\sum_{j=1}^\infty
\frac{-\Ree p_j}{1+|p_j|^2}
=
\infty,
\qquad p_j\in\mathbb C_-,
\label{eq:nk-adi-non-blaschke}
\end{equation}
then $S_jS_j^*$ converges strongly to the infinite-time Gramian.  Condition
\eqref{eq:nk-adi-non-blaschke} holds, for example, for periodic shift
sequences in $\mathbb C_-$. If the observation space of $\mathcal O_k$ is finite dimensional, the ADI
approximants have finite rank. In particular, for the direct feedback-cost
Gramian this is ensured by $\dim Y<\infty$ and $\dim U<\infty$, since the
performance output takes values in $Y\times U$. For a Newton decrement
Gramian, finite dimensionality of $U$ alone ensures finite-rank ADI
approximants.

It remains to evaluate the shifted resolvents required by ADI after a Newton
update without assembling the new generator. Notice that the fixed
normalization parameter $\alpha$ of the Newton iteration is unrelated to
the ADI shift parameters $p_j$.

Fix $\lambda\in\mathbb C_+$. The state
$x=(\lambda-A_{k+1})^{-1}f$ is characterized in the coordinates of
$\Sigma_k$ by
\[
\lambda x- A_k\&B_k\spvek{x}{v_k}=f,
\qquad
\Gamma_{k+1}\&\Lambda_{k+1}\spvek{x}{v_k}=0.
\]
For prescribed $v_k$, the first equation gives
\[
x
=
(\lambda-A_k)^{-1}f
+(\lambda-A_{k,-1})^{-1}B_kv_k.
\]
The corresponding auxiliary-input transfer function is the operator on $U$ defined by
\begin{equation}
\mathbf N_{k+1}(\lambda)v
:=
\Gamma_{k+1}\&\Lambda_{k+1}
\spvek{(\lambda-A_{k,-1})^{-1}B_kv}{v}.
\label{eq:nk-computational-port-symbol}
\end{equation}
Section~\ref{sec:nk-wellposed} proves that $\mathbf N_{k+1}(\lambda)$ is invertible
for all $\lambda\in\mathbb C_+$ and that
$\|\mathbf N_{k+1}(\lambda)^{-1}\|\leq2$. Eliminating $v_k$ therefore yields
\begin{equation}
\begin{aligned}
(\lambda-A_{k+1})^{-1}f
={}&(\lambda-A_k)^{-1}f\\
&-(\lambda-A_{k,-1})^{-1}B_k
  \mathbf N_{k+1}(\lambda)^{-1}
  \Gamma_{k+1}\&\Lambda_{k+1}
  \spvek{(\lambda-A_k)^{-1}f}{0}.
\end{aligned}
\label{eq:nk-computational-resolvent-update}
\end{equation}
Thus the resolvent of $A_{k+1}$ is obtained from resolvent solves for
$A_k$ and an inversion on the input space $U$; the generator $A_{k+1}$
itself need not be assembled.

For the updated system node $\Sigma_{k+1}$, the correction output satisfies
\begin{equation}
\Theta_{k+1}(\lambda-A_{k+1})^{-1}f
=
- \mathbf N_{k+1}(\lambda)^{-1}
  \Gamma_{k+1}\&\Lambda_{k+1}
  \spvek{(\lambda-A_k)^{-1}f}{0}.
\label{eq:nk-computational-correction-resolvent}
\end{equation}
Taking adjoints of
\eqref{eq:nk-computational-resolvent-update} and
\eqref{eq:nk-computational-correction-resolvent} provides exactly the
shifted adjoint solves required in the ADI recursion above. Indeed, for
$p\in\mathbb C_-$ one uses
\[
(A_{k+1}^*+pI)^{-1}
=-\bigl(( -\overline p-A_{k+1})^{-1}\bigr)^*.
\]
The Schur-complement reduction is performed before spatial discretization.
For PDE systems, only the shifted state and adjoint resolvent problems
associated with $A_k$ have to be approximated; the updated generator itself
need not be assembled.

\section{Feasibility and stability of the Newton step}
\label{sec:nk-wellposed}

We establish the two ingredients needed to complete the Newton step:
the construction of the normalized Newton factor and the flow inversion
defining the next feedback. Fix $k\geq0$ and suppose that
$\Sigma_k$ is stabilizing.  To lighten the notation in this section, write
\[
A=A_k,\qquad B=B_k,\qquad Z\&W=Z_k\&W_k,\qquad P=P_{k+1},
\qquad \mathcal E=\dom(A\&B).
\]
The space $\mathcal E$ is equipped with the graph norm.  Since the semigroup
$\mathcal T_k$ is strongly stable, the uniform boundedness principle gives
$\sup_{t\geq0}\|\mathcal T_k(t)\|<\infty$.  Hence
$\Cplus\subset\rho(A)$ by the standard Laplace-resolvent formula; see
\cite[Prop.~2.3.1]{TucsnakWeiss2009}.

Let $\mathfrak b_k$ denote the Hermitian polarization of
\eqref{eq:nk-lure-form} with $P=P_{k+1}$, that is,
\begin{equation}
\begin{aligned}
\mathfrak b_k\!\left(\spvek{x}{v},\spvek{\xi}{w}\right)
={}&
\left\langle A\&B\spvek{x}{v},P\xi\right\rangle
+\left\langle Px,A\&B\spvek{\xi}{w}\right\rangle \\
&+\left\langle
Z\&W\spvek{x}{v},Z\&W\spvek{\xi}{w}
\right\rangle.
\end{aligned}
\label{eq:nk-polarized-form}
\end{equation}
By \eqref{eq:nk-lure-zero-slice} and polarization,
\begin{equation}
\mathfrak b_k\!\left(\spvek{x}{0},\spvek{\xi}{0}\right)=0,
\qquad x,\xi\in\dom(A).
\label{eq:nk-form-kernel-zero}
\end{equation}

\subsection{Construction and admissibility of the Newton factor}
\label{subsec:nk-factor-construction}

For the resol\-vent-normalized iteration, the fixed lift from
Section~\ref{sec:nk} is
\[
\iota_kv
=
\spvek{(\alpha-A_{-1})^{-1}Bv}{v},
\qquad v\in U.
\]  By
\cite[Lem.~4.7.3]{Staffans2005}, $\iota_k$ is bounded from $U$ into
$\mathcal E$.

Since the zero-input--zero-input block vanishes by
\eqref{eq:nk-form-kernel-zero}, the polarized form is determined by its
state--input and input--input blocks. Moreover, for every
$\spvek{x}{v}\in\mathcal E$,
\[
x-(\alpha-A_{-1})^{-1}Bv\in\dom(A),
\]
because subtracting the resolvent lift $\iota_kv$ leaves an element of
$\mathcal E$ with zero input component. Define
$\Delta_{k+1}:\dom(A)\to U$ and
$E_{k+1}=E_{k+1}^*\in\mathcal L(U)$ by
\begin{align}
\langle\Delta_{k+1}x,w\rangle
&=
\mathfrak b_k\!\left(\spvek{x}{0},\iota_kw\right),
& x&\in\dom(A),\ w\in U,
\notag\\
\langle E_{k+1}v,w\rangle
&=
\mathfrak b_k(\iota_kv,\iota_kw),
& v,w&\in U,
\label{eq:nk-factor-data}
\end{align}
and set
\begin{equation}
\Gamma_{k+1}\&\Lambda_{k+1}\spvek{x}{v}
=
\Delta_{k+1}\!\left(
x-(\alpha-A_{-1})^{-1}Bv
\right)
+\tfrac12(E_{k+1}+I)v.
\label{eq:nk-factor}
\end{equation}
\begin{lemma}[Normalized Newton factor]
\label{lem:nk-port-polarization}
Let $k\geq0$ and suppose that $\Sigma_k$ is stabilizing.  Let
$\Delta_{k+1}$ and $E_{k+1}$ be defined by
\eqref{eq:nk-factor-data}.  Then
\[
\Delta_{k+1}\in\mathcal L(\dom(A_k),U),
\qquad
E_{k+1}=E_{k+1}^*\in\mathcal L(U),
\]
where $\dom(A_k)$ carries its graph norm.  The operator
$\Gamma_{k+1}\&\Lambda_{k+1}$ defined by
\eqref{eq:nk-factor} belongs to
\[
\mathcal L(\dom(A_k\&B_k),U)
\]
and satisfies
\begin{equation}
\begin{aligned}
\mathfrak b_k\!\left(\spvek{x}{v},\spvek{\xi}{w}\right)
={}&
\left\langle
(\Gamma_{k+1}\&\Lambda_{k+1})\spvek{x}{v},w
\right\rangle
\\
&+
\left\langle
v,(\Gamma_{k+1}\&\Lambda_{k+1})\spvek{\xi}{w}
\right\rangle
-\langle v,w\rangle
\end{aligned}
\label{eq:nk-polarized-newton}
\end{equation}
for all
$\spvek{x}{v},\spvek{\xi}{w}\in\dom(A_k\&B_k)$.
It is the unique factor in
$\mathcal L(\dom(A_k\&B_k),U)$ satisfying
\eqref{eq:nk-polarized-newton} for which
\[
(\Gamma_{k+1}\&\Lambda_{k+1})\iota_k:U\to U
\]
is self-adjoint.  In particular,
\eqref{eq:nk-linearized-lure} holds and
\begin{equation}
(\Gamma_{k+1}\&\Lambda_{k+1})\spvek{x}{0}
=
\Delta_{k+1}x,
\qquad x\in\dom(A_k).
\label{eq:nk-factor-zero-input}
\end{equation}
\end{lemma}
\begin{proof}
For $x\in\dom(A)$, the map
$w\mapsto\mathfrak b_k(\spvek{x}{0},\iota_kw)$ is a bounded functional on
$U$, and hence defines $\Delta_{k+1}x$ by the Riesz representation theorem.
The boundedness of $\iota_k$ gives
$\Delta_{k+1}\in\mathcal L(\dom(A),U)$, where $\dom(A)$ carries its graph
norm.  Likewise, the restriction
$(v,w)\mapsto\mathfrak b_k(\iota_kv,\iota_kw)$ is a bounded Hermitian form
on $U$ and therefore defines a bounded self-adjoint operator $E_{k+1}$.

Write $\iota_kv=\spvek{x_v}{v}$.  Every
$\spvek{x}{v}\in\dom(A\&B)$ decomposes as
$\spvek{x}{v}=\spvek{x-x_v}{0}+\iota_kv$.  Applying this decomposition to
both arguments of $\mathfrak b_k$ and using
\eqref{eq:nk-form-kernel-zero} gives
\[
\begin{aligned}
\mathfrak b_k\!\left(\spvek{x}{v},\spvek{\xi}{w}\right)
={}&
\langle\Delta_{k+1}(x-x_v),w\rangle
+\langle v,\Delta_{k+1}(\xi-x_w)\rangle
+\langle E_{k+1}v,w\rangle .
\end{aligned}
\]
This is exactly \eqref{eq:nk-polarized-newton} with the factor
\eqref{eq:nk-factor}. On the diagonal it gives
\eqref{eq:nk-linearized-lure}, and $v=0$ gives
\eqref{eq:nk-factor-zero-input}.

If two factors satisfy \eqref{eq:nk-polarized-newton}, their difference
vanishes on all $\spvek{x}{0}$ with $x\in\dom(A)$ and hence has the form $Kv$ with
$K\in\mathcal L(U)$.  The polarized identity implies $K^*=-K$, whereas the
normalization that $(\Gamma_{k+1}\&\Lambda_{k+1})\iota_k$ be self-adjoint forces $K=0$.
\end{proof}

For the subsequent flow inversion we also need the zero-input part of the
Newton factor to be infinite-time admissible. This follows directly from
the finite-time maps of the current system node.

\begin{lemma}[Admissibility of the zero-input factor]
\label{lem:nk-delta-admissible}
Let $\Delta_{k+1}$ be the zero-input part of the Newton factor constructed
in \eqref{eq:nk-factor-data}.  Then
$\Delta_{k+1}:\dom(A_k)\to U$
is infinite-time admissible for $\mathcal T_k$.  If
$\mathcal B_{k,T}$, $\mathcal C_{k,T}$, and $\mathcal D_{k,T}$ denote the
finite-time maps of $\Sigma_k$, then the finite-time
observation maps of $\Delta_{k+1}$ satisfy
\begin{equation}
\mathcal C_{\Delta,k+1,T}
=
\mathcal B_{k,T}^*P_{k+1}\mathcal T_k(T)
+\mathcal D_{k,T}^*\mathcal C_{k,T},
\qquad T>0.
\label{eq:nk-delta-green-map}
\end{equation}
\end{lemma}
\begin{proof}
Let $x\in\dom(A)$ and $\spvek{\xi}{v}\in\mathcal E$. Since
$\spvek{\xi}{v}-\iota_kv$ has zero input component,
\eqref{eq:nk-form-kernel-zero} gives
\[
\langle\Delta_{k+1}x,v\rangle
=
\mathfrak b_k\!\left(\spvek{x}{0},\spvek{\xi}{v}\right).
\]
Hence, by \eqref{eq:nk-polarized-form},
\begin{equation}
\begin{aligned}
\langle\Delta_{k+1}x,v\rangle
={}&
\langle PAx,\xi\rangle
+\left\langle Px, A\&B\spvek{\xi}{v}\right\rangle
+\left\langle Zx,Z\&W\spvek{\xi}{v}\right\rangle .
\end{aligned}
\label{eq:nk-adjoint-correction}
\end{equation}
Let $h\in C_c^\infty((0,T);U)$ and let $(\xi,h,\zeta)$ be the zero-state
classical trajectory of $\Sigma_k$.  Its existence follows from
\cite[Thm.~4.3.9]{Staffans2005}.  Applying
\eqref{eq:nk-adjoint-correction} to $\mathcal T_k(t)x$ and
$\spvek{\xi(t)}{h(t)}$, and integrating over $[0,T]$, yields
\[
\int_0^T
\langle\Delta_{k+1}\mathcal T_k(t)x,h(t)\rangle\,dt
=
\left\langle
\bigl(\mathcal B_{k,T}^*P\mathcal T_k(T)
+\mathcal D_{k,T}^*\mathcal C_{k,T}\bigr)x,h
\right\rangle_{L^2}.
\]
Density gives \eqref{eq:nk-delta-green-map}.  Since $\Sigma_k$ is
infinite-time well posed and $\mathcal T_k$ is bounded, the right-hand side
of
\[
\|\mathcal C_{\Delta,k+1,T}\|
\leq
\|\mathcal B_{k,T}\|\,\|P\|\,\|\mathcal T_k(T)\|
+\|\mathcal D_{k,T}\|\,\|\mathcal C_{k,T}\|
\]
is bounded uniformly in $T$.  Thus $\Delta_{k+1}$ is infinite-time
admissible.
\end{proof}

We denote the resulting infinite-time observation map by
$\mathcal C_{\Delta,k+1,\infty}$.

\subsection{Uniform invertibility of the Newton transfer function}

Recall the transfer function $\mathbf N_{k+1}(\lambda)$ from
\eqref{eq:nk-computational-port-symbol}.  The following estimate is the key
to both flow inversion and the numerical resolvent formula.
\begin{theorem}[Uniform invertibility of the Newton transfer function]
\label{thm:nk-newton-symbol}
Let $k\geq0$, suppose that $\Sigma_k$ is stabilizing, and let
$\Gamma_{k+1}\&\Lambda_{k+1}$ be the normalized Newton factor from
Lemma~\ref{lem:nk-port-polarization}.  Then, for all
$\lambda\in\Cplus$ and $v\in U$,
\begin{equation}
\begin{aligned}
2\Ree\langle \mathbf N_{k+1}(\lambda)v,v\rangle
={}&
\|v\|^2
+
2\Ree\lambda\,
\left\langle
P_{k+1}(\lambda-A_{k,-1})^{-1}B_kv,
(\lambda-A_{k,-1})^{-1}B_kv
\right\rangle
\\
&+
\left\|
(Z_k\&W_k)
\spvek{(\lambda-A_{k,-1})^{-1}B_kv}{v}
\right\|^2 .
\end{aligned}
\label{eq:nk-frequency-green}
\end{equation}
Consequently,
\begin{equation}
\Ree\langle \mathbf N_{k+1}(\lambda)v,v\rangle
\geq\tfrac12\|v\|^2,
\qquad \lambda\in\Cplus,
\label{eq:nk-strict-accretivity}
\end{equation}
and $\mathbf N_{k+1}(\lambda)$ is invertible with
\begin{equation}
\sup_{\lambda\in\Cplus}
\|\mathbf N_{k+1}(\lambda)^{-1}\|
\leq2.
\label{eq:nk-Q-Hinf}
\end{equation}
Moreover,
\[
\lambda\mapsto\mathbf N_{k+1}(\lambda)^{-1}
\]
belongs to $H^\infty(\Cplus;\mathcal L(U))$, and
\begin{equation}
\mathbf N_{k+1}(\lambda)-\mathbf N_{k+1}(\mu)
=
\Delta_{k+1}
\bigl(
(\lambda-A_{k,-1})^{-1}
-
(\mu-A_{k,-1})^{-1}
\bigr)B_k.
\label{eq:nk-port-resolvent-identity}
\end{equation}
\end{theorem}
\begin{proof}
The element
$\spvek{(\lambda-A_{-1})^{-1}Bv}{v}$ satisfies
\[
A\&B
\spvek{(\lambda-A_{-1})^{-1}Bv}{v}
=
\lambda(\lambda-A_{-1})^{-1}Bv.
\]
Substitution into \eqref{eq:nk-linearized-lure} gives
\eqref{eq:nk-frequency-green}.  Since $P\geq0$,
\eqref{eq:nk-strict-accretivity} follows. By the Cauchy--Schwarz inequality,
\eqref{eq:nk-strict-accretivity} implies
\[
\|\mathbf N_{k+1}(\lambda)v\|
\geq \tfrac12\|v\|,
\qquad v\in U.
\]
Hence $\mathbf N_{k+1}(\lambda)$ is injective and has closed range.
The same accretivity estimate holds for
$\mathbf N_{k+1}(\lambda)^*$, so that
$\ker \mathbf N_{k+1}(\lambda)^*=\{0\}$. Consequently,
$\operatorname{ran}\mathbf N_{k+1}(\lambda)$ is dense in $U$.
It is therefore equal to $U$, and
\[
\|\mathbf N_{k+1}(\lambda)^{-1}\|\leq2.
\]
This proves invertibility and \eqref{eq:nk-Q-Hinf}.  The map
$\lambda\mapsto\mathbf N_{k+1}(\lambda)$ is holomorphic on $\Cplus$
by the system-node resolvent representation. Since it is invertible
throughout $\Cplus$, the holomorphic inverse theorem shows that
$\lambda\mapsto\mathbf N_{k+1}(\lambda)^{-1}$ is holomorphic.
Together with \eqref{eq:nk-Q-Hinf}, this gives
$\mathbf N_{k+1}^{-1}\in H^\infty(\Cplus;\mathcal L(U))$.
Finally, subtracting the two resolvent lifts in the definition of
$\mathbf N_{k+1}$ and using \eqref{eq:nk-factor-zero-input} gives
\eqref{eq:nk-port-resolvent-identity}.
\end{proof}

\subsection{Flow inversion and preservation of stability}

We can now complete one Newton step.
\begin{theorem}[Feasibility of one Newton step]
\label{thm:nk-flow-invertible}
Let $k\geq0$ and suppose that $\Sigma_k$ is stabilizing.  Let
$\Gamma_{k+1}\&\Lambda_{k+1}$ be the normalized Newton factor from
Lemma~\ref{lem:nk-port-polarization}.  Then the port relation
\[
v_{k+1}
=
(\Gamma_{k+1}\&\Lambda_{k+1})
\spvek{x}{v_k}
\]
is flow invertible.  Its flow inverse is an infinite-time well-posed system
node with external input $v_{k+1}$, and its zero-input semigroup
$\mathcal T_{k+1}$ is strongly stable. Hence setting $v_{k+1}=0$ defines
the next stabilizing feedback, with associated system node $\Sigma_{k+1}$.

Its generator $A_{k+1}$ satisfies
\begin{equation}
\begin{aligned}
(\lambda-A_{k+1})^{-1}
={}&
(\lambda-A_k)^{-1}
\\
&-
(\lambda-A_{k,-1})^{-1}B_k
\mathbf N_{k+1}(\lambda)^{-1}
\Delta_{k+1}(\lambda-A_k)^{-1},
\qquad \lambda\in\Cplus.
\end{aligned}
\label{eq:nk-closed-resolvent}
\end{equation}
Moreover, the zero-input observation operator
$\Theta_{k+1}:\dom(A_{k+1})\to U$
of the correction output
$\vartheta_{k+1}=v_k-v_{k+1}$ is infinite-time admissible for
$\mathcal T_{k+1}$.
\end{theorem}
\begin{proof}
The proof proceeds in three steps: we first construct the candidate
closed-loop semigroup, then prove its strong stability, and finally identify
it with the flow inverse of the Newton port.

Let $\mathfrak Q_{k+1}$ be the bounded causal time-invariant operator
on $L^2(0,\infty;U)$ whose symbol is $\mathbf N_{k+1}^{-1}$; its existence and
$\|\mathfrak Q_{k+1}\|\leq2$ follow from the multiplier theorem recalled in
Section~\ref{sec:prelim} and Theorem~\ref{thm:nk-newton-symbol}.

For $\lambda\in\Cplus$ denote the right-hand side of
\eqref{eq:nk-closed-resolvent} by $M_\lambda$.  The usual resolvent identity,
combined with \eqref{eq:nk-port-resolvent-identity} and the identity for
inverses, gives
\begin{equation}
M_\lambda-M_\mu
=(\mu-\lambda)M_\lambda M_\mu,
\qquad \lambda,\mu\in\Cplus.
\label{eq:nk-pseudo-resolvent}
\end{equation}
For $x\in X$, put
\[
r_x:=
-\mathfrak Q_{k+1}\mathcal C_{\Delta,k+1,\infty}x
\in L^2(0,\infty;U)
\]
and define
\[
\mathcal T_{k+1}(t)x
:=
\mathcal T_k(t)x
+
\mathcal B_{k,t}(r_x|_{[0,t]}).
\]
Since $\Sigma_k$ is infinite-time well posed,
$\mathfrak Q_{k+1}$ is bounded, and
$\Delta_{k+1}$ is infinite-time admissible, this family is uniformly
bounded in $t\geq0$. Moreover,
$t\mapsto\mathcal B_{k,t}(r_x|_{[0,t]})$ is the zero-state trajectory
of $\Sigma_k$ driven by $r_x$ and is therefore continuous in $X$.
Hence $\mathcal T_{k+1}$ is a strongly continuous uniformly bounded
family with $\mathcal T_{k+1}(0)=I$. The Laplace
formulas for admissible control and observation operators
\cite[Prop.~2.3.1 and Secs.~4.3--4.4]{TucsnakWeiss2009} give
\begin{equation}
\int_0^\infty e^{-\lambda t}\mathcal T_{k+1}(t)x\,dt
=M_\lambda x,
\qquad \lambda\in\Cplus.
\label{eq:nk-laplace-M}
\end{equation}
For $\lambda,\mu\in\Cplus$, $\lambda\neq\mu$, the double Laplace
transform of $(t,s)\mapsto\mathcal T_{k+1}(t+s)x$ is
\[
\frac{M_\lambda-M_\mu}{\mu-\lambda}x
=
M_\lambda M_\mu x
\]
by \eqref{eq:nk-pseudo-resolvent}. On the other hand, using
\eqref{eq:nk-laplace-M} first in the $s$-variable and then in the
$t$-variable gives
\[
\begin{aligned}
&\int_0^\infty\int_0^\infty
e^{-\lambda t-\mu s}
\mathcal T_{k+1}(t)\mathcal T_{k+1}(s)x\,ds\,dt
\\
&\qquad=
\int_0^\infty
e^{-\lambda t}\mathcal T_{k+1}(t)
\left(
\int_0^\infty e^{-\mu s}\mathcal T_{k+1}(s)x\,ds
\right)\,dt
\\
&\qquad=
\int_0^\infty
e^{-\lambda t}\mathcal T_{k+1}(t)M_\mu x\,dt
=
M_\lambda M_\mu x.
\end{aligned}
\]
Thus the two functions
\[
(t,s)\mapsto\mathcal T_{k+1}(t+s)x
\qquad\text{and}\qquad
(t,s)\mapsto\mathcal T_{k+1}(t)\mathcal T_{k+1}(s)x
\]
have the same double Laplace transform. After pairing their difference
with an arbitrary element of $X$, scalar uniqueness of the double
Laplace transform gives equality almost everywhere in $(t,s)$.
Since both functions are strongly continuous, the equality holds for
all $t,s\geq0$. Hence
\[
\mathcal T_{k+1}(t+s)
=
\mathcal T_{k+1}(t)\mathcal T_{k+1}(s),
\qquad t,s\geq0.
\]
Hence
$\mathcal T_{k+1}$ is a bounded $C_0$-semigroup, and
\eqref{eq:nk-laplace-M} yields \eqref{eq:nk-closed-resolvent} for its
generator $A_{k+1}$.

The new semigroup is strongly stable. Indeed, by the definition above,
$\mathcal T_{k+1}(t)x=\mathcal T_k(t)x+
\mathcal B_{k,t}(r_x|_{[0,t]})$.  For $t>S$, the concatenation identity gives
\[
\mathcal B_{k,t}(r_x|_{[0,t]})
=
\mathcal T_k(t-S)\mathcal B_{k,S}(r_x|_{[0,S]})
+
\mathcal B_{k,t-S}(r_x(S+\cdot)|_{[0,t-S]}).
\]
The first term tends to zero for fixed $S$ by strong stability of
$\mathcal T_k$, whereas the second is bounded by
$\sup_T\|\mathcal B_{k,T}\|\,\|r_x\|_{L^2(S,\infty;U)}$.  Letting first
$t\to\infty$ and then $S\to\infty$ proves
$\mathcal T_{k+1}(t)x\to0$.

It remains to identify this semigroup with the feedback relation generated by the Newton step.  Since
$\Gamma_{k+1}\&\Lambda_{k+1}\in\mathcal L(\dom(A\&B),U)$, appending it as an artificial
output to the current state equation gives a system node.  At the normalization point $\alpha$, the resolvent map in Staffans'
flow-inversion criterion is solved, for $f\in X$ and $g\in U$, by
\[
\begin{aligned}
v_k&=-\mathbf N_{k+1}(\alpha)^{-1}
\bigl(g+\Delta_{k+1}(\alpha-A)^{-1}f\bigr),\\
x&=(\alpha-A)^{-1}f
+(\alpha-A_{-1})^{-1}Bv_k.
\end{aligned}
\]
It is therefore bijective, and the upper-left block of its inverse is
$M_{11}=(\alpha-A_{k+1})^{-1}$.
Thus $M_{11}$ is injective with dense range and
$-M_{11}^{-1}=A_{k+1}-\alpha I$
generates a $C_0$-semigroup.  Staffans' flow-inversion criterion
\cite[Thm.~6.6.11]{Staffans2005} now shows that the auxiliary node is flow
invertible and that its flow inverse is a system node.

We next identify its finite-time maps.  First consider classical
trajectories of the flow inverse with compatible smooth data.  By the
trajectory-interchange theorem
\cite[Thm.~6.6.15]{Staffans2005}, interchanging $v_k$ and $v_{k+1}$
produces a classical trajectory of the auxiliary node.  The Laplace-domain
representation of system-node trajectories therefore gives
\[
\widehat v_{k+1}(\lambda)
=
\Delta_{k+1}(\lambda-A)^{-1}x_0
+
\mathbf N_{k+1}(\lambda)\widehat v_k(\lambda).
\]
Since $\mathbf N_{k+1}^{-1}$ is the symbol of
$\mathfrak Q_{k+1}$, causality implies on every finite interval
\[
v_k
=
\mathfrak Q_{k+1,T}
\bigl(v_{k+1}-\mathcal C_{\Delta,k+1,T}x_0\bigr).
\]
Substitution into the finite-time maps of $\Sigma_k$ gives
\begin{equation}
\begin{bmatrix}
x(T)\\ z\\ v_k
\end{bmatrix}
=
\begin{bmatrix}
\mathcal T_k(T)-\mathcal B_{k,T}\mathfrak Q_{k+1,T}
\mathcal C_{\Delta,k+1,T}
&
\mathcal B_{k,T}\mathfrak Q_{k+1,T}
\\
\mathcal C_{k,T}-\mathcal D_{k,T}\mathfrak Q_{k+1,T}
\mathcal C_{\Delta,k+1,T}
&
\mathcal D_{k,T}\mathfrak Q_{k+1,T}
\\
-\mathfrak Q_{k+1,T}\mathcal C_{\Delta,k+1,T}
&
\mathfrak Q_{k+1,T}
\end{bmatrix}
\pvek{x_0}{v_{k+1}}.
\label{eq:nk-inverse-state-maps}
\end{equation}

Compatible classical data are dense in the generalized-trajectory topology;
see \cite[Secs.~4.6--4.7]{Staffans2005}.  Since every block in
\eqref{eq:nk-inverse-state-maps} is bounded, the identity extends by density
to arbitrary initial states in $X$ and inputs in $L^2(0,T;U)$.  Thus these
are precisely the finite-time maps of the flow inverse.

Moreover, their norms are bounded uniformly in $T$, because $\Sigma_k$ is
infinite-time well posed, $\Delta_{k+1}$ is infinite-time admissible, and
\[
\|\mathfrak Q_{k+1,T}\|
\leq
\|\mathfrak Q_{k+1}\|
\leq2.
\]
Hence the flow inverse is infinite-time well posed.  For zero new input,
\[
v_k
=
-\mathfrak Q_{k+1}
\mathcal C_{\Delta,k+1,\infty}x_0.
\]
Since $\vartheta_{k+1}=v_k-v_{k+1}$, this proves infinite-time
admissibility of the zero-input correction observation
$\Theta_{k+1}$. Together with the strong stability proved above, this shows
that the feedback associated with $\Sigma_{k+1}$ is stabilizing.
\end{proof}

Theorem~\ref{thm:nk-flow-invertible} closes the induction in
Section~\ref{sec:nk}.  Its resolvent formula is exactly the Schur-complement
update \eqref{eq:nk-computational-resolvent-update} used in the ADI
realization.
\section{Convergence}
\label{sec:nk-convergence}

The feasibility result of Section~\ref{sec:nk-wellposed} and the decrement
identity of Proposition~\ref{prop:nk-decrement} show that every stabilizing
initialization generates an infinite sequence of stabilizing feedbacks.  We
first show that the feedback cost operators decrease monotonically and that $P_kx$
converges for all $x\in X$, without any restriction on the input space.
We then identify the limit of the resolvent-normalized iteration with the
optimal value operator, assuming finite-cost detectability and
$\dim U<\infty$. Pointwise convergence $P_kx\to Px$ for all $x\in X$
is usually called strong operator convergence.

\subsection{Monotone convergence}

\begin{theorem}[Monotone convergence of the feedback cost operators]
\label{thm:nk-monotone-convergence}
Let the resolvent-normalized iteration be initialized by a stabilizing
feedback. Then
\begin{equation}
P_\star\leq P_{k+1}\leq P_k\leq P_1,
\qquad k\geq1.
\label{eq:nk-order}
\end{equation}
There exists a bounded nonnegative self-adjoint operator
$P_\infty\in\mathcal L(X)$ such that
\begin{equation}
P_kx\rightarrow P_\infty x,
\qquad x\in X.
\label{eq:nk-strong-limit}
\end{equation}
\end{theorem}

\begin{proof}
For $k\geq1$, $P_k$ is the cost operator of the zero-input trajectory of
$\Sigma_{k-1}$.  This trajectory is admissible for the original optimal
control problem, and hence $P_\star\leq P_k$.  Proposition~\ref{prop:nk-decrement}
gives $P_k-P_{k+1}\geq0$, proving \eqref{eq:nk-order}.

For each $x\in X$, the sequence $\langle P_kx,x\rangle$ is decreasing and
bounded from below. Define
\[
q_\infty(x)
=
\lim_{k\to\infty}\langle P_kx,x\rangle.
\]
The quadratic-form identities satisfied by
$\langle P_k\,\cdot,\cdot\rangle$ pass to the pointwise limit, and
\[
0\leq q_\infty(x)\leq\langle P_1x,x\rangle
\leq\|P_1\|\,\|x\|^2.
\]
Hence polarization yields a bounded nonnegative self-adjoint operator
$P_\infty\in\mathcal L(X)$ satisfying
$q_\infty(x)=\langle P_\infty x,x\rangle$ and
$P_\star\leq P_\infty\leq P_k$. Put $D_k=P_k-P_\infty$. Then $D_k\geq0$
and
\[
\|D_kx\|^2
\leq
\|P_1-P_\infty\|\,\langle D_kx,x\rangle
\rightarrow0,
\]
which proves \eqref{eq:nk-strong-limit}.
\end{proof}

\subsection{A dissipation inequality for the limit}
\label{subsec:nk-limit-dissipation}

For the remainder of this section we assume $\dim U<\infty$ and use the
same normalization parameter $\alpha>0$ that was fixed for the Newton
iteration, with $\alpha\in\rho(A)$, where $A$ is the main operator of the
original system node.
The Cayley transform provides fixed bounded coordinates in which all
successive auxiliary inputs can be compared on the same space $X\times U$.
We recall only the identities needed below; see
\cite[Sec.~12.3]{Staffans2005} for the general construction.

For $\spvek{x}{u}\in\dom(A\&B)$ set
\begin{equation}
\xi
=
\frac{\alpha x- A\&B\spvek{x}{u}}{\sqrt{2\alpha}},
\qquad
x^+
=
\frac{\alpha x+A\&B\spvek{x}{u}}{\sqrt{2\alpha}}.
\label{eq:nk-cayley-coordinates}
\end{equation}
Using the augmented performance output $(y,u)$, the internal Cayley
transform has bounded blocks
$\widehat A,\widehat B,\widehat C,\widehat D$ such that
\[
x^+=\widehat A\xi+\widehat Bu,
\qquad
\spvek{y}{u}=\widehat C\xi+\widehat Du.
\]
The second component of $\widehat C\xi$ is zero and the second component
of $\widehat Du$ is $u$. Hence
$\widehat D^*\widehat D\geq I_U$.

\begin{lemma}[Cayley coordinates and auxiliary inputs]
\label{lem:nk-cayley-ports}
Let $\alpha>0$ be the fixed normalization parameter of the
resolvent-normalized iteration, with $\alpha\in\rho(A)$, where
$A$ is the main operator of the original system node.  Then the map
\[
(\xi,u)
\mapsto
\spvek{x}{u},
\qquad
x=
\sqrt{2\alpha}\,(\alpha-A)^{-1}\xi
+
(\alpha-A_{-1})^{-1}Bu,
\]
is a bounded isomorphism from $X\times U$ onto
$\dom(A\&B)$ equipped with its graph norm, with inverse given by
\eqref{eq:nk-cayley-coordinates}.

For every $k\geq0$, the auxiliary input $v_k$, pulled back to the
original system-node domain, is bounded with respect to the graph norm.  Consequently there exist
unique operators
$F_k\in\mathcal L(X,U)$,
$H_k\in\mathcal L(U)$,
such that
\begin{equation}
v_k=F_k\xi+H_ku.
\label{eq:nk-cayley-port}
\end{equation}
If $H_k$ is invertible, then the resolvent lift
\[
\iota_kv
=
\spvek{(\alpha-A_{k,-1})^{-1}B_kv}{v}
\]
corresponds to the fixed Cayley coordinates
$(0,H_k^{-1}v)$, and hence
\begin{equation}
(\Gamma_{k+1}\&\Lambda_{k+1})\iota_k
=
H_{k+1}H_k^{-1}.
\label{eq:nk-cayley-lift}
\end{equation}
By the normalization of the Newton factor, this operator is self-adjoint.
\end{lemma}

\begin{proof}
Solving the first equation in \eqref{eq:nk-cayley-coordinates} for $x$
gives the displayed formula.  Moreover,
\[
A_{-1}x+Bu
=
\alpha x-\sqrt{2\alpha}\,\xi\in X,
\]
so \eqref{eq:prelim-node-domain} shows that
$\spvek{x}{u}\in\dom(A\&B)$.  The formula and
\eqref{eq:prelim-node-domain} give boundedness into the graph norm, whereas
\eqref{eq:nk-cayley-coordinates} gives boundedness of the inverse.  This is
the standard internal Cayley parametrization; see
\cite[Sec.~12.3]{Staffans2005}.

For $k=0$, boundedness of the auxiliary-input coordinate with respect to
the graph norm is part of the system node associated with the initial
stabilizing feedback.  In one Newton step the flow inversion
maps
\[
\spvek{x}{v_k}
\mapsto
\spvek{x}{v_{k+1}},
\qquad
v_{k+1}=\Gamma_{k+1}\&\Lambda_{k+1}\spvek{x}{v_k}.
\]
This map is bounded between the corresponding system-node domains,
equipped with their graph norms, because
$\Gamma_{k+1}\&\Lambda_{k+1}$ is bounded with respect to the graph norm
and the state derivative is unchanged.
Its inverse is bounded because the old auxiliary input is a bounded output of
the flow-inverted system node. Thus repeated flow inversion gives a
bounded isomorphism back to the original system-node domain; see
\cite[Thms.~6.6.11 and 6.6.15]{Staffans2005}.  Composing the pulled-back
auxiliary-input coordinate with the Cayley isomorphism therefore gives a bounded map
from $X\times U$ to $U$, which has the unique block representation
\eqref{eq:nk-cayley-port}.

Finally, the state component of $\iota_kv$ satisfies
\[
(A_k\&B_k)\iota_kv
=
\alpha(\alpha-A_{k,-1})^{-1}B_kv.
\]
Hence its fixed Cayley state coordinate is zero.  If $H_k$ is invertible,
\eqref{eq:nk-cayley-port} then gives $u=H_k^{-1}v$, and evaluation of the
next auxiliary input yields
\[
(\Gamma_{k+1}\&\Lambda_{k+1})\iota_kv
=
H_{k+1}H_k^{-1}v.
\]
The left-hand side is self-adjoint as an operator on $U$ by the fixed
normalization, and hence so is $H_{k+1}H_k^{-1}$.
\end{proof}

For $P=P^*\in\mathcal L(X)$ define
\begin{equation}
\begin{aligned}
d_P(\xi,u)
:={}&
\langle P(\widehat A\xi+\widehat Bu),
        \widehat A\xi+\widehat Bu\rangle
-\langle P\xi,\xi\rangle
+\|\widehat C\xi+\widehat Du\|^2.
\end{aligned}
\label{eq:nk-cayley-form}
\end{equation}
A direct calculation using \eqref{eq:nk-cayley-coordinates} shows that
\begin{equation}
d_P(\xi,u)
=
2\Ree\left\langle
Px,A\&B\spvek{x}{u}
\right\rangle
+
\left\|C\&D\spvek{x}{u}\right\|^2
+\|u\|^2.
\label{eq:nk-cayley-continuous-form}
\end{equation}
Pulling \eqref{eq:nk-linearized-lure} back to these coordinates and using
\eqref{eq:nk-cayley-port} gives
\begin{equation}
\begin{aligned}
d_{P_{k+1}}(\xi,u)
={}&
2\Ree\langle F_{k+1}\xi+H_{k+1}u,
              F_k\xi+H_ku\rangle
-\|F_k\xi+H_ku\|^2
\\
={}&
\|F_{k+1}\xi+H_{k+1}u\|^2
-
\|(F_{k+1}-F_k)\xi+(H_{k+1}-H_k)u\|^2 .
\end{aligned}
\label{eq:nk-cayley-newton}
\end{equation}
The next lemma is the only point where finite dimensionality of the input
space is used.

\begin{lemma}[Vanishing auxiliary-input increments]
\label{lem:nk-port-increments}
Assume $\dim U<\infty$, and let $F_k$ and $H_k$ be the operators from
Lemma~\ref{lem:nk-cayley-ports}.  Then, for all $\xi\in X$ and $u\in U$,
\begin{equation}
(F_{k+1}-F_k)\xi+(H_{k+1}-H_k)u
\rightarrow0.
\label{eq:nk-port-increments}
\end{equation}
\end{lemma}
\begin{proof}
We first control the input blocks $H_k$, then the corresponding Cayley
feedbacks, and finally combine the two increment estimates.

Setting $\xi=0$ in \eqref{eq:nk-cayley-newton} gives
\begin{equation}
\widehat B^*P_{k+1}\widehat B+\widehat D^*\widehat D
=
H_{k+1}^*H_{k+1}
-(H_{k+1}-H_k)^*(H_{k+1}-H_k).
\label{eq:nk-H-difference}
\end{equation}
Since $P_{k+1}\geq0$ and $\widehat D^*\widehat D\geq I_U$,
$H_{k+1}^*H_{k+1}\geq I_U$.  Finite dimensionality of $U$ therefore implies
that $H_{k+1}$ is invertible and $\|H_{k+1}^{-1}\|\leq1$.

For $k\geq1$, $H_k$ is invertible by the preceding estimate. Set
\[
R_k:=H_{k+1}H_k^{-1}.
\]
By \eqref{eq:nk-cayley-lift} and
Lemma~\ref{lem:nk-port-polarization}, $R_k$ is self-adjoint.
Taking the $U$--$U$ part of the first identity in
\eqref{eq:nk-cayley-newton} therefore yields
\[
\widehat B^*P_{k+1}\widehat B+\widehat D^*\widehat D
=
H_k^*(2R_k-I)H_k.
\]
By monotonicity of $P_k$ and \eqref{eq:nk-H-difference} from the preceding
Newton step,
\[
I_U
\leq
\widehat B^*P_{k+1}\widehat B+\widehat D^*\widehat D
\leq
\widehat B^*P_k\widehat B+\widehat D^*\widehat D
\leq H_k^*H_k.
\]
Consequently,
\begin{equation}
\tfrac12I_U
\leq
R_k
\leq I_U,
\qquad k\geq1.
\label{eq:nk-H-ratio}
\end{equation}
Since
\[
H_{k+1}^*H_{k+1}
=
H_k^*R_k^2H_k
\leq H_k^*H_k,
\]
the sequence $H_k^*H_k$ is decreasing and, since $U$ is finite dimensional,
converges in operator norm. Moreover,
\[
I-R_k^2
=
H_k^{-*}
\bigl(H_k^*H_k-H_{k+1}^*H_{k+1}\bigr)
H_k^{-1}
\rightarrow0,
\]
where we also use $\|H_k^{-1}\|\leq1$. By
\eqref{eq:nk-H-ratio},
\[
I-R_k
=
(I-R_k^2)(I+R_k)^{-1},
\]
and hence $R_k\to I$ in operator norm. Since $(H_k)$ is bounded,
\[
H_{k+1}-H_k=(R_k-I)H_k,
\]
and therefore
\begin{equation}
\|H_{k+1}-H_k\|\rightarrow0.
\label{eq:nk-H-increments}
\end{equation}
For $k\geq1$ put
$\widehat K_k=-H_k^{-1}F_k$.
The relation $v_k=0$ is then
$u=\widehat K_k\xi$.  Feedback evaluation, expressed through
\eqref{eq:nk-cayley-form}, gives
$d_{P_{k+1}}(\xi,\widehat K_k\xi)=0$.  Since the performance output contains
the physical input, this implies
\begin{equation}
\|\widehat K_k\xi\|^2
\leq
\|\widehat C\xi+\widehat D\widehat K_k\xi\|^2
\leq
\langle P_{k+1}\xi,\xi\rangle
\leq
\langle P_1\xi,\xi\rangle,
\label{eq:nk-K-bound}
\end{equation}
so
\[
\|\widehat K_k\|\leq\|P_1\|^{1/2},
\qquad k\geq1.
\]
For the feedback defined by $v_{k+1}=0$, one has
$u=\widehat K_{k+1}\xi$, and the correction output is
$H_k(\widehat K_{k+1}-\widehat K_k)\xi$.
Put
\[
D_{k+1}=P_{k+1}-P_{k+2}\geq0
\]
and
\[
x_{k+1}^+
=
\widehat A\xi+\widehat B\widehat K_{k+1}\xi.
\]
Applying the Cayley transformation to the decrement Lyapunov equation
\eqref{eq:nk-decrement-lyap} yields
\[
\langle D_{k+1}x_{k+1}^+,x_{k+1}^+\rangle
-
\langle D_{k+1}\xi,\xi\rangle
+
\|H_k(\widehat K_{k+1}-\widehat K_k)\xi\|^2
=
0.
\]
Since $D_{k+1}\geq0$, it follows that
\[
\|H_k(\widehat K_{k+1}-\widehat K_k)\xi\|^2
\leq
\langle(P_{k+1}-P_{k+2})\xi,\xi\rangle.
\]
The right-hand side tends to zero by
Theorem~\ref{thm:nk-monotone-convergence}. Finally,
\[
(F_{k+1}-F_k)\xi+(H_{k+1}-H_k)u
=
(H_{k+1}-H_k)(u-\widehat K_{k+1}\xi)
+H_k(\widehat K_k-\widehat K_{k+1})\xi,
\]
and \eqref{eq:nk-port-increments} follows from
\eqref{eq:nk-H-increments}, \eqref{eq:nk-K-bound}, and the preceding
estimate.
\end{proof}
\begin{proposition}[Limit dissipation inequality]
\label{prop:nk-limit-dissipation}
Assume $\dim U<\infty$ and consider the resolvent-normalized iteration with
fixed normalization parameter
$\alpha>0$, $\alpha\in\rho(A)$.
Let $P_\infty$ be the limit from
Theorem~\ref{thm:nk-monotone-convergence}.  Then
\begin{equation}
2\Ree\left\langle
P_\infty x,A\&B\spvek{x}{u}
\right\rangle
+
\left\|C\&D\spvek{x}{u}\right\|^2
+
\|u\|^2
\geq0
\label{eq:nk-limit-dissipation}
\end{equation}
for all $\spvek{x}{u}\in\dom(A\&B)$ of the original system node.
\end{proposition}
\begin{proof}
Fix $(\xi,u)\in X\times U$.  Since the four Cayley operators are bounded,
Theorem~\ref{thm:nk-monotone-convergence} gives
$d_{P_{k+1}}(\xi,u)\to d_{P_\infty}(\xi,u)$.  By
\eqref{eq:nk-cayley-newton} and Lemma~\ref{lem:nk-port-increments},
\[
d_{P_\infty}(\xi,u)
=
\lim_{k\to\infty}\|F_{k+1}\xi+H_{k+1}u\|^2
\geq0.
\]
Equation~\eqref{eq:nk-cayley-continuous-form} now gives
\eqref{eq:nk-limit-dissipation}.
\end{proof}

\subsection{Convergence to the optimal value operator}
\label{subsec:nk-optimal-limit}

\begin{theorem}[Convergence to the optimal value operator]
\label{thm:nk-optimal-convergence}
Let the resolvent-normalized iteration be initialized by a stabilizing
feedback, and fix any $\alpha>0$ with
$\alpha\in\rho(A)$, where $A$ is the main operator of the original
system node.  If the original system is finite-cost detectable and
$\dim U<\infty$, then
\begin{equation}
P_kx\rightarrow P_\star x
\qquad\text{for all }x\in X.
\label{eq:nk-convergence-Pstar}
\end{equation}
\end{theorem}
\begin{proof}
By Theorem~\ref{thm:nk-monotone-convergence},
$P_kx\to P_\infty x$ for all $x\in X$, and $P_\star\leq P_\infty$.
It remains to prove the reverse inequality.

Let $(x,u,y)$ be any infinite-time generalized trajectory of the original
system with finite cost. Integrating \eqref{eq:nk-limit-dissipation} along
classical trajectories and then using the approximation from
Section~\ref{subsec:prelim-system-nodes} gives the same inequality for
generalized trajectories: the states converge in $C([0,T];X)$ and the
input and output signals converge in their respective $L^2$ spaces.
Consequently, for every $T>0$,
\[
\int_0^T\bigl(\|y(t)\|^2+\|u(t)\|^2\bigr)\,dt
\geq
\langle P_\infty x(0),x(0)\rangle
-
\langle P_\infty x(T),x(T)\rangle.
\]
By finite-cost detectability, $x(T)\to0$. Letting
$T\to\infty$ therefore yields
\[
J(u,y)\geq\langle P_\infty x(0),x(0)\rangle.
\]
Taking the infimum over all finite-cost trajectories with the same initial
state gives $P_\star\geq P_\infty$.  Hence $P_\infty=P_\star$, and
\eqref{eq:nk-convergence-Pstar} follows from
Theorem~\ref{thm:nk-monotone-convergence}.
\end{proof}

\begin{remark}
Finite dimensionality of $U$ is used only in
Lemma~\ref{lem:nk-port-increments}; feasibility, stability preservation, and
the convergence $P_kx\to P_\infty x$ hold for arbitrary Hilbert input
spaces.  Finite-cost detectability is needed only to identify
$P_\infty=P_\star$.  Hence, under the assumptions of
Theorem~\ref{thm:nk-optimal-convergence}, the limit is independent of the
fixed normalization parameter, although the intermediate feedbacks generally
are not.
\end{remark}

\section{Control of a linearized cylinder wake}
\label{sec:cylinder-wake}

We first illustrate our Newton--Kleinman framework on a linearized cylinder-wake problem with genuinely unbounded boundary control and observation. The example therefore requires a formulation that remains meaningful for unbounded input and output operators already at the continuous PDE level.

We formulate the Oseen problem as a well-posed system node and then present an adaptive Taylor--Hood discretization with one Newton--Kleinman update per spatial level. The numerical results illustrate the convergence behavior of the adaptive
iteration and the resulting closed-loop wake suppression.

\subsection{Model and system-node realization}
\label{subsec:cylinder-model}

We consider the two-dimensional channel
\[
\Omega
=
(0,L)\times(0,H)\setminus\overline{B_R(c)},
\qquad
L=2.2,\quad H=0.41,
\]
containing a circular cylinder with
\[
R=0.05,
\qquad
c=(0.2,0.2),
\qquad
D=2R=0.1.
\]
The boundary is decomposed as
\[
\partial\Omega
=
\overline{\Gamma_{\rm in}}
\cup
\overline{\Gamma_{\rm w}}
\cup
\overline{\Gamma_{\rm out}}
\cup
\overline{\Gamma_{\rm cyl}},
\]
where $\Gamma_{\rm in}$ and $\Gamma_{\rm out}$ denote the inlet and
outlet, $\Gamma_{\rm w}$ consists of the upper and lower channel walls,
and $\Gamma_{\rm cyl}$ is the cylinder boundary. On each smooth boundary
segment, $n=(n_1,n_2)^\top$ denotes the outward unit normal of the fluid
domain, and we fix the unit tangent by
\[
\tau=\spvek{\phantom{-}n_2}{-n_1}.
\]

\begin{figure}[t]
\centering
\IfFileExists{figures/cylinder_wake_setup.pdf}{
  \begin{overpic}[width=.94\textwidth]{figures/cylinder_wake_setup.pdf}
    \put(46,61){\color{red}\Large $u_1$}
    \put(46,37){\color{blue}\Large $u_2$}
  \end{overpic}
}{
  \IfFileExists{cylwake_initial.png}{
    \begin{overpic}[width=.94\textwidth]{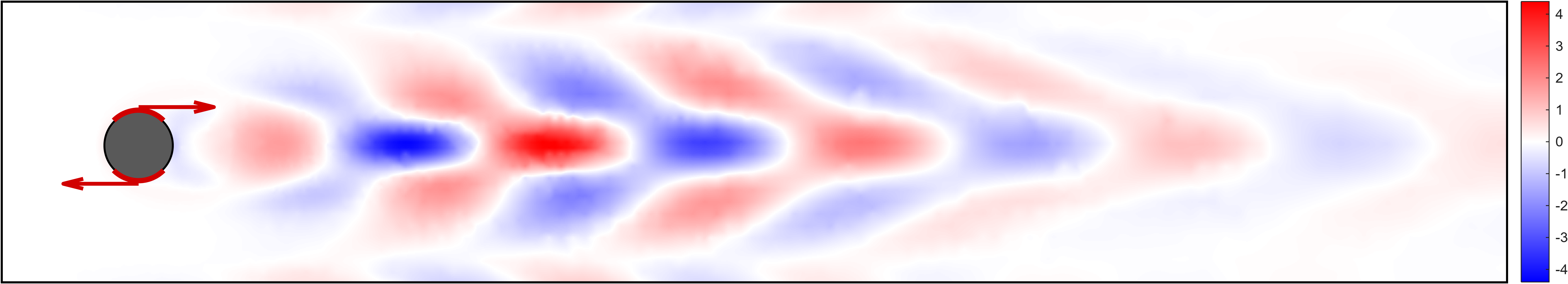}
      \put(10,12.5){\color{red}$u_1$}
      \put(4,4){\color{red}$u_2$}
    \end{overpic}
  }{
    \fbox{
      \parbox[c][4.8cm][c]{.9\textwidth}{
        \centering
        \textbf{Cylinder-wake setup}\\[0.5em]
        Placeholder for the channel geometry, stationary base flow,
        tangential cylinder actuator, and the two wake-output regions.
      }
    }
  }
}
\caption{Cylinder-wake configuration. The scalar control acts through a
localized tangential traction on the cylinder, while the two performance
outputs measure vorticity moments in the downstream wake.}
\label{fig:cylinder-setup}
\end{figure}

The mean inflow velocity is normalized to one. Throughout the numerical
experiment we use
\[
\mathrm{Re}_D
=
\tfrac{D}{\nu}
=
50,
\qquad
\nu=2\cdot10^{-3}.
\]
Let $(\bar v,\bar p)$ denote the corresponding stationary
Navier--Stokes solution. With
\[
\varepsilon(v)
=
\tfrac12\bigl(\nabla v+\nabla v^\top\bigr),
\qquad
\sigma(v,p)
=
2\nu\varepsilon(v)-pI,
\]
the stationary equations are
\[
\begin{aligned}
-\operatorname{div}\sigma(\bar v,\bar p)
+
(\bar v\cdot\nabla)\bar v
&=0
&&\text{in }\Omega,
\\
\operatorname{div}\bar v
&=0
&&\text{in }\Omega.
\end{aligned}
\]
At the inlet we prescribe the parabolic velocity profile with unit mean
velocity,
\[
\bar v
=
\tfrac{6x_2(H-x_2)}{H^2}e_1
\qquad\text{on }\Gamma_{\rm in},
\]
and impose
$\bar v=0$ on $\Gamma_{\rm w}$.
At the outlet we use the natural stress condition
\[
\sigma(\bar v,\bar p)n=0
\qquad\text{on }\Gamma_{\rm out}.
\]
The cylinder is impermeable and satisfies the tangential Navier condition
\[
\bar v\cdot n=0,
\qquad
\bigl(\sigma(\bar v,\bar p)n\bigr)\cdot\tau
+
\beta_{\rm R}\,\bar v\cdot\tau
=
0
\qquad\text{on }\Gamma_{\rm cyl},
\]
where
\[
\ell_{\rm slip}=0.02D,
\qquad
\beta_{\rm R}
=
\tfrac{\nu}{\ell_{\rm slip}}
=
1.
\]
The controlled dynamics are obtained by linearization about the stationary
state $(\bar v,\bar p)$. Denoting the velocity and pressure perturbations by
$v$ and $p$, respectively, we obtain
\begin{equation}
\begin{aligned}
\dot v
-
\operatorname{div}\sigma(v,p)
+
(\bar v\cdot\nabla)v
+
(v\cdot\nabla)\bar v
&=0
&&\text{in }\Omega,
\\
\operatorname{div}v
&=0
&&\text{in }\Omega.
\end{aligned}
\label{eq:cylinder-linearized}
\end{equation}
The perturbation satisfies
\[
\begin{aligned}
v&=0
&&\text{on }\Gamma_{\rm in}\cup\Gamma_{\rm w},
\\
\sigma(v,p)n&=0
&&\text{on }\Gamma_{\rm out},
\\
v\cdot n&=0
&&\text{on }\Gamma_{\rm cyl}.
\end{aligned}
\]
The control is a scalar tangential traction on the cylinder. We parametrize
$\Gamma_{\rm cyl}$ by the polar angle $\theta\in(-\pi,\pi]$. Since the
outward normal of the fluid domain points into the cylinder,
\[
n=-\spvek{\cos\theta}{\sin\theta},
\qquad
\tau=\spvek{-\sin\theta}{\phantom{-}\cos\theta}
\qquad\text{on }\Gamma_{\rm cyl}.
\]
We define the raised-cosine patch
\[
\psi(\theta;\theta_0)
=
\begin{cases}
\displaystyle
\tfrac12\left(
1+\cos\bigl(4\,d(\theta,\theta_0)\bigr)
\right),
&
|d(\theta,\theta_0)|\leq\frac{\pi}{4},
\\[1ex]
0,
&
|d(\theta,\theta_0)|>\frac{\pi}{4},
\end{cases}
\]
where
\[
d(\theta,\theta_0)
=
\operatorname{atan2}
\bigl(
\sin(\theta-\theta_0),
\cos(\theta-\theta_0)
\bigr)
\]
is the signed periodic angular distance. The actuator profile is
\[
g_{\rm u}(\theta)
=
\sqrt{\tfrac{8}{3\pi R}}
\left[
\psi(\theta;\tfrac{\pi}{2})
+
\psi(\theta;-\tfrac{\pi}{2})
\right].
\]
Thus the scalar actuator consists of two equal localized tangential
traction patches on the upper and lower parts of the cylinder. Since the
two supports are disjoint,
\[
\|g_{\rm u}\|_{L^2(\Gamma_{\rm cyl})}=1.
\]
For a physical traction amplitude $u\in\mathbb C$, we impose
\[
\bigl(\sigma(v,p)n\bigr)\cdot\tau
+
\beta_{\rm R}\,v\cdot\tau
=
g_{\rm u}u
\qquad\text{on }\Gamma_{\rm cyl}.
\]
The performance output consists of two vorticity moments,
\[
y_j
=
C_jv
:=
\int_\Omega q_j\,\operatorname{curl}v\,dx,
\qquad
j=1,2,
\]
The weight functions $q_1,q_2$ are constructed once from the dominant
open-loop wake eigenmode on the initial spatial mesh. Let $u_{\rm w}$ denote
the velocity component of this complex eigenmode, and let $S_h$ denote the
continuous scalar quadratic finite-element space on the initial mesh. We
compute the $L^2(\Omega)$-projections
$\widetilde q_R,\widetilde q_I\in S_h$ of
$\operatorname{curl}(\Ree u_{\rm w})$ and
$\operatorname{curl}(\operatorname{Im}u_{\rm w})$,
respectively. In the present computation these two functions are linearly
independent. The functions $q_1,q_2$ are then chosen as an
$L^2(\Omega)$-orthonormal basis of
$\operatorname{span}\{\widetilde q_R,\widetilde q_I\}$.
Thus
\[
\operatorname{span}\{q_1,q_2\}
=
\operatorname{span}\{\widetilde q_R,\widetilde q_I\},
\qquad
\langle q_i,q_j\rangle_{L^2(\Omega)}=\delta_{ij}.
\]
After this construction, $q_1$ and $q_2$ are kept fixed: on subsequently
refined meshes they are transferred to the new scalar quadratic
finite-element space and are not recomputed from updated flow or spectral
data.

For the two weight functions used in the computation, the traces on the
cylinder are nonzero. In the final reproduction run we obtain
\[
\|\gamma q_1\|_{L^2(\Gamma_{\rm cyl})}
=
4.518437548099\cdot10^{-1},
\qquad
\|\gamma q_2\|_{L^2(\Gamma_{\rm cyl})}
=
1.749981765199\cdot10^{-1}.
\]
Thus both outputs satisfy the nonvanishing trace condition used below to
establish that the corresponding observation operators are genuinely
unbounded on the kinetic-energy state space.

We write
$y
=
\spvek{y_1}{y_2}
\in Y:=\mathbb C^2$.
The physical LQ cost is
\[
J(u,y)
=
\int_0^\infty
\left(
\|y(t)\|_{\mathbb C^2}^2
+
r_{\rm u}|u(t)|^2
\right)\,dt,
\qquad
r_{\rm u}
=
1.871174597508\cdot10^5.
\]
We introduce the scaled input
$w=\sqrt{r_{\rm u}}\,u$
and take $U:=\mathbb C$. Then
\[
J(w,y)
=
\int_0^\infty
\left(
\|y(t)\|_{\mathbb C^2}^2
+
|w(t)|^2
\right)\,dt,
\]
while the boundary condition becomes
\[
\bigl(\sigma(v,p)n\bigr)\cdot\tau
+
\beta_{\rm R}\,v\cdot\tau
=
r_{\rm u}^{-1/2}g_{\rm u}w
\qquad\text{on }\Gamma_{\rm cyl}.
\]
We use the kinetic-energy state space
\[
V_\sigma
=
\left\{
v\in H^1(\Omega;\mathbb C^2)
\;\middle|\;
\begin{array}{l}
\operatorname{div}v=0,\\
\gamma v=0
\text{ on }\Gamma_{\rm in}\cup\Gamma_{\rm w},\\
\gamma v\cdot n=0
\text{ on }\Gamma_{\rm cyl}
\end{array}
\right\},
\qquad
X
=
\overline{V_\sigma}^{\,L^2(\Omega;\mathbb C^2)}.
\]

For $\spvek vw\in X\times U$, we put
$\spvek vw\in\dom(A\&B)$ if $v\in V_\sigma$ and there exists
$p\in L^2(\Omega)$ such that
\[
\sigma(v,p)
\in
H(\operatorname{div},\Omega;\mathbb C^{2\times2}),
\qquad f
:=
\operatorname{div}\sigma(v,p)
-
(\bar v\cdot\nabla)v
-
(v\cdot\nabla)\bar v
\in X,
\]
and
$\sigma(v,p)n=0$ on $\Gamma_{\rm out}$,
together with
\[
\bigl(\sigma(v,p)n\bigr)\cdot\tau
+
\beta_{\rm R}\,v\cdot\tau
=
r_{\rm u}^{-1/2}g_{\rm u}w
\qquad\text{on }\Gamma_{\rm cyl},
\]
in the corresponding weak trace sense. We define
\begin{equation}
A\&B\spvek vw
=
f,
\qquad
C\&D\spvek vw
=
\spvek{C_1v}{C_2v}.
\label{eq:cylinder-node}
\end{equation}
Hence
\[
\mathcal S_{\rm cw}
=
\begin{bmatrix}
A\&B\\
C\&D
\end{bmatrix}
:
\dom(A\&B)
\subset
X\times\mathbb C
\rightarrow
X\times\mathbb C^2.
\]

\begin{theorem}[System-node realization]
\label{thm:cylinder-analytic-node}
Assume that
$\bar v\in W^{1,\infty}(\Omega;\mathbb R^2)$,
let $g_{\rm u}\in L^2(\Gamma_{\rm cyl})$, and let
$q_1,q_2\in H^1(\Omega)$ be real valued.

Then $\mathcal S_{\rm cw}$ defined in \eqref{eq:cylinder-node} is a
well-posed system node on $(X,\mathbb C,\mathbb C^2)$. Its main operator
$A$ generates an analytic $C_0$-semigroup on $X$ and has compact
resolvent.

If $g_{\rm u}\neq0$, then the boundary control is genuinely unbounded,
$B\in\mathcal L(\mathbb C,X_{-1})$,
$B\notin\mathcal L(\mathbb C,X)$.

Assume, in addition, that
$\gamma q_j|_{\Gamma_{\rm cyl}}\not\equiv0$,
$j=1,2$.
Then both wake observations are genuinely unbounded on the kinetic-energy
state space,
\[
C_j\in\mathcal L(\dom(A),\mathbb C),
\qquad
C_j\notin\mathcal L(X,\mathbb C),
\qquad
j=1,2.
\]
\end{theorem}

\begin{proof}
On $V_\sigma$ consider the sesquilinear form
\begin{equation}
\begin{aligned}
a(v,\varphi)
={}&
2\nu
\langle\varepsilon(\varphi),\varepsilon(v)\rangle_{L^2(\Omega)}
+
\beta_{\rm R}
\langle \varphi\cdot\tau,v\cdot\tau\rangle_{L^2(\Gamma_{\rm cyl})}
\\
&+
\langle\varphi,(\bar v\cdot\nabla)v\rangle_{L^2(\Omega)}
+
\langle\varphi,(v\cdot\nabla)\bar v\rangle_{L^2(\Omega)} .
\end{aligned}
\label{eq:cylinder-form}
\end{equation}
The form is bounded on $V_\sigma\times V_\sigma$. The trace theorem \cite[Thm.~7.39]{AdamsFournier2003} gives the required
estimate for the Robin term, while
$\bar v\in W^{1,\infty}(\Omega;\mathbb R^2)$ yields
\[
\|(\bar v\cdot\nabla)v\|_{L^2}
\leq
\|\bar v\|_{L^\infty}\|\nabla v\|_{L^2},\quad
\|(v\cdot\nabla)\bar v\|_{L^2}
\leq
\|\nabla\bar v\|_{L^\infty}\|v\|_{L^2}.
\]
By Korn's inequality \cite[Thm.~6.3-3]{Ciarlet1988} and the trace theorem
\cite[Thm.~7.39]{AdamsFournier2003}, the lower-order terms are relatively
bounded with respect to the $H^1(\Omega)$-norm. Applying Young's inequality
to these terms therefore yields $c_1>0$ and $c_0\geq0$ such that
\[
\Ree a(v,v)
\geq
c_1\|v\|_{H^1(\Omega)}^2
-
c_0\|v\|_{L^2(\Omega)}^2,
\qquad
v\in V_\sigma.
\]
Hence $a$ is a densely defined closed sectorial form on $X$. The negative
of the operator associated with $a$ is precisely the zero-input operator
$A$. Consequently, $A$ generates an analytic $C_0$-semigroup on $X$.

Since $\Omega$ is bounded, the embedding
$V_\sigma\hookrightarrow X$
is compact. The standard compactness result for operators associated with
closed sectorial forms therefore implies that
the operator $(\lambda-A)^{-1}$ is compact
for one, and hence for every, $\lambda\in\rho(A)$. Thus $A$ has compact
resolvent.

The boundary control is represented on $V_\sigma$ by
\[
\langle\mathfrak B w,\varphi\rangle_{V_\sigma',V_\sigma}
=
r_{\rm u}^{-1/2}w
\int_{\Gamma_{\rm cyl}}
g_{\rm u}\,\overline{\varphi\cdot\tau}\,ds .
\]
The trace theorem gives
$\mathfrak B\in\mathcal L(\mathbb C,V_\sigma')$.
For every $T>0$, the variational evolution problem
\[
\langle\dot v(t),\varphi\rangle_{V_\sigma',V_\sigma}
+
a(v(t),\varphi)
=
\langle\mathfrak B w(t),\varphi\rangle_{V_\sigma',V_\sigma},
\qquad
\varphi\in V_\sigma,
\]
therefore has, for $v_0\in X$ and
$w\in L^2(0,T;\mathbb C)$, a unique solution
\[
v\in
C([0,T];X)
\cap
L^2(0,T;V_\sigma),
\qquad
\dot v\in L^2(0,T;V_\sigma').
\]
The G{\aa}rding estimate and Gronwall's inequality yield
\begin{equation}
\sup_{0\leq t\leq T}\|v(t)\|_X^2
+
\int_0^T\|v(t)\|_{H^1(\Omega)}^2\,dt
\leq
c_T
\left(
\|v_0\|_X^2
+
\|w\|_{L^2(0,T)}^2
\right).
\label{eq:cylinder-energy-estimate}
\end{equation}

For classical solutions, the pressure is recovered by the standard
De Rham argument. Green's formula then yields the stress boundary
conditions in the definition of $\dom(A\&B)$. The variational
realization is closed, and, for every prescribed $w\in\mathbb C$, a
sufficiently shifted variational problem provides an element
$\spvek vw\in\dom(A\&B)$. Hence the state-control operator satisfies
the closedness and input-surjectivity requirements of a system node.
The corresponding Oseen system-node construction is detailed in
\cite[Sec.~4]{ReisSchallerOseen}.

The shifted form estimate also yields
\[
\|v\|_{V_\sigma}
\leq
c
\left(
\|v\|_X
+
\left\|A\&B\spvek vw\right\|_X
+
|w|
\right),
\qquad
\spvek vw\in\dom(A\&B).
\]
For the observations,
\[
|C_jv|
\leq
\|q_j\|_{L^2(\Omega)}
\|\operatorname{curl}v\|_{L^2(\Omega)}
\leq
c_j\|v\|_{H^1(\Omega)}.
\]
Hence $C\&D$ is bounded with respect to the graph norm of $A\&B$.
Together with \eqref{eq:cylinder-energy-estimate}, this gives the
finite-time estimate
\[
\|v(T)\|_X
+
\|y\|_{L^2(0,T;\mathbb C^2)}
\leq
c_T
\left(
\|v_0\|_X
+
\|w\|_{L^2(0,T)}
\right),
\]
and therefore $\mathcal S_{\rm cw}$ is a well-posed system node.

The boundary traction functional defining $\mathfrak B$ is not continuous
with respect to the $L^2(\Omega)$ norm whenever $g_{\rm u}\neq0$.
Consequently its extrapolation-space realization does not take values in
$X$, and
$B\notin\mathcal L(\mathbb C,X)$.

Finally, since $q_j$ is real valued, integration by parts gives
\[
C_jv
=
\left\langle
\nabla^\perp q_j,v
\right\rangle_{L^2(\Omega)}
-
\int_{\partial\Omega}
q_j\,v\cdot\tau\,ds.
\]
The traces of $v$ vanish on
$\Gamma_{\rm in}\cup\Gamma_{\rm w}$, whereas on
$\Gamma_{\rm cyl}$ only the normal component is prescribed. Hence
$\gamma q_j|_{\Gamma_{\rm cyl}}\not\equiv0$ produces a nontrivial
functional of the free tangential trace.

This functional is not continuous with respect to the $L^2(\Omega)$
norm. Indeed, divergence-free boundary-layer fields supported in an
arbitrarily thin collar of $\Gamma_{\rm cyl}$ can be chosen with zero
normal trace and prescribed nonzero tangential trace, while their
$L^2(\Omega)$ norm tends to zero.

Since $A$ is the operator associated with the closed sectorial form
$a$, its domain $\dom(A)$ is dense in $V_\sigma$ with respect to the
form norm. If $C_j|_{\dom(A)}$ admitted a bounded extension to $X$,
then continuity with respect to the $V_\sigma$-norm and density of
$\dom(A)$ in $V_\sigma$ would imply that the functional above is
$L^2(\Omega)$-continuous on all of $V_\sigma$, a contradiction.
Consequently,
$C_j\notin\mathcal L(X,\mathbb C)$.
On the other hand, the preceding graph-norm estimate gives
\[
C_j\in\mathcal L(\dom(A),\mathbb C),
\qquad
j=1,2.
\]
\end{proof}
\paragraph{Open-loop stability.}
Theorem~\ref{thm:cylinder-analytic-node} reduces exponential stability
of the open-loop system to a spectral condition. Since the main operator
$A$ generates an analytic semigroup, its exponential growth bound agrees
with its spectral bound,
\begin{equation}
\omega_0(A)
=
s(A)
:=
\sup\setdef{\Ree\lambda}{\lambda\in\sigma(A)};
\label{eq:cylinder-spectral-growth}
\end{equation}
see \cite[Cor.~IV.3.12]{EngelNagel2000}. Moreover, $A$ has compact
resolvent, so its spectrum consists of isolated eigenvalues of finite
algebraic multiplicity, with no finite accumulation point. Consequently,
\[
s(A)<0
\quad\Longleftrightarrow\quad
\text{$A$ generates an exponentially stable semigroup}.
\]

The sign of $s(A)$ depends on the stationary base flow and is therefore
not implied by the general assumptions of
Theorem~\ref{thm:cylinder-analytic-node}. For the operating point
$\mathrm{Re}_D=50$ used below, the numerical spectral computations in
Subsection~\ref{subsec:cylinder-results} provide evidence that $s(A)<0$.

\begin{corollary}[Consequence of a negative spectral bound]
\label{cor:cylinder-infinite-time}
If
$s(A)<0$,
then the main semigroup is exponentially stable and
$\mathcal S_{\rm cw}$ is infinite-time well posed. In particular, $B$ and
$C$ are infinite-time admissible, and the zero feedback is a stabilizing
initialization for the Newton--Kleinman iteration.
\end{corollary}

\begin{proof}
By analyticity and \eqref{eq:cylinder-spectral-growth},
$\omega_0(A)=s(A)<0$. Hence the main semigroup is exponentially stable.
The exponential growth rate of a well-posed linear system agrees with that
of its main semigroup; see \cite[Lem.~3]{Staffans1998Coprime}.
Consequently all finite-time system maps are bounded uniformly for
$T>0$. Thus $\mathcal S_{\rm cw}$ is infinite-time well posed.
\end{proof}
\subsection{Adjoint and shifted-resolvent discretization}
\label{subsec:cylinder-adjoint}

The low-rank feedback evaluations below require shifted resolvents of the
adjoint system node. We therefore derive an explicit realization of
$\mathcal S_{\rm cw}^*$. Besides providing the adjoint equations, this
makes it possible to apply $C^*$ without representing the unbounded
observation operator as an $X$-valued functional.

With the tangent orientation fixed above, define, for $q\in H^1(\Omega)$,
\[
\nabla^\perp q
=
\spvek{\phantom{-}\partial_2q}{-\partial_1q}.
\]
For $z=\spvek{z_1}{z_2}\in\mathbb C^2$, set
\[
q_z
=
z_1q_1+z_2q_2.
\]
Since the functions $q_1$ and $q_2$ are real valued, the definition of the
performance output gives
\begin{equation}
\langle Cv,z\rangle_{\mathbb C^2}
=
\int_\Omega
q_z\,\overline{\operatorname{curl}v}\,dx
=
\left\langle
v,\nabla^\perp q_z
\right\rangle_{L^2(\Omega)}
-
\int_{\partial\Omega}
q_z\,\overline{v\cdot\tau}\,ds.
\label{eq:cylinder-output-green}
\end{equation}

Let $r$ be a smooth divergence-free test field and let $\pi$ denote the
corresponding adjoint pressure. Integration by parts in the differential
operator from \eqref{eq:cylinder-linearized} gives
\begin{align}
&
\left\langle
\operatorname{div}\sigma(v,p)
-(\bar v\cdot\nabla)v
-(v\cdot\nabla)\bar v,
r
\right\rangle_{L^2(\Omega)}
\notag\\
&\quad=
\left\langle
v,
\operatorname{div}\sigma(r,\pi)
+
(\bar v\cdot\nabla)r
-
(\nabla\bar v)^\top r
\right\rangle_{L^2(\Omega)}
\notag\\
&\qquad
+
\int_{\partial\Omega}
\left(
\overline{\sigma(v,p)n}\cdot r
-
\overline v\cdot
\bigl(
\sigma(r,\pi)n+(\bar v\cdot n)r
\bigr)
\right)\,ds .
\label{eq:cylinder-adjoint-green}
\end{align}
The term involving $\bar v\cdot n$ is essential here, since the stationary
flow has nonzero normal velocity at the inlet and outlet.

Combining \eqref{eq:cylinder-output-green} and
\eqref{eq:cylinder-adjoint-green} determines the adjoint boundary
conditions. On $\Gamma_{\rm in}\cup\Gamma_{\rm w}$ the primal velocity
vanishes while its traction is free, and hence
$r=0$ on $\Gamma_{\rm in}\cup\Gamma_{\rm w}$.
On the outlet, the primal condition
$\sigma(v,p)n=0$ gives
\begin{equation}
\sigma(r,\pi)n
+
(\bar v\cdot n)r
=
-q_z\tau
\qquad
\text{on }\Gamma_{\rm out}.
\label{eq:cylinder-adjoint-outlet}
\end{equation}
Finally, on $\Gamma_{\rm cyl}$ one has
$\bar v\cdot n=0$,
$v\cdot n=0$, and
\[
\bigl(\sigma(v,p)n\bigr)\cdot\tau
+
\beta_{\rm R}v\cdot\tau
=
r_{\rm u}^{-1/2}g_{\rm u}w.
\]
Since the normal traction of the primal problem is otherwise free, the
adjoint state satisfies
$r\cdot n=0$ on $\Gamma_{\rm cyl}$.
The tangential part of the boundary identity then yields
\begin{equation}
\bigl(\sigma(r,\pi)n\bigr)\cdot\tau
+
\beta_{\rm R}r\cdot\tau
=
-q_z
\qquad
\text{on }\Gamma_{\rm cyl}.
\label{eq:cylinder-adjoint-cylinder}
\end{equation}
The remaining cylinder term is
\[
\overline w\,
r_{\rm u}^{-1/2}
\int_{\Gamma_{\rm cyl}}
g_{\rm u}\,r\cdot\tau\,ds,
\]
and therefore determines the output of the adjoint node.

More precisely, for
$\spvek rz\in X\times\mathbb C^2$ we put
$\spvek rz\in\dom(\mathcal S_{\rm cw}^*)$ if $r\in V_\sigma$ and there
exists a pressure $\pi\in L^2(\Omega)$ such that
\begin{align*}
&\sigma(r,\pi)
\in
H(\operatorname{div},\Omega;\mathbb C^{2\times2}),\\
&f^*
:=
\operatorname{div}\sigma(r,\pi)
+
(\bar v\cdot\nabla)r
-
(\nabla\bar v)^\top r
+
\nabla^\perp q_z
\in X,
\end{align*}
and the boundary conditions
\eqref{eq:cylinder-adjoint-outlet} and
\eqref{eq:cylinder-adjoint-cylinder} hold in the corresponding weak trace
spaces. We then define
\begin{equation}
\mathcal S_{\rm cw}^*
\pvek rz
=
\spvek{
\displaystyle
\operatorname{div}\sigma(r,\pi)
+
(\bar v\cdot\nabla)r
-
(\nabla\bar v)^\top r
+
\nabla^\perp q_z
}{
\displaystyle
r_{\rm u}^{-1/2}
\int_{\Gamma_{\rm cyl}}
g_{\rm u}\,r\cdot\tau\,ds
}.
\label{eq:cylinder-adjoint-node}
\end{equation}

Indeed, for
$\spvek vw\in\dom(\mathcal S_{\rm cw})$ and
$\spvek rz\in\dom(\mathcal S_{\rm cw}^*)$,
the preceding Green identities and boundary conditions give
\[
\left\langle
(A\&B)\spvek vw,r
\right\rangle_X
+
\left\langle
(C\&D)\spvek vw,z
\right\rangle_{\mathbb C^2}
=
\left\langle
v,
(A^*\&C^*)\spvek rz
\right\rangle_X
+
\left\langle
w,
(B^*\&D^*)\spvek rz
\right\rangle_{\mathbb C}.
\]
Thus \eqref{eq:cylinder-adjoint-node} is precisely the Hilbert-space
adjoint of the system node from \eqref{eq:cylinder-node}. By
\cite[Lem.~6.2.14]{Staffans2005}, the adjoint of a system node is again a
system node.

In particular, the action of $C^*$ required in the feedback evaluation is
realized directly through the volume term
$\nabla^\perp q_z$ together with the boundary data
$-q_z\tau$ on $\Gamma_{\rm out}$ and $-q_z$ in the tangential Robin condition
on $\Gamma_{\rm cyl}$. No $X$-valued Riesz representative of the
unbounded wake observation is formed.

\paragraph{Taylor--Hood discretization and shifted resolvents.}

The numerical realization is based directly on shifted resolvent problems
for the primal and adjoint system nodes. Spatial discretization is therefore
introduced only after the continuous system-node and adjoint constructions
of the preceding subsections. In particular, we do not form a discrete
divergence-free basis and do not assemble a finite-dimensional realization
of the generator.

Let $\mathcal T_h$ be a conforming triangulation of $\Omega$. We use the
Taylor--Hood pair
\[
V_h\subset H^1(\Omega;\mathbb C^2),
\qquad
Q_h\subset L^2(\Omega),
\]
consisting of continuous piecewise quadratic velocities and continuous
piecewise affine pressures. The essential velocity conditions
\[
v_h=0
\quad\text{on }\Gamma_{\rm in}\cup\Gamma_{\rm w},
\qquad
v_h\cdot n=0
\quad\text{on }\Gamma_{\rm cyl}
\]
are incorporated into $V_h$. The outlet condition and the tangential
Robin condition on the cylinder remain natural boundary conditions. Since
the outlet traction fixes the pressure level, no additional zero-mean
constraint is imposed on $Q_h$.

We use the restriction of the Oseen form
\eqref{eq:cylinder-form} to $V_h\times V_h$ and define the discrete
divergence pairing by
\[
b(v_h,q_h)
=
\int_\Omega
\overline{q_h}\,\operatorname{div}v_h\,dx.
\]
Let $M_h$, $L_h$, and $G_h$ denote the matrices representing the velocity
mass form, the Oseen form, and the divergence pairing, respectively.

The scaled cylinder traction defines the discrete input functional
$b_h\in V_h'$ by
\[
\langle b_hw,\varphi_h\rangle
=
r_{\rm u}^{-1/2}w
\int_{\Gamma_{\rm cyl}}
g_{\rm u}\,\overline{\varphi_h\cdot\tau}\,ds,
\qquad
\varphi_h\in V_h.
\]
Here $b_h$ represents the boundary functional in the finite-element dual
space; it is not interpreted as an $X$-valued approximation of the
unbounded continuous control operator.

Let $\lambda\in\rho(A)$. For $f\in X$ and $w\in\mathbb C$, the Galerkin
approximation of
\[
(\lambda-A_{-1})^{-1}(f+Bw)
\]
is obtained by finding $(v_h,p_h)\in V_h\times Q_h$ such that, for all $(\varphi_h,q_h)\in V_h\times Q_h$,
\[
\begin{aligned}
\lambda
\langle \varphi_h,v_h\rangle_{L^2(\Omega)}
+
a(v_h,\varphi_h)
-
\int_\Omega
p_h\,\overline{\operatorname{div}\varphi_h}\,dx
&=
\langle \varphi_h,f\rangle_{L^2(\Omega)}
+
r_{\rm u}^{-1/2}w
\int_{\Gamma_{\rm cyl}}
g_{\rm u}\,\overline{\varphi_h\cdot\tau}\,ds,
\\
\int_\Omega
\overline{q_h}\,\operatorname{div}v_h\,dx
&=0.
\end{aligned}
\]
Writing $f_h$ for the coefficient vector of the volume functional
$\varphi_h\mapsto\langle f,\varphi_h\rangle_{L^2(\Omega)}$, the resulting
bordered system is
\[
\begin{bmatrix}
\lambda M_h+L_h & -G_h^*\\
G_h             & 0
\end{bmatrix}
\begin{bmatrix}
v_h\\
p_h
\end{bmatrix}
=
\begin{bmatrix}
f_h+b_hw\\
0
\end{bmatrix}.
\]
For later use, we write
\[
\mathbb K_h(\lambda)
=
\begin{bmatrix}
\lambda M_h+L_h & -G_h^*\\
G_h             & 0
\end{bmatrix}.
\]
In particular, the discrete input lift required by the system-node
resolvent formulas is obtained from
\begin{equation}
\mathbb K_h(\lambda)
\begin{bmatrix}
v_h^{\rm in}\\
p_h^{\rm in}
\end{bmatrix}
=
\begin{bmatrix}
b_hw\\
0
\end{bmatrix}.
\label{eq:cylinder-input-lift}
\end{equation}
The pressure is retained only as a Lagrange multiplier enforcing discrete
incompressibility and is not part of the state space used by the
Newton--Kleinman iteration.

The two wake outputs are discretized in exactly the form in which they
occur in the continuous system node. We define
\[
C_h:V_h\to\mathbb C^2,
\qquad
(C_hv_h)_j
=
\int_\Omega
q_j\,\operatorname{curl}v_h\,dx,
\qquad
j=1,2.
\]
No integration by parts is used in the assembly of $C_h$. Thus the
finite-element output is the direct Galerkin realization of the
continuous vorticity observation rather than an approximation obtained
from an $X$-Riesz representative.

The adjoint resolvent is treated in the corresponding mixed formulation.
Let $z\in\mathbb C^2$. The Galerkin approximation of the adjoint-node
input resolvent corresponding to the input $z$ is obtained from
\begin{equation}
\begin{bmatrix}
\lambda M_h+L_h^* & -G_h^*\\
G_h               & 0
\end{bmatrix}
\begin{bmatrix}
r_h\\
\pi_h
\end{bmatrix}
=
\begin{bmatrix}
C_h^*z\\
0
\end{bmatrix}.
\label{eq:cylinder-adjoint-shifted}
\end{equation}
Here $C_h^*z$ is the coefficient vector of the corresponding
finite-element dual functional and is not interpreted as an
$L^2$-Riesz representative of the continuous observation adjoint.

More generally, if $g_h$ is the coefficient vector of an additional
state-space dual right-hand side, one solves
\begin{equation}
\begin{bmatrix}
\lambda M_h+L_h^* & -G_h^*\\
G_h               & 0
\end{bmatrix}
\begin{bmatrix}
r_h\\
\pi_h
\end{bmatrix}
=
\begin{bmatrix}
g_h+C_h^*z\\
0
\end{bmatrix}.
\label{eq:cylinder-adjoint-general-shifted}
\end{equation}
Equivalently, this dual functional can be assembled directly from the
explicit adjoint system node derived above. For
$q_z=z_1q_1+z_2q_2$,
the observation contribution consists of the volume forcing
$\nabla^\perp q_z$, the outlet traction $-q_z\tau$, and the tangential
cylinder datum $-q_z$.

All resolvent actions required below can therefore be reduced to sparse
systems with the corresponding primal or adjoint Taylor--Hood
saddle-point matrices. The initial feedback
evaluation requires adjoint solves of the form
\eqref{eq:cylinder-adjoint-shifted}, while subsequent ADI steps require
\eqref{eq:cylinder-adjoint-general-shifted}. Primal input lifts of the form
\eqref{eq:cylinder-input-lift} enter the feedback and
flow-inversion formulas.

Consequently, neither $A_h$ nor $A_h^*$ is explicitly assembled as an
operator on a discrete divergence-free state space. The numerical
realization uses only shifted Taylor--Hood saddle-point solves and
finite-dimensional operations on the input and output spaces. This
preserves the order of operations of the continuous theory,
\[
\text{system node}
\;\longrightarrow\;
\text{shifted resolvent problem}
\;\longrightarrow\;
\text{Taylor--Hood discretization}.
\]
\subsection{Adaptive Newton--Kleinman realization}
\label{subsec:cylinder-adaptive-nk}

By Corollary~\ref{cor:cylinder-infinite-time}, the zero feedback is a
stabilizing initialization whenever $s(A)<0$. For
$\mathrm{Re}_D=50$, the numerical spectral computations in
Subsection~\ref{subsec:cylinder-results} provide evidence for this
condition. We therefore use
$\Sigma_0=\mathcal S_{\rm cw}$
as the initialization of the numerical iteration.

For the current system node $\Sigma_k$, associated with a stabilizing
feedback and with auxiliary input $v_k\in\mathbb C$, the performance output is
\[
z_k=\spvek{y}{w}\in\mathbb C^3.
\]
We use direct feedback evaluation throughout:
\begin{equation}
P_{k+1}
=
\mathcal C_{k,\infty}^*
\mathcal C_{k,\infty},
\label{eq:cylinder-direct-policy}
\end{equation}
where $\mathcal C_{k,\infty}$ is the infinite-time zero-input observation
map of $\Sigma_k$. The resulting cost operator is then used to construct
the next resolvent-normalized Newton factor and the updated feedback
relation.

\paragraph{Low-rank feedback evaluation and update.}
The feedback-cost Gramian in \eqref{eq:cylinder-direct-policy} is approximated by the
low-rank ADI iteration from
Subsection~\ref{subsec:nk-computational}. The required adjoint actions are
evaluated through shifted resolvent problems rather than through an
$X$-valued realization of the unbounded observation operator. The corresponding cost
operator is stored in low-rank form,
\begin{equation}
P_{k+1}
\approx
Z_{P,k+1}Z_{P,k+1}^*.
\label{eq:cylinder-low-rank-policy}
\end{equation}
The low-rank factor is understood with respect to the discrete state
inner product induced by the velocity mass matrix.

Complex shifts are processed together with their conjugates.

For $k=0$, the shifted resolvents are the Taylor--Hood problems derived
above. At later steps, the closed-loop resolvents are obtained recursively
from the same shifted Taylor--Hood solves by
\eqref{eq:nk-computational-resolvent-update}. Since the auxiliary input is scalar, the inversion of
$\mathbf N_{k+1}(\lambda)$ is scalar.

The next Newton factor
$\Gamma_{k+1}\&\Lambda_{k+1}$ is obtained from the linearized Lur'e
identity \eqref{eq:nk-linearized-lure}, using the fixed normalization
parameter
\[
\alpha=1.
\]
The condition $v_{k+1}=0$ defines the updated feedback; the corresponding
system node is $\Sigma_{k+1}$. Thus neither the
closed-loop generator $A_k$ nor its adjoint is assembled explicitly; all
successive feedback evaluations remain based on shifted primal and adjoint
Taylor--Hood solves.

\paragraph{ADI shifts and stopping.}
For every direct feedback evaluation, ADI shifts are selected on the current spatial level by a Penzl-type
spectral heuristic based on the available shifted-resolvent information; see
\cite{Penzl2000,BennerMenaSaak2008}. Thus no shift
sequence is transferred from a preceding mesh or Newton step. Nonreal shifts
are used together with their complex conjugates, and the resulting finite
shift sequence is repeated cyclically. These ADI shifts are distinct from
the fixed sampling points $\Lambda_{\rm sp}$ used below for spatial error
estimation.

For the initial zero-feedback evaluation $P_1$ on $\mathcal M_0$, we use
the resolvent-trans\-formed ADI residual at the fixed reference point
$\beta=1$. If $R_j^{\rm ADI}$ denotes the formal ADI residual factor after
the $j$th complete shift cycle, set
\[
R_{\beta,j}
=
(\beta-A_0^*)^{-1}R_j^{\rm ADI}
\]
and
\[
\rho_{\beta,j}
=
\frac{\|R_{\beta,j}\|_{M_h,\rm HS}}
     {\|R_{\beta,0}\|_{M_h,\rm HS}},
\qquad
\|R\|_{M_h,\rm HS}^2
=
\operatorname{trace}(R^*M_hR).
\]
The initial evaluation is terminated after a complete shift cycle once
\[
\rho_{\beta,j}\leq10^{-7}.
\]

For all subsequent feedback evaluations, $\rho_{\beta,j}$ is retained only
as a diagnostic quantity. Stopping is instead based on the stabilization
of the low-rank cost-operator approximation together with an independent
check of the prospective Newton update. Let $P_{k+1}^{(j)}$ denote the
low-rank approximation after the $j$th complete shift cycle. On a fixed
family of normalized state probes $x_\ell$, we monitor the relative changes
of the actions $P_{k+1}^{(j)}x_\ell$ and of the quadratic values
\[
\left\langle
P_{k+1}^{(j)}x_\ell,x_\ell
\right\rangle.
\]
We denote the corresponding maximal relative changes by
$\delta_{{\rm act},j}$ and $\delta_{{\rm val},j}$, respectively. In
addition, $\delta_{{\rm L},j}$ denotes the maximal relative defect of the
linearized Lur'e identity for the prospective next Newton factor, evaluated
on a fixed family of six independent graph probes.

A feedback evaluation with $k\geq1$ is terminated only after a complete
shift cycle for which all three conditions
\[
\delta_{{\rm act},j}\leq2\cdot10^{-6},
\qquad
\delta_{{\rm val},j}\leq2\cdot10^{-8},
\qquad
\delta_{{\rm L},j}\leq5\cdot10^{-9}
\]
are satisfied. The resulting low-rank factor
$Z_{P,k+1}$ from \eqref{eq:cylinder-low-rank-policy} then provides the
cost-operator approximation used in the Newton-factor construction.

\paragraph{Coupled Newton and mesh iteration.}
The computation starts on a deliberately coarse spatial level
$\mathcal M_0$. On every level $\mathcal M_k$ we perform exactly one
complete Newton--Kleinman update: the current feedback cost operator is
evaluated, the next Newton factor is constructed, and the system node
$\Sigma_{k+1}$ associated with the updated feedback is obtained. Only
afterwards is the spatial accuracy assessed. If refinement
is required, local refinement produces $\mathcal M_{k+1}$ and the next
Newton step is performed there.

Thus the LR--ADI iteration resolves the feedback cost operator on the current
spatial level, whereas the outer adaptive iteration controls the spatial
resolution. After the update
$\Sigma_k\to\Sigma_{k+1}$ has been completed on $\mathcal M_k$, the
spatial indicator $\eta_{{\rm space},k}$ is evaluated for this updated
system node. The same feedback relation defining $\Sigma_{k+1}$ is
transferred to the temporary uniform child $\mathcal M_k^+$, and the
coarse- and fine-grid resolvent quantities are compared there. Hence the
two-grid estimator measures spatial error at the newly computed feedback,
rather than at the preceding feedback $\Sigma_k$. This avoids solving
several Newton steps on an underresolved mesh or resolving the PDE to final
accuracy before the feedback iteration has converged.

\paragraph{Two-grid resolvent estimator.}
Spatial accuracy is assessed by a temporary uniformly refined child mesh
$\mathcal M_k^+$ of the current adaptive mesh $\mathcal M_k$. The child
mesh is used only for error estimation and is not used as the next
adaptive mesh.

The estimator probes the resolvent at five fixed points,
\[
\Lambda_{\rm sp}
=
\left\{
0.1,\,
1,\,
0.1+\tfrac12 i\omega_{\rm w},\,
0.1+i\omega_{\rm w},\,
0.1+\tfrac32 i\omega_{\rm w}
\right\},
\qquad
\omega_{\rm w}=16.88482141388.
\]
The wake frequency $\omega_{\rm w}$ is obtained from an open-loop
eigensolve on $\mathcal M_0$ before the adaptive Newton--Kleinman loop and
is then kept fixed. The five sampling points are therefore unchanged from
one Newton step to the next. They cover the low-frequency part of the
resolvent together with the frequency range around the dominant wake
oscillation.

At every $\lambda\in\Lambda_{\rm sp}$ we compare the resolvent quantities
of the updated system node $\Sigma_{k+1}$ on $\mathcal M_k$ and
$\mathcal M_k^+$. Since its reopened auxiliary input is scalar, there is
one primal input-resolvent solve,
\[
(\lambda-A_{k+1,-1})^{-1}B_{k+1},
\]
and, since its performance output
$z_{k+1}=\spvek{y}{w}$ has three scalar components, three adjoint
resolvent solves are required for the corresponding output directions. Thus each sample point contributes one primal and three adjoint
comparisons.

Let $v_{k+1,\lambda,h}$ and $v_{k+1,\lambda,+}$ denote the velocity
components of the primal resolvent solutions on $\mathcal M_k$ and
$\mathcal M_k^+$, respectively, and let
$r_{k+1,\lambda,h}^{(\ell)}$ and
$r_{k+1,\lambda,+}^{(\ell)}$, $\ell=1,2,3$, denote the corresponding
adjoint velocity solutions. Let $I_k^+$ denote the reference-$P_2$
interpolation from the coarse velocity space on $\mathcal M_k$ to the
corresponding space on $\mathcal M_k^+$. We define the two-grid
differences
\[
e_{k,\lambda}^{(0)}
=
v_{k+1,\lambda,+}-I_k^+v_{k+1,\lambda,h},
\]
and
\[
e_{k,\lambda}^{(\ell)}
=
r_{k+1,\lambda,+}^{(\ell)}
-
I_k^+r_{k+1,\lambda,h}^{(\ell)},
\qquad
\ell=1,2,3.
\]
For a velocity field $e$ and an element $T\in\mathcal M_k$, let
\[
\|e\|_{E(T)}^2
=
\|e\|_{L^2(T)}^2
+
\nu\|\nabla e\|_{L^2(T)}^2.
\]
Here the integrals over $T$ are evaluated by summing over its children in
$\mathcal M_k^+$. Each of the primal and adjoint comparisons is normalized
by the corresponding child-grid solution, so that quantities of different
absolute size contribute on a comparable relative scale. The resulting
local indicator is
\[
\eta_{T,k}^2
=
\frac{1}{|\Lambda_{\rm sp}|}
\sum_{\lambda\in\Lambda_{\rm sp}}
\left[
\frac{
\|e_{k,\lambda}^{(0)}\|_{E(T)}^2
}{
\|v_{k,\lambda,+}\|_{E(\Omega)}^2
}
+
\frac13
\sum_{\ell=1}^3
\frac{
\|e_{k,\lambda}^{(\ell)}\|_{E(T)}^2
}{
\|r_{k,\lambda,+}^{(\ell)}\|_{E(\Omega)}^2
}
\right].
\]
The global two-grid quantity is
\[
\eta_{{\rm space},k}
=
\left(
\sum_{T\in\mathcal M_k}
\eta_{T,k}^2
\right)^{1/2}.
\]
This estimator is deliberately based on resolvent quantities rather than
on the entries of a discrete cost matrix. It therefore tests exactly the
objects through which both the LR--ADI feedback evaluation and the
feedback update are realized.

\paragraph{Mesh refinement.}
If the spatial criterion is not satisfied, elements are selected by
D\"orfler marking. Thus a set
$\mathcal A_k\subset\mathcal M_k$ is chosen such that
\[
\sum_{T\in\mathcal A_k}\eta_{T,k}^2
\geq
\theta
\sum_{T\in\mathcal M_k}\eta_{T,k}^2.
\]
The marked elements are refined to obtain the next spatial level
$\mathcal M_{k+1}$. Throughout the adaptive computation we use
\[
\theta=0.35.
\]

The temporary uniform child $\mathcal M_k^+$ is therefore used only to
decide where refinement is required. The subsequent Newton step is carried
out on the locally refined adaptive mesh $\mathcal M_{k+1}$ rather than
on the generally much larger uniform child.

No dense cost matrix, dense feedback matrix, or closed-loop generator
is transferred between meshes. Instead, the low-rank cost-operator information
and the data representing the current feedback relation are transferred to
the refined finite-element space by the corresponding interpolation and
projection operators. The sparse resolvent realization and the shifted
resolvent problems are then assembled on the new mesh. In particular,
the spatial refinement does not alter the operator-level definition of
the Newton step.

\paragraph{Feedback-change indicator.}
On a fixed spatial level, the scalar feedback relation obtained by setting the auxiliary input to zero can be written in discrete coordinates as
\[
H_{\dot x,k}\dot x_h+H_{x,k}x_h+d_k w=0,
\]
where $\dot x_h$ denotes the discrete state derivative. We measure the change
between successive feedback relations in the dual norm induced by the
discrete kinetic-energy inner product and define, for $k\geq1$,
\[
\eta_{{\rm fb},k}
=
\frac12
\left(
\frac{
\|H_{\dot x,k+1}/d_{k+1}-H_{\dot x,k}/d_k\|_{X_h'}
}{
\|H_{\dot x,k}/d_k\|_{X_h'}
}
+
\frac{
\|H_{x,k+1}/d_{k+1}-H_{x,k}/d_k\|_{X_h'}
}{
\|H_{x,k}/d_k\|_{X_h'}
}
\right).
\]
Thus $\eta_{{\rm fb},k}$ measures the relative change of the
feedback relation itself rather than the change of a dense feedback
matrix.
\paragraph{Stopping of the coupled iteration.}
Newton convergence and spatial convergence are monitored separately. We
use the feedback tolerance
\[
\eta_{{\rm fb},k}
\leq
10^{-2}
\]
and the spatial tolerance
\[
\eta_{{\rm space},k}
\leq
5\cdot10^{-2}.
\]
The computation is terminated only if both criteria are satisfied on the
same completed adaptive level. Thus a small Newton correction does not
terminate the computation while the spatial discretization remains
underresolved, and conversely a sufficiently fine mesh does not trigger
additional Newton steps once the feedback iteration has converged.

The resulting strategy intentionally balances the two error sources.
Early Newton steps are carried out on relatively coarse meshes, while the
spatial discretization is sharpened as the feedback iteration approaches its limit.
This avoids resolving the PDE to the final accuracy before the feedback
itself is sufficiently accurate.

\subsection{Numerical results}
\label{subsec:cylinder-results}

\paragraph{Open-loop stability validation.}
We first validate numerically the spectral condition used in
Corollary~\ref{cor:cylinder-infinite-time}. The rightmost oscillatory
wake eigenvalue is computed independently on the first five
Taylor--Hood levels. The initial eigenvalue seed on $\mathcal M_0$ is
obtained from an independent broad spectral search, and the corresponding
branch is then tracked across the spatial hierarchy. The results are
shown in Table~\ref{tab:cylinder-openloop-spectrum}.

\begin{table}[t]
\centering
\caption{Mesh dependence of the tracked open-loop wake eigenvalue at
$\mathrm{Re}_D=50$.}
\label{tab:cylinder-openloop-spectrum}
\begin{tabular}{c|r|r|r}
\hline
level & velocity / pressure dofs
& $\Ree\lambda_{\rm w}$ & $\operatorname{Im}\lambda_{\rm w}$\\
\hline
$\mathcal M_0$ & $5510/785$   & $-0.07598984$ & $16.88482141$\\
$\mathcal M_1$ & $6552/922$   & $-0.04620347$ & $17.07011599$\\
$\mathcal M_2$ & $7898/1100$  & $-0.05759089$ & $17.08191403$\\
$\mathcal M_3$ & $9534/1323$  & $-0.07337304$ & $17.12327949$\\
$\mathcal M_4$ & $12378/1716$ & $-0.07038228$ & $17.12771185$\\
\hline
\end{tabular}
\end{table}

On $\mathcal M_4$, the tracked wake eigenpair has relative generalized-eigenvalue
residual $5.9\cdot10^{-13}$. To check that no other computed eigenvalue
lies further to the right, we perform an independent shift-invert screen
with
\[
\sigma=a+i\omega,
\qquad
a\in\{0.05,0.5\},
\qquad
\omega=0,3,\ldots,30.
\]
All $22$ shifted eigensolves converge. The rightmost screened eigenvalue,
$-0.07038228+17.12771185i$, agrees with the tracked wake branch to within
$2.4\cdot10^{-12}$ and has relative residual $5.5\cdot10^{-13}$; the
next rightmost candidate is the real eigenvalue $-2.08696314$.
An independent adjoint eigensolve reproduces the conjugate eigenvalue to
within $1.2\cdot10^{-12}$, with residual $7.3\cdot10^{-15}$.

These computations support the condition $s(A)<0$ for the stationary
flow at $\mathrm{Re}_D=50$ and hence the use of the zero feedback as the
initial stabilizing feedback. They constitute a numerical spectral
validation of the discretized model, rather than a computer-assisted
enclosure of the spectrum of the continuous operator.

\paragraph{Adaptive Newton--Kleinman convergence.}
Starting from the zero feedback, the initial level $\mathcal M_0$ produces
the first updated system node $\Sigma_1$. Thereafter, exactly one
Newton--Kleinman update is performed on every refined spatial level.
Table~\ref{tab:cylinder-adaptive-nk} reports the feedback and spatial
indicators from $\mathcal M_1$ onward.

\begin{table}[t]
\centering
\caption{Coupled Newton--Kleinman and spatial convergence. The second
column shows the system node obtained after the Newton update on the
corresponding spatial level.}
\label{tab:cylinder-adaptive-nk}
\begin{tabular}{c|c|r|r}
\hline
level & updated system node
& $\eta_{{\rm fb},k}$
& $\eta_{{\rm space},k}$\\
\hline
$\mathcal M_1$ & $\Sigma_2$ & $4.8545\cdot10^{-1}$ & $2.8681\cdot10^{-1}$\\
$\mathcal M_2$ & $\Sigma_3$ & $4.1097\cdot10^{-1}$ & $2.4373\cdot10^{-1}$\\
$\mathcal M_3$ & $\Sigma_4$ & $2.3043\cdot10^{-1}$ & $1.9434\cdot10^{-1}$\\
$\mathcal M_4$ & $\Sigma_5$ & $3.6701\cdot10^{-2}$ & $1.4991\cdot10^{-1}$\\
$\mathcal M_5$ & $\Sigma_6$ & $1.8705\cdot10^{-2}$ & $1.1166\cdot10^{-1}$\\
$\mathcal M_6$ & $\Sigma_7$ & $1.5424\cdot10^{-2}$ & $8.2914\cdot10^{-2}$\\
$\mathcal M_7$ & $\Sigma_8$ & $7.2318\cdot10^{-3}$ & $6.2106\cdot10^{-2}$\\
$\mathcal M_8$ & $\Sigma_9$ & $4.9068\cdot10^{-3}$ & $4.5513\cdot10^{-2}$\\
\hline
\end{tabular}
\end{table}

The feedback criterion is first satisfied on $\mathcal M_7$, whereas the spatial
indicator is still above its tolerance there. The coupled stopping
criterion therefore forces one further spatial refinement. On $\mathcal M_8$,
both
\begin{align*}
\eta_{{\rm fb},8}
&=
4.906771853226\cdot10^{-3}
<
10^{-2}\\
\eta_{{\rm space},8}
&=
4.551276149837\cdot10^{-2}
<
5\cdot10^{-2}
\end{align*}
are satisfied, and the adaptive computation terminates after constructing
$\Sigma_9$.

The final level contains $12149$ triangles and has
$48030$ velocity and $6298$ pressure degrees of freedom. Its temporary
uniform child $\mathcal M_8^+$ contains $48596$ triangles and is used
only for the final two-grid spatial assessment. The complete adaptive
computation from the open-loop initialization to the stopping decision
required $142.57$ minutes of wall-clock time.

\paragraph{Closed-loop wake suppression.}
To assess dynamically the final feedback associated with $\Sigma_9$, we use
an open-loop wake
eigenmode on $\mathcal M_8$ as the common initial perturbation for the open- and
closed-loop simulations. The perturbation is scaled to a maximum initial
speed of $10^{-1}$. 

\paragraph{Visualization.}
A visualization of the open-loop and closed-loop vorticity dynamics is
provided in the supplementary material. Starting from the same wake
perturbation, it illustrates the substantially faster decay of the wake
structure under the final feedback.

\section{Boundary-controlled elasticity}
\label{sec:numerics}

We next consider a damped linear elasticity problem on an L-shaped domain
with boundary traction control and boundary displacement observation. In
contrast to the cylinder-wake example, the performance observation is
bounded on the energy space, whereas the boundary control remains
unbounded. We again combine the Newton--Kleinman iteration with adaptive
spatial discretization, performing exactly one Newton--Kleinman update on
each production mesh.

\subsection{Model and system-node realization}
\label{subsec:num-model}

Consider the L-shaped domain
\[
\Omega
=
(-1,1)^2
\setminus
\bigl([0,1]\times[-1,0]\bigr),
\]
with
\[
\Gamma_{\rm D}
=
\setdef{(-1,s)}{-1<s<1},
\qquad
\Gamma_{\rm N}
=
\partial\Omega\setminus\overline{\Gamma_{\rm D}}.
\]

\begin{figure}[t]
\centering
\IfFileExists{elasticity_initial.png}{
  \begin{overpic}[width=.72\textwidth]{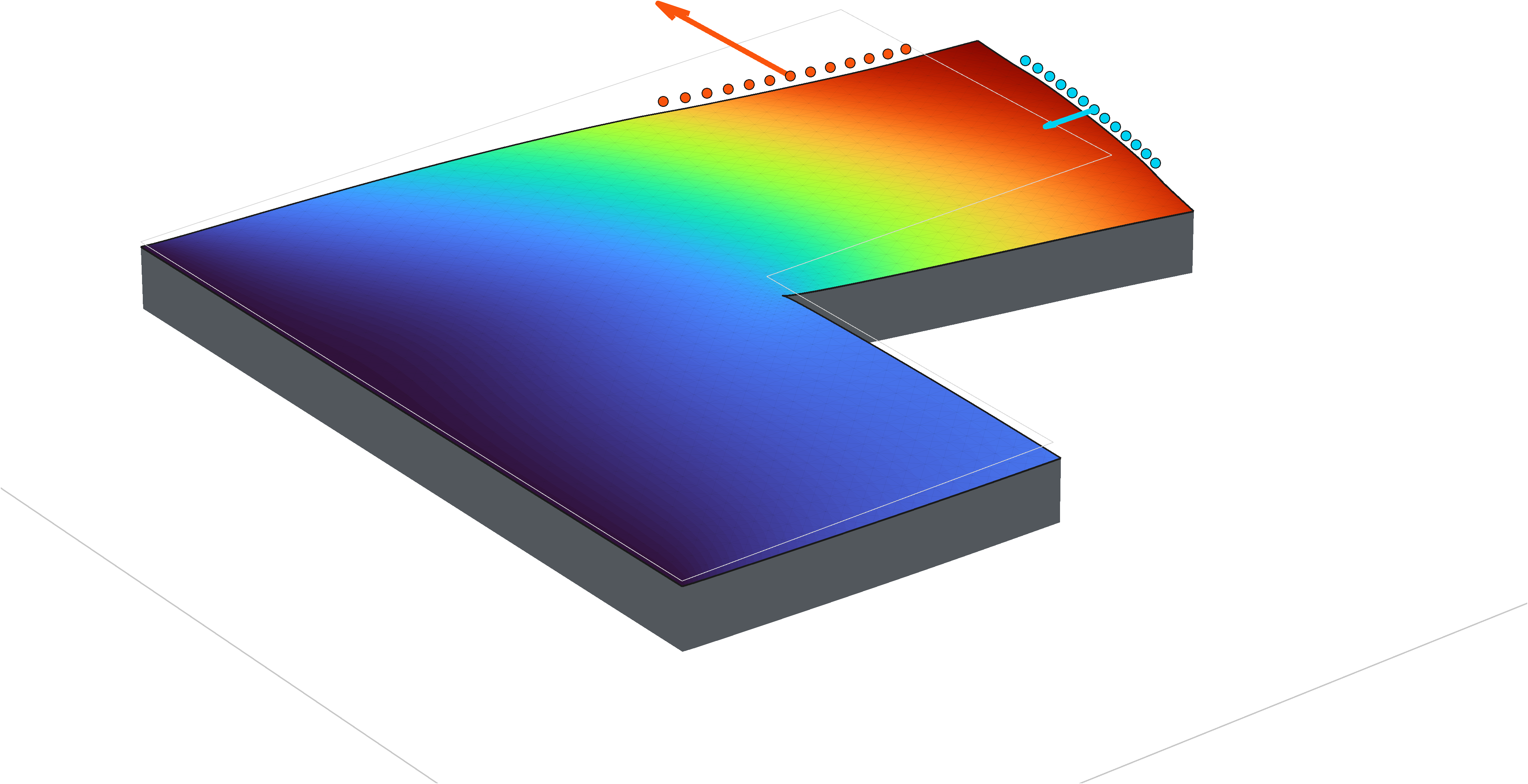}\definecolor{uone}{RGB}{0,170,170}
    \put(73,45){\color{uone}$u_1$}
    \put(44,53){\color{red}$u_2$}
  \end{overpic}
}{
  \fbox{
    \parbox[c][6cm][c]{.68\textwidth}{
      \centering
      \textbf{Boundary-controlled elasticity}\\[0.5em]
      Placeholder for the initial deformation of the L-shaped body.
    }
  }
}
\caption{Boundary-controlled elasticity configuration. The left boundary is
clamped, and two scalar traction inputs act on the indicated patches of the
right and upper boundary. The figure shows the common initial deformation
used for the open- and closed-loop transient simulations; the arrows show
the control forces generated by the final feedback at $t=0$. The deformation
is visually amplified by a factor $2.5$.}
\label{fig:elasticity-setup}
\end{figure}

Although the physical problem is real, we work with the complexification
of the state, input, and output spaces in order to remain within the
Hilbert-space framework used throughout the paper. For real initial data
and inputs, the corresponding solution and output remain real. The body is
homogeneous and isotropic, with displacement
$q:\Omega\times[0,\infty)\to\mathbb C^2$,
\[
\varepsilon(q)
=
\tfrac12\bigl(\nabla q+\nabla q^\top\bigr),
\qquad
\sigma(q)
=
2\mu\varepsilon(q)
+
\lambda_{\rm L}\operatorname{div}(q)I.
\]
Throughout the numerical experiment we use
\[
\mu=2,
\qquad
\lambda_{\rm L}=3,
\qquad
\rho=1.7,
\qquad
\eta=0.05.
\]
The body is clamped on $\Gamma_{\rm D}$ and actuated by two scalar
tractions on
\[
\Gamma_1^{\rm u}
=
\setdef{(1,s)}{0.2<s<0.8},
\qquad
\Gamma_2^{\rm u}
=
\setdef{(s,1)}{0.2<s<0.8}.
\]
Let
$g_j=|\Gamma_j^{\rm u}|^{-1/2}e_j$, $j=1,2$,
and, for $u=\spvek{u_1}{u_2}\in U:=\mathbb C^2$, set
\[
Gu
=
u_1g_1\chi_{\Gamma_1^{\rm u}}
+
u_2g_2\chi_{\Gamma_2^{\rm u}}.
\]
The dynamics are
\[
\begin{aligned}
\rho\ddot q-\operatorname{div}\sigma(q)&=0
&&\text{in }\Omega,\\
q&=0
&&\text{on }\Gamma_{\rm D},\\
\sigma(q)n+\eta\dot q&=Gu
&&\text{on }\Gamma_{\rm N}.
\end{aligned}
\]

The regulated output measures the boundary displacement over the same
boundary patches on which the traction actuators act, i.e.,
\[
y_{\rm pos}
=
G^*\gamma q
=
\pvek{
\displaystyle\int_{\Gamma_1^{\rm u}}g_1^*\gamma q\,ds
}{
\displaystyle\int_{\Gamma_2^{\rm u}}g_2^*\gamma q\,ds
}.
\]
We use the performance output
$y=2y_{\rm pos}$,
and therefore
\[
J(x_0,u)
=
\int_0^\infty
\left(
4\|y_{\rm pos}(t)\|_{\mathbb C^2}^2
+
\|u(t)\|_{\mathbb C^2}^2
\right)\,dt.
\]
With momentum density $p=\rho\dot q$, set
\[
H^1_{\Gamma_{\rm D}}(\Omega;\mathbb C^2)
=
\setdef{\phi\in H^1(\Omega;\mathbb C^2)}
{\phi|_{\Gamma_{\rm D}}=0},
\qquad
X
=
H^1_{\Gamma_{\rm D}}(\Omega;\mathbb C^2)
\times
L^2(\Omega;\mathbb C^2).
\]
By Korn's inequality \cite[Thm.~6.3-3]{Ciarlet1988},
\[
\left\|\pvek qp\right\|_X^2
=
\int_\Omega
\left(
2\mu\,\varepsilon(q):\varepsilon(q)
+
\lambda_{\rm L}|\operatorname{div}q|^2
+
\rho^{-1}|p|^2
\right)\,dx
\]
is an equivalent Hilbert-space norm on $X$; here,
$A:B$ is the trace of $A^*B$.

Let $\gamma$ denote the trace operator and let
$\gamma_n:H(\operatorname{div},\Omega;\mathbb C^{2\times2})
\to H^{-1/2}(\partial\Omega;\mathbb C^2)$
be the normal trace. For symmetric $T$,
\begin{equation}
\langle\operatorname{div}T,\varphi\rangle_{L^2(\Omega)}
+
\int_\Omega T:\varepsilon(\varphi)\,dx
=
\langle\gamma_nT,\gamma\varphi\rangle_{\partial\Omega},
\qquad
\varphi\in H^1_{\Gamma_{\rm D}}(\Omega;\mathbb C^2);
\label{eq:num-green}
\end{equation}
see \cite[Thm.~5.1 and Thm.~5.2]{GudoshnikovKrizek2025}, after
complexification.

Define
$\mathcal S_{\rm el}:
\dom(\mathcal S_{\rm el})\subset X\times U
\to X\times\mathbb C^2$
by
\begin{equation}
\begin{aligned}
\mathcal S_{\rm el}
\pvek{\spvek qp}{u}
&=
\pvek{
\displaystyle
\spvek{\rho^{-1}p}{\operatorname{div}\sigma(q)}
}{
2G^*\gamma q
},
\\[0.4em]
\dom(\mathcal S_{\rm el})
&=
\setdef{
\pvek{\spvek qp}{u}\in X\times U
}{
\begin{array}{l}
p\in H^1_{\Gamma_{\rm D}}(\Omega;\mathbb C^2),\\
\sigma(q)\in H(\operatorname{div},\Omega;\mathbb C^{2\times2}),\\
\gamma_n\sigma(q)+\rho^{-1}\eta\,\gamma p=Gu
\quad\text{on }\Gamma_{\rm N}
\end{array}
}.
\end{aligned}
\label{eq:num-node-el}
\end{equation}
Boundary equalities are understood weakly. The main operator is
\[
A\pvek qp
=
\pvek{\rho^{-1}p}{\operatorname{div}\sigma(q)}
\]
with
\[
\dom(A)
=
\setdef{
\pvek qp\in X
}{
\begin{array}{l}
p\in H^1_{\Gamma_{\rm D}}(\Omega;\mathbb C^2),\\
\sigma(q)\in H(\operatorname{div},\Omega;\mathbb C^{2\times2}),\\
\gamma_n\sigma(q)+\rho^{-1}\eta\,\gamma p=0
\quad\text{on }\Gamma_{\rm N}
\end{array}
}.
\]
To initialize the Newton--Kleinman iteration, we use the collocated
boundary-velocity feedback. Thus it suffices to verify the system-node
property of the open model and the infinite-time well-posedness and strong
stability of the corresponding closed-loop system. To this end, we
define
$C_{\rm v}\&D_{\rm v}
\in
\mathcal L\bigl(\dom(A\&B),\mathbb C^2\bigr)$
by
\[
C_{\rm v}\&D_{\rm v}
\pvek{\spvek qp}{u}
=
\rho^{-1}G^*\gamma p.
\]
For a fixed $\kappa>0$, set
$K_0\&L_0
=
-\kappa(C_{\rm v}\&D_{\rm v})$.
Introducing a new external input $v_0$, the feedback relation is written as
$v_0
=
u+\kappa\rho^{-1}G^*\gamma p$. In the numerical computation we use $\kappa=1$.

\begin{proposition}
\label{prop:num-system-node}
The operator $\mathcal S_{\rm el}$ in \eqref{eq:num-node-el} is a system
node on $(X,\mathbb C^2,\mathbb C^2)$. Moreover, $K_0\&L_0$ is an
admissible state feedback. The corresponding feedback system node is infinite-time well
posed, and its main operator $A_0$ generates a strongly stable contraction
semigroup on $X$.
\end{proposition}

\begin{proof}
Put
\[
a(q,\varphi)
=
\int_\Omega \sigma(q):\varepsilon(\varphi)\,dx.
\]
By Korn's inequality \cite[Thm.~6.3-3]{Ciarlet1988}, $a$ is coercive on $H^1_{\Gamma_{\rm D}}(\Omega;\mathbb C^2)$.

We first verify the system-node property of $\mathcal S_{\rm el}$. The state-control operator
\[
A\&B:\dom(A\&B)\subset X\times U\rightarrow X
\]
is closed. Indeed, convergence in its graph norm gives convergence of
$q$ in $H^1_{\Gamma_{\rm D}}(\Omega;\mathbb C^2)$, of $p$ in
$H^1_{\Gamma_{\rm D}}(\Omega;\mathbb C^2)$, and of
$\operatorname{div}\sigma(q)$ in $L^2(\Omega;\mathbb C^2)$.
Since the map
$q\mapsto \sigma(q)$
is bounded from
$H^1_{\Gamma_{\rm D}}(\Omega;\mathbb C^2)$ into
$L^2(\Omega;\mathbb C^{2\times2})$, it follows that
$\sigma(q)$ converges in
$H(\operatorname{div},\Omega;\mathbb C^{2\times2})$.
Continuity of the trace and normal trace then permits passage to the
limit in the weak boundary condition.

The input projection of $\dom(A\&B)$ is all of $U$. In fact, for every
$u\in U$, the variational problem
\[
a(q_u,\varphi)
=
\langle Gu,\gamma\varphi\rangle_{L^2(\Gamma_{\rm N})},
\qquad
\varphi\in H^1_{\Gamma_{\rm D}}(\Omega;\mathbb C^2),
\]
has a unique solution by Lax--Milgram; equivalently, see
\cite[Thm.~5.4 and Prop.~5.3]{GudoshnikovKrizek2025}.
The corresponding pair with $p=0$ belongs to $\dom(A\&B)$ and has input
component $u$.

The output map of $\mathcal S_{\rm el}$,
\[
C\&D\pvek{\spvek qp}{u}
=
2G^*\gamma q,
\]
is bounded already with respect to the state norm, by the trace theorem
\cite[Thm.~7.39]{AdamsFournier2003} and Korn's inequality. It remains
to verify generation of the main operator $A$.

For $x=\spvek qp\in\dom(A)$, Green's formula
\eqref{eq:num-green} and the homogeneous boundary condition give
\[
\Ree\langle Ax,x\rangle_X
=
-\eta
\|\rho^{-1}\gamma p\|_{L^2(\Gamma_{\rm N})}^2
\leq0.
\]
For $\alpha>0$, eliminating $p=\rho(\alpha q-f)$ from
\[
(\alpha I-A)\spvek qp=\spvek fg
\]
gives
\begin{multline*}
a(q,\varphi)
+
\rho\alpha^2\langle q,\varphi\rangle_{L^2(\Omega)}
+
\eta\alpha
\langle\gamma q,\gamma\varphi\rangle_{L^2(\Gamma_{\rm N})}
\\
=
\rho\alpha\langle f,\varphi\rangle_{L^2(\Omega)}
+
\langle g,\varphi\rangle_{L^2(\Omega)}
+
\eta\langle\gamma f,\gamma\varphi\rangle_{L^2(\Gamma_{\rm N})}.
\end{multline*}
Lax--Milgram yields a unique $q$. Green's formula then recovers the
boundary condition, so that $\ran(\alpha I-A)=X$. Hence $A$ is maximal
dissipative, and the Lumer--Phillips theorem
\cite[Thm.~II.3.15]{EngelNagel2000} shows that $A$ generates a
contraction semigroup. Together with the properties established above,
this proves that $\mathcal S_{\rm el}$ is a system node.

We next consider the system node with the same state-control operator
$A\&B$ and the collocated boundary-velocity output $C_{\rm v}\&D_{\rm v}$.
The latter is bounded with respect to the graph norm of $A\&B$: the first
state component of
\[
A\&B\pvek{\spvek qp}{u}
\]
is $\rho^{-1}p\in H^1_{\Gamma_{\rm D}}(\Omega;\mathbb C^2)$, and the
trace theorem \cite[Thm.~7.39]{AdamsFournier2003} gives the required
bound. Moreover, Green's formula yields
\begin{equation}
\Ree
\left\langle
(A\&B)\spvek{x}{u},x
\right\rangle_X
=
-\eta
\|\rho^{-1}\gamma p\|_{L^2(\Gamma_{\rm N})}^2
+
\Ree
\left\langle
u,
(C_{\rm v}\&D_{\rm v})\spvek{x}{u}
\right\rangle_U .
\label{eq:num-velocity-passivity}
\end{equation}
Thus this auxiliary system node is impedance passive. By the
infinite-time feedback result for impedance-passive system nodes
\cite[Thm.~2.5]{HastirPaunonen2025}, the feedback
\[
K_0\&L_0=-\kappa(C_{\rm v}\&D_{\rm v})
\]
is admissible and the corresponding system node with auxiliary input
$v_0$ is infinite-time well posed.

Its main operator is
\begin{align*}
A_0\pvek qp
&=
\pvek{\rho^{-1}p}{\operatorname{div}\sigma(q)},\\
\dom(A_0)
&=
\setdef{
\pvek qp\in X
}{
\begin{array}{l}
p\in H^1_{\Gamma_{\rm D}}(\Omega;\mathbb C^2),\\
\sigma(q)\in H(\operatorname{div},\Omega;\mathbb C^{2\times2}),\\
\gamma_n\sigma(q)
+
\rho^{-1}\eta\,\gamma p
+
\kappa\rho^{-1}GG^*\gamma p
=0
\;\text{ on }\Gamma_{\rm N}
\end{array}
}.
\end{align*}
Setting $v_0=0$ in \eqref{eq:num-velocity-passivity} gives
\begin{equation}
\Ree\langle A_0x,x\rangle_X
=
-\eta
\|\rho^{-1}\gamma p\|_{L^2(\Gamma_{\rm N})}^2
-
\kappa
\left\|
C_{\rm v}\&D_{\rm v}
\pvek{x}{-\kappa(C_{\rm v}\&D_{\rm v})\spvek{x}{0}}
\right\|_U^2.
\label{eq:num-feedback-dissipation}
\end{equation}
By definition of $C_{\rm v}\&D_{\rm v}$,
\[
(C_{\rm v}\&D_{\rm v})
\pvek{\spvek qp}{u}
=
\rho^{-1}G^*\gamma p.
\]
Hence, for $x\in\dom(A_0)$, the feedback relation
\[
u=-\kappa
(C_{\rm v}\&D_{\rm v})
\spvek{x}{u}
\]
and \eqref{eq:num-velocity-passivity} give
\[
\Ree\langle A_0x,x\rangle_X
=
-\eta
\|\rho^{-1}\gamma p\|_{L^2(\Gamma_{\rm N})}^2
-
\kappa
\|\rho^{-1}G^*\gamma p\|_U^2.
\]
In particular, the semigroup generated by $A_0$ is contractive.

We show that it is strongly stable. First, $A_0$ has compact resolvent.
Writing
$D_0=\eta I+\kappa GG^*$
on $L^2(\Gamma_{\rm N};\mathbb C^2)$, the resolvent equation
\[
(\alpha I-A_0)\spvek qp=\spvek fg
\]
reduces to
\begin{multline}
a(q,\varphi)
+
\rho\alpha^2\langle q,\varphi\rangle_{L^2(\Omega)}
+
\alpha
\langle D_0\gamma q,\gamma\varphi\rangle_{L^2(\Gamma_{\rm N})}
\\
=
\rho\alpha\langle f,\varphi\rangle_{L^2(\Omega)}
+
\langle g,\varphi\rangle_{L^2(\Omega)}
+
\langle D_0\gamma f,\gamma\varphi\rangle_{L^2(\Gamma_{\rm N})}.
\label{eq:num-feedback-resolvent}
\end{multline}
For a bounded sequence $(f_n,g_n)$ in $X$, the compact embedding
\[
H^1_{\Gamma_{\rm D}}(\Omega;\mathbb C^2)
\hookrightarrow
L^2(\Omega;\mathbb C^2)
\]
and the compactness of the trace embedding into
$L^2(\Gamma_{\rm N};\mathbb C^2)$, together with the continuity of
\[
L^2(\Omega;\mathbb C^2)
\hookrightarrow
\bigl(H^1_{\Gamma_{\rm D}}(\Omega;\mathbb C^2)\bigr)',
\]
show that the right-hand sides of
\eqref{eq:num-feedback-resolvent} are precompact in
$\bigl(H^1_{\Gamma_{\rm D}}(\Omega;\mathbb C^2)\bigr)'$; see
\cite[Thm.~6.3 and Thm.~7.39]{AdamsFournier2003}.
Coercivity then yields convergence of $q_n$ in
$H^1_{\Gamma_{\rm D}}(\Omega;\mathbb C^2)$, and
$p_n=\rho(\alpha q_n-f_n)$
converges in $L^2(\Omega;\mathbb C^2)$. Hence
$(\alpha I-A_0)^{-1}$ is compact.

Suppose now that
\[
A_0\spvek qp=i\omega\spvek qp.
\]
Equation \eqref{eq:num-feedback-dissipation} gives
$\gamma p=0$ on $\Gamma_{\rm N}$.
If $\omega=0$, then $p=0$, and testing the equilibrium equation with
$q$ gives $a(q,q)=0$. Korn's inequality therefore yields $q=0$.

Let $\omega\neq0$. The first component of the eigenvalue equation gives
$p=i\rho\omega q$,
and hence
$\gamma q=0$ on $\Gamma_{\rm N}$.
The boundary condition for $A_0$ then also gives
$\gamma_n\sigma(q)=0$
on $\Gamma_{\rm N}$,
whereas the second component yields
\[
-\operatorname{div}\sigma(q)=\rho\omega^2q.
\]
Thus both Cauchy data of $q$ vanish on a nonempty straight segment of
$\Gamma_{\rm N}$. Extension by zero across this segment, Holmgren
uniqueness for the constant-coefficient Lam\'e system, and analytic
continuation imply $q=0$; compare
\cite[Sec.~1]{DiaoLiuWang2020}. Consequently $p=0$ as well.

We have therefore proved
$\ker(i\omega I-A_0)=\{0\}$,
$\omega\in\mathbb R$.
Since $A_0$ has compact resolvent, every spectral point of $A_0$ is an
eigenvalue; see
\cite[Cor.~IV.1.19]{EngelNagel2000}. Hence
\[
\sigma(A_0)\cap i\mathbb R=\varnothing.
\]
The semigroup generated by $A_0$ is contractive and therefore bounded.
Since $\sigma(A_0^*)=\overline{\sigma(A_0)}$, the
Arendt--Batty--Lyubich--V{\~u} theorem
\cite[Thm.~V.2.21]{EngelNagel2000}
now yields strong stability. In particular,
$0\in\rho(A_0)$.

It remains to pass from the auxiliary velocity output to the position
output of $\mathcal S_{\rm el}$. Put
\[
C_{\rm pos}\spvek qp=G^*\gamma q.
\]
The first component of $A_0$ gives
\[
C_{\rm pos}A_0x
=
(C_{\rm v}\&D_{\rm v})
\pvek{x}{(K_0\&L_0)\spvek{x}{0}},
\qquad
x\in\dom(A_0).
\]
Since $D_{\rm v}=0$, this reduces to
$C_{\rm pos}A_0x=\rho^{-1}G^*\gamma p$,
and therefore
\begin{equation}
C_{\rm pos}
=
C_{{\rm v},0}A_0^{-1},
\label{eq:num-feedback-position-velocity}
\end{equation}
where $C_{{\rm v},0}$ denotes the observation operator of the feedback-transformed
auxiliary system. Since $C_{{\rm v},0}$ is infinite-time admissible,
\eqref{eq:num-feedback-position-velocity} and $0\in\rho(A_0)$ imply
infinite-time admissibility of $C_{\rm pos}$.

Finally, let $\mathbf G_{\rm v}$ and $\mathbf G_{\rm pos}$ denote the
zero-state transfer functions of the feedback-transformed system with outputs
$C_{{\rm v},0}$ and $C_{\rm pos}$, respectively. The first resolvent
component gives
\[
\mathbf G_{\rm v}(\lambda)
=
\lambda\mathbf G_{\rm pos}(\lambda),
\qquad
\lambda\in\mathbb C_+.
\]
Infinite-time well-posedness of the auxiliary system node yields
\[
\mathbf G_{\rm v}
\in
H^\infty(\mathbb C_+;\mathcal L(U)).
\]
On the other hand, $0\in\rho(A_0)$ implies that
$\mathbf G_{\rm pos}$ extends holomorphically to a neighbourhood of the
origin. Hence the preceding identity implies
\[
\mathbf G_{\rm pos}
\in
H^\infty(\mathbb C_+;\mathcal L(U)).
\]
Thus the zero-state input-output operator associated with $C_{\rm pos}$
is bounded and defined on $L^2(0,\infty;U)$. Together with the infinite-time
input-state bound of the auxiliary system node and the infinite-time
admissibility of $C_{\rm pos}$, this proves infinite-time well-posedness
of the feedback-transformed system with output $2C_{\rm pos}=2G^*\gamma q$.
\end{proof}

\subsection{Shifted-resolvent discretization}
\label{subsec:num-resolvent}

Since the performance observation $2G^*\gamma q$ is bounded on the energy
space, no separate continuous adjoint system-node realization is needed for
the numerical implementation. Instead, after Galerkin discretization the
adjoint resolvent actions required by ADI are obtained directly from the
transpose of the discrete energy representation.

We use conforming vector-valued piecewise affine finite elements. After
elimination of the clamped degrees of freedom, let
$K_{\varepsilon,h}$, $K_{\operatorname{div},h}$, $M_h$, and $D_h$
denote the matrices associated with the strain, divergence, volume
$L^2$, and boundary $L^2$ forms, respectively, and set
\[
K_h
=
2\mu K_{\varepsilon,h}
+
\lambda_{\rm L}K_{\operatorname{div},h}.
\]
Let $B_h$ denote the matrix of the boundary traction functional associated
with $G$. Since the displacement observation uses the same boundary
profiles as the input, its discrete physical output is
$y_{{\rm pos},h}=B_h^*q_h$.
Thus the performance output is represented by
\[
C_h
=
2
\begin{bmatrix}
B_h^*&0
\end{bmatrix}.
\]
For the coefficient state
$x_h=\spvek{q_h}{p_h}$,
define
\[
H_h
=
\begin{bmatrix}
K_h&0\\
0&\rho^{-1}M_h
\end{bmatrix},
\qquad
\mathcal A_h
=
\begin{bmatrix}
0&\rho^{-1}K_h\\
-\rho^{-1}K_h&-\rho^{-2}\eta D_h
\end{bmatrix},
\]
and
\[
\mathcal B_h
=
\rho^{-1}\bvek{0}{B_h}.
\]
The Galerkin equations can then be written in the generalized form
\[
H_h\dot x_h
=
\mathcal A_hx_h+\mathcal B_hu,
\qquad
y_h=C_hx_h.
\]
Here $H_h$ is the Gram matrix of the discrete mechanical-energy inner
product. This first-order representation is used only for the algebraic
description of the resolvents and feedback relations; the shifted PDE
solves themselves are carried out in second-order form.

Indeed, for $\lambda\in\mathbb C_+$ the open input resolvent
\[
(\lambda H_h-\mathcal A_h)
\spvek{q_h}{p_h}
=
\mathcal B_hu
\]
reduces to
$p_h=\rho\lambda q_h$
and
\begin{equation}
S_h(\lambda)q_h
=
B_hu,
\qquad
S_h(\lambda)
:=
K_h+\eta\lambda D_h+\rho\lambda^2M_h.
\label{eq:num-shifted-primal}
\end{equation}
Thus every primal input-resolvent application requires only a sparse solve
with the second-order shifted elasticity matrix $S_h(\lambda)$.

The same matrix is used for the adjoint resolvent actions. Consider a
general coefficient-space right-hand side
\[
(\lambda H_h-\mathcal A_h^*)
\spvek{r_h}{s_h}
=
\spvek{R_{q,h}}{R_{p,h}}.
\]
Putting
$\zeta_h=K_h^{-1}R_{q,h}$,
the first block row gives
$s_h
=
\rho\zeta_h-\rho\lambda r_h$,
and substitution into the second block row yields
\begin{equation}
S_h(\lambda)r_h
=
\bigl(\rho\lambda M_h+\eta D_h\bigr)\zeta_h
-
\rho R_{p,h}.
\label{eq:num-shifted-adjoint-general}
\end{equation}
Consequently, the adjoint actions required by ADI use exactly the same
second-order shifted matrix
\[
S_h(\lambda)
=
K_h+\eta\lambda D_h+\rho\lambda^2M_h
\]
as the primal input resolvent. In particular, for the displacement
performance output,
\[
C_h^*z
=
\spvek{2B_hz}{0},
\]
so its adjoint action is handled by
\eqref{eq:num-shifted-adjoint-general}; no artificial boundary-velocity
observation is introduced.

\subsection{Adaptive Newton--Kleinman realization}
\label{subsec:num-adaptive-nk}

\paragraph{Low-rank feedback evaluation and update.}

The current feedback is stored as an implicit relation
\[
v_k
=
\Gamma_{k,h}x_h+\Lambda_{k,h}u,
\qquad
\Lambda_{k,h}\in\mathbb C^{2\times2},
\]
rather than as a dense state-feedback matrix. Setting $v_k=0$ gives
\[
u
=
F_{k,h}x_h,
\qquad
F_{k,h}
=
-\Lambda_{k,h}^{-1}\Gamma_{k,h}.
\]
For the initial collocated boundary-velocity feedback,
\[
\Gamma_{0,h}
=
\kappa\mathcal B_h^*,
\qquad
\Lambda_{0,h}=I,
\qquad
\kappa=1.
\]

The feedback-constrained primal input resolvent is evaluated without
assembling a closed-loop first-order matrix. Let
\[
G_h(\lambda)
=
(\lambda H_h-\mathcal A_h)^{-1}\mathcal B_h,
\]
where the two columns of $G_h(\lambda)$ are obtained from
\eqref{eq:num-shifted-primal}. For prescribed auxiliary input $v_k$, the
feedback relation gives
\[
\bigl(
\Lambda_{k,h}
+
\Gamma_{k,h}G_h(\lambda)
\bigr)u
=
v_k.
\]
Hence only the $2\times2$ matrix
\[
Q_{k,h}(\lambda)
=
\Lambda_{k,h}
+
\Gamma_{k,h}G_h(\lambda)
\]
has to be inverted, after which
\[
u=Q_{k,h}(\lambda)^{-1}v_k,
\qquad
x_h=G_h(\lambda)u.
\]

The corresponding closed-loop adjoint resolvent is treated in the same
way. For a general right-hand side $R$, we solve
\begin{equation}
\begin{aligned}
(\lambda H_h-\mathcal A_h^*)X+\Gamma_{k,h}^*Y&=R,\\
\Lambda_{k,h}^*Y&=\mathcal B_h^*X.
\end{aligned}
\label{eq:num-feedback-adjoint-system}
\end{equation}
Let
\[
X_R
=
(\lambda H_h-\mathcal A_h^*)^{-1}R,
\qquad
X_\Gamma
=
(\lambda H_h-\mathcal A_h^*)^{-1}\Gamma_{k,h}^*.
\]
Both are computed through the second-order reduction
\eqref{eq:num-shifted-adjoint-general}. Eliminating the two-dimensional
port variable gives
\[
Y
=
\bigl(
\Lambda_{k,h}^*
+
\mathcal B_h^*X_\Gamma
\bigr)^{-1}
\mathcal B_h^*X_R,
\qquad
X=X_R-X_\Gamma Y.
\]
Thus every closed-loop adjoint resolvent requires sparse second-order
elasticity solves together with an inversion of a $2\times2$ matrix.

For every current feedback we evaluate its cost operator directly. The
zero-input performance observation is
\[
Z_{k,h}
=
\begin{bmatrix}
C_h\\
F_{k,h}
\end{bmatrix},
\]
because the quadratic cost contains both the displacement output and the
physical control input. Accordingly,
\[
P_{k+1}
=
\mathcal C_{k,\infty}^*
\mathcal C_{k,\infty}
\]
is approximated by the low-rank ADI iteration from
Subsection~\ref{subsec:nk-computational}. No decrement Gramian is used in
the numerical realization. If $Z_{P,k+1}$ denotes the compressed ADI
factor, the discrete value action is represented as
\[
P_{k+1,h}x_h
\approx
Z_{P,k+1}Z_{P,k+1}^*H_hx_h.
\]
This action is used directly in the subsequent Newton--Lur'e update with
the fixed normalization parameter
\[
\alpha=1.
\]

Since all finite-element matrices are real, nonreal ADI shifts are processed
in complex-conjugate pairs. For a conjugate pair, the two scalar ADI
updates are combined algebraically into one real block update, so that only
one complex shifted elasticity solve is required. The resulting low-rank
factor and residual factor are real up to roundoff.

\paragraph{ADI shifts and stopping.}

The shift candidates are recomputed at every production level for the
current feedback. To avoid spectral computations on the growing adaptive
meshes, the current implicit feedback relation is restricted to a fixed
coarse elasticity model with $150$ triangles. The complete spectrum of the
corresponding coarse closed-loop problem is computed and used as the
candidate set for a continued Penzl-type shift selection
\cite{Penzl2000,BennerMenaSaak2008}. No fine-grid spectral data and no
stored shift sequence from a preceding Newton step are used.

ADI convergence is monitored by a resolvent-weighted residual at the fixed
point
\[
\beta=1.
\]
Let $W_j$ denote the ADI residual factor after a completed real or
complex-conjugate macro step, and let $U_j$ be the solution of the
feedback-constrained adjoint system
\eqref{eq:num-feedback-adjoint-system} with
$\lambda=\beta$ and right-hand side $R=W_j$. We use
\[
\rho_{\beta,j}
=
\frac{
\|U_j^*H_hU_j\|_{\rm F}
}{
\|U_0^*H_hU_0\|_{\rm F}
}.
\]
At adaptive level $\ell=0,1,\ldots$, the ADI tolerance is
\[
\tau_{{\rm ADI},\ell}
=
\max\left\{
2\cdot10^{-5},
10^{-3}2^{-\ell}
\right\}.
\]
The low-rank factor is subsequently compressed in the mechanical-energy
metric with tolerance
\[
\tau_{{\rm comp},\ell}
=
\min\left\{
5\cdot10^{-3},
10\tau_{{\rm ADI},\ell}
\right\}.
\]
The ADI iteration is terminated solely by
\[
\rho_{\beta,j}
\leq
\tau_{{\rm ADI},\ell}.
\]
The resulting direct feedback-cost factor is then used for exactly one
Newton--Klein\-man update on the current production mesh. The spatial
refinement and the coupled outer stopping criteria are described next.

\paragraph{Spatial adaptivity and coupled stopping.}

The initial mesh generated with mesh width $h_0=0.2$ consists of $150$
triangles. Before the adaptive Newton--Kleinman iteration is started, we
perform three uniform newest-vertex-bisection refinement cycles. Thus the
first production mesh, denoted by $T_3$, contains $1200$ triangles and has
first-order state dimension $2480$. All subsequent refinements are local.

\paragraph{Residual-based marking.}
The spatial refinement indicator is obtained directly from the shifted
elasticity solves already carried out during the current direct
feedback-cost ADI computation. Hence no additional family of PDE solves is
introduced solely for mesh marking.

To describe the indicator, consider one scalar member of an ADI macro step,
with shift $p\in\mathbb C_-$ and corresponding right-half-plane parameter
$\lambda=-p$.
For a general adjoint state-space right-hand side
$(R_{q,h},R_{p,h})$, put
\[
z_h=K_h^{-1}R_{q,h},
\qquad
M_hw_{p,h}=\rho R_{p,h}.
\]
The displacement component $r_h$ of the shifted adjoint solution then
satisfies
\[
\bigl(
K_h+\eta\lambda D_h+\rho\lambda^2M_h
\bigr)r_h
=
\bigl(
\rho\lambda M_h+\eta D_h
\bigr)z_h
-
\rho R_{p,h}.
\]
For an element $T$, an interior edge $E$, and an edge
$E\subset\Gamma_{\rm N}$, respectively, we use the residuals
\[
\begin{aligned}
R_T
&=
\rho\lambda z_h-w_{p,h}
+\operatorname{div}\sigma(r_h)
-\rho\lambda^2r_h,
\\
J_E
&=
[\sigma(r_h)n_E],
\\
R_E^{\rm N}
&=
\eta z_h
-\sigma(r_h)n
-\eta\lambda r_h.
\end{aligned}
\]
The local squared contribution is
\begin{equation}
\eta_{T,p}^2
=
h_T^2\|R_T\|_{L^2(T)}^2
+
\frac12
\sum_{\substack{E\subset\partial T\\E\ {\rm interior}}}
h_E\|J_E\|_{L^2(E)}^2
+
\sum_{E\subset\partial T\cap\Gamma_{\rm N}}
h_E\|R_E^{\rm N}\|_{L^2(E)}^2.
\label{eq:num-spatial-local}
\end{equation}
The contributions in \eqref{eq:num-spatial-local} are summed over all
right-hand-side columns of the ADI block. For a complex-conjugate ADI
macro step, both scalar members of the pair are included in this sum.

The element contributions accumulated over the complete feedback-cost ADI
solve are
\[
\eta_T^2
=
\sum_m
\bigl(-2\Ree p_m\bigr)
\eta_{T,p_m}^2,
\]
where a complex-conjugate pair is regarded as one macro step and
$\eta_{T,p_m}^2$ already contains the contributions of both members.
The quantity
\[
\eta_{\rm res}
=
\frac{
\left(\sum_T\eta_T^2\right)^{1/2}
}{
\|H_h^{1/2}Z_{P,k+1}^{\rm raw}\|_{\rm HS}
}
\]
is recorded as a normalized spatial diagnostic. It is not used as an
outer stopping criterion.

After every completed Newton--Kleinman update, refinement is mandatory.
We use D\"orfler marking \cite{Doerfler1996}: a minimal set of elements
$\mathcal A_k$ is selected such that
\[
\sum_{T\in\mathcal A_k}\eta_T^2
\geq
\theta
\sum_T\eta_T^2,
\qquad
\theta=0.5.
\]
The marked elements are then refined by newest-vertex bisection. Thus every
production mesh carries exactly one Newton--Kleinman update, followed by
one spatial refinement before any outer stopping decision is made.

\paragraph{Transfer of the current feedback.}
Only the current feedback relation is transferred to the refined mesh.
No low-rank value factor, previous Gramian, ADI shift sequence, or
closed-loop generator is transferred.

Let
$J:X_{h_{\rm c}}\rightarrow X_{h_{\rm f}}$
denote the nested $P_1$ injection from the current state space into the
refined state space and let $H_{\rm c}$ and $H_{\rm f}$ be the
corresponding mechanical-energy Gram matrices. We use the
$H_{\rm f}$-orthogonal projection
\[
R
=
\bigl(J^*H_{\rm f}J\bigr)^{-1}J^*H_{\rm f},
\qquad
RJ=I.
\]
If the current implicit feedback relation is
$\Gamma_{\rm c}x_{\rm c}
+
\Lambda_{\rm c}u
=
0$,
its transfer to the refined mesh is defined by
\begin{equation}
\Gamma_{\rm f}
=
\Gamma_{\rm c}R,
\qquad
\Lambda_{\rm f}
=
\Lambda_{\rm c}.
\label{eq:num-feedback-transfer}
\end{equation}
Consequently,
\[
\Gamma_{\rm f}Jx_{\rm c}
+
\Lambda_{\rm f}u
=
\Gamma_{\rm c}x_{\rm c}
+
\Lambda_{\rm c}u,
\]
so the transferred feedback agrees with the coarse feedback on the
injected coarse state space. On the next production mesh the feedback-cost
operator is computed anew by direct LR--ADI.

\paragraph{Newton convergence.}
The mesh changes after every Newton step. A direct norm of the complete
feedback difference would therefore mix the actual Newton correction with
new fine-grid directions that were not present on the preceding mesh.
We consequently measure Newton convergence on the inherited state
subspace.

For an implicit feedback relation
$\Gamma x+\Lambda u=0$, write
$F=-\Lambda^{-1}\Gamma$
and, for a row operator $L$ on a discrete state space with energy matrix
$H$, define
\[
\|L\|_{H^{-1}}
=
\sqrt{
\lambda_{\max}\bigl(LH^{-1}L^*\bigr)
}.
\]
Let $F_k^{\rm in}$ be the feedback entering the Newton step on the current
production mesh and let $F_k^{\rm new}$ be the feedback after the update.
If
$J_{k-1,k}:X_{h_{k-1}}\to X_{h_k}$
denotes the nested injection from the preceding production mesh, we define
\[
\eta_{{\rm N},k}
=
\frac{
\|
(F_k^{\rm new}-F_k^{\rm in})
J_{k-1,k}
\|_{H_{k-1}^{-1}}
}{
\|
F_k^{\rm in}J_{k-1,k}
\|_{H_{k-1}^{-1}}
}.
\]
Thus $\eta_{{\rm N},k}$ measures the Newton correction only in directions
already represented on the preceding production mesh. The complementary
correction in newly created fine-grid directions is monitored separately
as a diagnostic but does not enter the stopping rule. No Newton indicator
is assigned to the first production level.

\paragraph{Post-refinement two-grid validation.}
After the mandatory refinement, the newly computed feedback is transferred
to the child mesh by \eqref{eq:num-feedback-transfer}. We then compare the
parent and child discretizations with exactly this same feedback. This
validation performs neither an additional Newton step nor an ADI
iteration.

Let
\[
C_{{\rm perf},h}
=
\begin{bmatrix}
C_h\\
F_h
\end{bmatrix}
\]
denote the zero-input performance observation for the fixed feedback, so
that its four components consist of the two displacement outputs and the
two physical control inputs. Let
\[
\omega_k
=
\max_j|\operatorname{Im}p_j|
\]
be the largest imaginary part among the ADI shifts actually used in the
feedback evaluation on the parent mesh. The two-grid check uses the six
sampling points
\[
\Lambda_{{\rm grid},k}
=
\left\{
0.1,\,
1,\,
0.1+\tfrac14 i\omega_k,\,
0.1+\tfrac12 i\omega_k,\,
0.1+\tfrac34 i\omega_k,\,
0.1+i\omega_k
\right\}.
\]
For $\lambda\in\Lambda_{{\rm grid},k}$, let $X_h(\lambda)$ denote the
solution of the closed-loop adjoint resolvent problem
\begin{equation}
\bigl(
\lambda H_h-\mathcal A_{F,h}^*
\bigr)
X_h(\lambda)
=
C_{{\rm perf},h}^*,
\label{eq:num-grid-adjoint}
\end{equation}
where \eqref{eq:num-grid-adjoint} is evaluated through the implicit
feedback relation and the second-order shifted solves described above;
the closed-loop matrix $\mathcal A_{F,h}$ is not assembled. We form the
small Gram matrix
\[
G_h(\lambda)
=
X_h(\lambda)^*H_hX_h(\lambda).
\]
If $G_{\rm c}(\lambda)$ and $G_{\rm f}(\lambda)$ denote these matrices on
the parent and child meshes, respectively, the spatial validation quantity
is
\[
\eta_{{\rm grid},k}
=
\frac{
\left(
\displaystyle
\sum_{\lambda\in\Lambda_{{\rm grid},k}}
\|
G_{\rm f}(\lambda)-G_{\rm c}(\lambda)
\|_{\rm F}^2
\right)^{1/2}
}{
\left(
\displaystyle
\sum_{\lambda\in\Lambda_{{\rm grid},k}}
\|
G_{\rm f}(\lambda)
\|_{\rm F}^2
\right)^{1/2}
}.
\]
This is a sampled same-feedback two-grid diagnostic rather than a rigorous
uniform discretization-error bound. Pointwise relative differences at the
individual sampling points are recorded only diagnostically.

\paragraph{Coupled stopping rule.}
We use
\begin{equation}
\eta_{{\rm N},k}
\leq
5\cdot10^{-3},
\qquad
\eta_{{\rm grid},k}
\leq
5\cdot10^{-3}.
\label{eq:num-outer-stop}
\end{equation}
The computation is terminated only if both conditions hold after the
mandatory refinement of the same production level. In that case the parent
mesh together with the feedback computed there is accepted as the final
production result, while the refined child is used only for the final
two-grid validation. If either condition fails, the child becomes the next
production mesh, carrying only the transferred current feedback, and the
next direct feedback-cost ADI solve and Newton--Kleinman update are
performed there.

Thus the outer iteration has the fixed structure
\[
\boxed{
\begin{gathered}
\text{ADI feedback evaluation}
\rightarrow
\text{Newton--Kleinman update}
\\
\rightarrow
\text{D\"orfler/NVB refinement}
\rightarrow
\text{feedback transfer}
\rightarrow
\text{two-grid check}
\end{gathered}
}
\]
\subsection{Numerical results}
\label{subsec:num-results}

The adaptive computation is initialized by the collocated boundary-velocity
feedback $K_0$ with gain $\kappa=1$. Starting from the first production
mesh $T_3$, exactly one Newton--Kleinman update is performed on every
production mesh. Table~\ref{tab:num-nk} summarizes the resulting sequence.
The reported ADI rank is the rank after compression, and
$\rho_\beta$ is the final resolvent-weighted ADI residual ratio.

\begin{table}[t]
\centering
\caption{Adaptive Newton--Kleinman iteration for the elasticity problem.
The Newton indicator is not defined on the first production level.}
\label{tab:num-nk}
\scriptsize
\begin{tabular}{@{}c|r|r|c|r|c|c|c@{}}
\hline
mesh
& triangles
& $\dim X_h$
& update
& rank
& $\rho_\beta$
& $\eta_{\rm N}$
& $\eta_{\rm grid}$
\\
\hline
$T_3$
& 1200
& 2480
& $K_0\to K_1$
& 510
& $9.9982\cdot10^{-4}$
& --
& $9.8634\cdot10^{-2}$
\\
$T_4$
& 1561
& 3272
& $K_1\to K_2$
& 287
& $4.7901\cdot10^{-4}$
& $5.0509\cdot10^{-1}$
& $2.6798\cdot10^{-2}$
\\
$T_5$
& 2140
& 4480
& $K_2\to K_3$
& 187
& $2.3522\cdot10^{-4}$
& $1.8511\cdot10^{-1}$
& $1.2974\cdot10^{-2}$
\\
$T_6$
& 2881
& 5964
& $K_3\to K_4$
& 226
& $1.1830\cdot10^{-4}$
& $2.4681\cdot10^{-2}$
& $7.0376\cdot10^{-3}$
\\
$T_7$
& 3806
& 7804
& $K_4\to K_5$
& 306
& $6.1200\cdot10^{-5}$
& $4.3541\cdot10^{-3}$
& $6.6817\cdot10^{-3}$
\\
$T_8$
& 4827
& 9844
& $K_5\to K_6$
& 410
& $3.0227\cdot10^{-5}$
& $4.6425\cdot10^{-3}$
& $3.2048\cdot10^{-3}$
\\
\hline
\end{tabular}
\end{table}

The two stopping mechanisms are visibly distinct. On $T_7$, the Newton
criterion is already satisfied,
\[
\eta_{\rm N}
=
4.3541\cdot10^{-3}
<
5\cdot10^{-3},
\]
whereas the post-refinement spatial check still gives
\[
\eta_{\rm grid}
=
6.6817\cdot10^{-3}
>
5\cdot10^{-3}.
\]
The computation therefore does not terminate at this level. The feedback
$K_5$ is transferred to $T_8$, where one further direct feedback-cost
evaluation and one Newton--Kleinman update produce $K_6$.

As required by the outer algorithm, $T_8$ is then refined once more before
the stopping decision is made. The resulting child $T_9$ contains $6538$
triangles and has state dimension $13480$. No Newton--Kleinman step and no
ADI iteration are performed on this child. Instead, the feedback $K_6$ is
transferred from $T_8$ to $T_9$ and used only for the same-feedback
two-grid validation. This gives
\[
\eta_{\rm N}
=
4.642469575470\cdot10^{-3},
\qquad
\eta_{\rm grid}
=
3.204769605900\cdot10^{-3},
\]
so both conditions in \eqref{eq:num-outer-stop} are satisfied. Hence the
accepted production result is
$T_8/K_6$,
while $T_9$ is a validation-only mesh.

The complete adaptive Newton--Kleinman computation required
$727.6\,\mathrm{s}$ of wall-clock time. The final accepted discretization
contains $4827$ triangles and has first-order state dimension $9844$.
Throughout the computation the feedback remains in implicit two-port form;
no dense closed-loop generator or dense feedback matrix is assembled.

\paragraph{Visualization.}
A visualization comparing the open-loop and closed-loop deformation dynamics
for the final feedback is provided in the supplementary material. Starting
from the same initial deformation, it illustrates the decay of the elastic
oscillations under feedback control.

\section{Conclusion}

We have developed a Newton--Kleinman iteration for infinite-dimensional
linear-quadratic control problems in the system-node framework. The
formulation is based on a Lur'e identity for the value operator and provides
a well-defined Newton step without relying on the operator products appearing
in the classical algebraic Riccati equation. Under finite-cost detectability
and finite-dimensional input spaces, the resulting iteration converges
strongly to the optimal value operator.

The numerical realization is based on low-rank ADI iterations driven by shifted system-node resolvent solves. In the cylinder-wake example, this realization is combined with adaptive Taylor--Hood discretization and exactly one Newton--Kleinman update on each spatial level. The elasticity example demonstrates the same operator-level construction for a second-order mechanical system with boundary traction control and boundary displacement observation. In both examples, the computation follows the continuous system-node formulation through shifted primal and adjoint problems and finite-dimensional operations on the input space, without first forming a discrete algebraic Riccati equation or explicitly assembling the closed-loop generator.

The adaptive strategy used in this work is deliberately simple: on each
spatial level, one Newton--Kleinman update is performed, followed by a spatial
refinement step, while the low-rank Gramian computations are resolved on the
current fixed mesh. This provides a first numerical framework for combining
Riccati iterations with adaptive discretization. It should not be viewed as a
unique adaptive paradigm. Rather, it opens a broader numerical perspective in
which different couplings between Newton convergence, low-rank approximation
accuracy, and spatial discretization errors can be investigated. More flexible adaptive strategies could allow the spatial mesh to evolve already during the Newton and low-rank ADI iterations. In particular, spatial adaptation could already be performed within an ADI iteration, with local refinement where the current low-rank factors reveal insufficient spatial resolution and coarsening where the evolving approximation shows little spatial activity. Developing reliable criteria for such genuinely coupled Newton--ADI--mesh adaptation is an interesting direction for further work.

\section*{Statements and Declarations}

\bmhead{Funding}

No funding was received for conducting this study.

\bmhead{Competing interests}

The author has no relevant financial or non-financial interests to disclose.

\bmhead{Data availability}

All numerical data supporting the findings of this study are generated by
the reproduction code provided as supplementary material.

\bmhead{Code availability}

The source code used to generate the numerical results reported in this
article, together with the reproduction scripts and documentation, is
provided as supplementary material.

\bmhead{Author contributions}

Timo Reis conceived the study, developed the mathematical analysis and
numerical methodology, implemented and validated the computations, and wrote
the manuscript.

\bmhead{Use of generative AI}

OpenAI ChatGPT was used to assist with the development and debugging of parts
of the computational code and, on a limited basis, for language polishing.
All AI-assisted code was reviewed, tested, and validated by the author. The
author takes full responsibility for the content of the manuscript.

\bibliography{references}

\end{document}